\documentclass[brochure,12pt,french]{smfbourbaki}
        
\usepackage[T1]{fontenc}
\usepackage{lmodern,amssymb,bm}

\usepackage[matrix,arrow,cmtip]{xy}
\usepackage{babel}

\usepackage[utf8]{inputenc}
\usepackage[colorlinks=true, linkcolor=blue, citecolor=red, urlcolor=blue]{hyperref}

\usepackage[
backend=biber,
style=authoryear, 
citestyle=authoryear-comp,
maxnames=7,
sortcites=false 
]{biblatex}
\usepackage{csquotes}

\DefineBibliographyExtras{french}{\restorecommand\mkbibnamefamily}

\DeclareDelimFormat{nameyeardelim}{\addcomma\space}

\DeclareNameAlias{sortname}{given-family}

\renewcommand{\bibnamedash}{\leavevmode\raise3pt\hbox to3em{\hrulefill}\space}

\AtEveryBibitem{%
 \clearfield{issn} 
  \clearfield{isbn} 
  \clearfield{doi} 
  \clearlist{language} 
  \ifentrytype{online}{}{
  \ifentrytype{unpublished}{}{
   \clearfield{url}
 }
 }
}

\renewbibmacro{in:}{%
   \ifentrytype{article}{}{\printtext{\bibstring{in}\intitlepunct}}}

\DeclareFieldFormat[article,periodical,inreference]{number}{\mkbibparens{#1}}
\DeclareFieldFormat[article,periodical,inreference]{volume}{\mkbibbold{#1}}
\renewbibmacro*{volume+number+eid}{%
   \setunit*{\addcomma\space}
   \printfield{volume}%
  \setunit*{\adddot}
   \setunit*{\addthinspace}
   \printfield{number}%
    \setunit{\addcomma\space}%
    \printfield{eid}}

\DeclareFieldFormat[article,inbook,incollection]{title}{\enquote{#1}\addcomma} 

\date{Janvier 2023}
\bbkannee{75\textsuperscript{e} année, 2022--2023}  
\bbknumero{1201}                                      

\title{La conjecture du facteur direct}
\subtitle{d'après André et Bhatt}

\author{Gabriel Dospinescu}
\address{UMPA, ENS Lyon, CNRS}
\email{gabriel.dospinescu@ens-lyon.fr}

\newcommand{\coloneqq}{\mathrel{\mathop:}=}

\begin{document}

\maketitle
         
La conjecture du facteur direct (théorème~\ref{DSC} ci-dessous) est un énoncé d'algèbre
commutative presque aussi élémentaire que le Théorème de Fermat et qui est resté ouvert
pendant près de 50 ans: énoncée en $1969$ (cf. \textcite{H2} pour la version publi\'ee), elle a \'et\'e 
démontrée par André en 2016 (cf. \textcite{AndreDSC} pour la version publiée).
Cet énoncé fait partie d'un faisceau de conjectures, les {\og conjectures homologiques\fg}
dont la liste et les relations donnent un peu le tournis\footnote{ 
Voir le th\'eor\`eme~\ref{hconj} pour un condensé loin d'être exhaustif, ainsi que \textcite{H7,Rhom,H4}.}. 
En particulier, \textcite{H3,H4} avait prouvé que l'existence
de $A$-algèbres (ou m\^eme seulement de $A$-modules) de Cohen--Macaulay (voir \S~\ref{coma}), pour tout anneau local noethérien~$A$, 
impliquait la plupart de ces conjectures, par exemple celle du facteur direct. 
\textcite{HH1} avaient aussi montr\'e\footnote{L'existence de $A$-modules de Cohen--Macaulay pour $A$ d'\'egale caract\'eristique avait \'et\'e \'etablie bien avant, cf.\ \cite{H3}.} l'existence de telles $A$-alg\`ebres 
dans le cas d'égale caractéristique, i.e.\ quand $A$~contient un corps.


Le but de cet exposé est d'expliquer les techniques introduites par André 
dans ses trois articles monumentaux \parencite{AndreDSC,AndreIHES,AndreJAMS}, 
en particulier comment les espaces perfectoïdes introduits par \textcite{Scholzethese} 
permettent de construire 
des algèbres de Cohen--Macaulay pour les anneaux locaux noethériens d'inégale caractéristique
et donc de prouver la conjecture du facteur direct.\footnote{Que les espaces perfectoïdes soient un outil indispensable en théorie de Hodge $p$-adique ne faisait guère de doute après leurs premières applications spectaculaires \parencite{Scholzethese,ScholzeRAV,Scholzetorsion}. Qu'ils puissent aussi résoudre les conjectures homologiques était moins clair: la plupart de ces problèmes concernent des anneaux noethériens, propriété quasiment jamais satisfaite par les anneaux perfectoïdes.}
Les travaux récents et spectaculaires de  \textcite{Bhattabs} poussent encore plus 
loin les techniques perfectoïdes (via la théorie prismatique de  \textcite{BS} et la correspondance de Riemann--Hilbert $p$-adique de  \textcite{BL}) et établissent (théorème~\ref{Bhattamazing} ci-dessous) un 
analogue en inégale caractéristique d'un célèbre théorème de  \textcite{HH1} (en caractéristique positive), qui 
implique tous les résultats exposés ici, et bien plus. 

La preuve du résultat 
de Bhatt est un véritable tour de force, et tous les détails ne sont pas encore (à notre connaissance) disponibles, nous avons donc décidé de nous concentrer sur les articles d'André dans cet exposé, en fournissant des preuves complètes (autant que faire se peut) des résultats principaux des trois articles mentionnés ci-dessus, tout en utilisant des idées de  \textcite{BhattInv} pour simplifier certains arguments.\footnote{On trouvera deux preuves de l'existence de $A$-alg\`ebres de Cohen--Macaulay
dans ce rapport. Elles partagent un ingr\'edient fondamental, le lemme de platitude d'Andr\'e; l'une utilise 
le th\'eor\`eme de presque puret\'e de \textcite{Falp, Falal}, raffin\'e et \'etendu par \textcite{Scholzethese} et par \textcite{KL}, l'autre n'en fait pas usage.}
Le chemin que nous avons choisi pour arriver à la preuve de la conjecture du facteur direct 
n'est pas le plus court (la géodésique se trouve dans l'article 
 \cite{BhattInv}); il nous fera visiter des résultats plus puissants
 que la conjecture elle-même. Pour des applications de ces idées et techniques \`a la théorie (naissante) des singularités en caractéristique mixte nous renvoyons aux travaux de  \textcite{MaSchInv,MaSchDuke,Many} (entre autres) et pour de vastes généralisations des résultats présentés ici le lecteur pourra consulter le livre (de longueur presque infinie\dots{}) de  \textcite{GR}.

\smallskip
\noindent
\textbf{Convention:} Tous les anneaux sont supposés commutatifs et unitaires, 
et les morphismes d'anneaux sont unitaires. 
Si $I$ est un idéal d'un anneau $A$, on dit que $A$ est \emph{$I$-complet} 
si $A$ est séparé complet pour la topologie 
$I$-adique. Pour $a\in A$ on dit que 
$A$ est \emph{$a$-complet} si $A$ est $aA$-complet, et on note $A/a\coloneqq A/aA$.
On note $A[I]$ l'id\'eal de $A$ des éléments annulés par tous les éléments de $I$.
Si $a_1,\ldots, a_d$ sont des éléments de $A$, 
on note $(a_1,\ldots, a_d)=\sum_{i=1}^d a_i A$ l'idéal de 
$A$ qu'ils engendrent.

\smallskip
\noindent
\textbf{Remerciements:} Mes plus vifs remerciements vont \`a Yves Andr\'e, Bharghav Bhatt, Nicolas Bourbaki, K\k{e}stutis \v{C}esnavi\v{c}ius, Pierre Colmez, Luc Illusie, Wies{\l}awa Nizio{\l} et Olivier Ta\"{i}bi. Leurs commentaires et leurs suggestions ont permis au b\'eotien du sujet d'\'eviter bon nombre de pi\`eges et ont grandement am\'elior\'e le contenu et la lisibilit\'e de ce rapport. Je remercie tout particuli\`erement Yves Andr\'e pour sa disponibilit\'e, son enthousiasme et ses multiples remarques.

\section{Les multiples visages de la conjecture}

Le but de cette section est d'énoncer les principaux résultats d'algèbre commutative \og classique\fg{}\footnote{Les perfectoïdes n'apparaissent donc pas dans cette section. Le lecteur trouvera dans \textcite{AndreICM} un survol des preuves fait par le maître, et qui semble impossible à dépasser en terme de présentation.} démontrés par André et Bhatt, et d'expliquer les liens qu'ils entretiennent.

\subsection{Les anneaux réguliers, sources de problèmes}

Un anneau local noethérien $A$ de dimension $d$
  est dit \emph{régulier} si son unique idéal maximal $\mathfrak{m}$ est engendré par 
  $d$ éléments (c'est le nombre minimal possible de générateurs). Des exemples typiques de tels anneaux
  sont les anneaux locaux des variétés algébriques lisses sur un corps (ou sur un anneau de valuation discrète), ainsi que leurs complétés, par exemple 
  $K[[X_1,\ldots, X_n]]$ ($K$ étant un corps ou un anneau de valuation discrète), 
 mais aussi $\mathbf{Z}_p[[X, Y, Z]]/(p-X^5-Y^7-Z^9)$, etc.

 En dépit de leur définition très simple, les anneaux locaux réguliers sont une source inépuisable de problèmes délicats, et il 
  n'est pas facile d'établir même des propriétés très basiques 
  comme la stabilité de la régularité par localisation en un idéal premier (cela se déduit de l'interprétation homologique de la régularité fournie par le théorème de Serre), ce qui permet de globaliser\footnote{Un anneau noethérien est régulier si ses localisés en des idéaux premiers quelconques sont des anneaux locaux réguliers.} la notion de régularité. Il n'est pas difficile de montrer que  
tout anneau local régulier
   est intègre, mais il faut se fatiguer un peu pour montrer qu'il est
   normal\footnote{Autrement dit intégralement clos dans son corps des fractions.}, et bien plus pour montrer qu'il est même factoriel (théorème d'Auslander--Buchsbaum).
   
   Si $(A,\mathfrak{m})$ est un anneau local régulier, son 
   complété
   $\hat{A}=\varprojlim_{n} A/\mathfrak{m}^n$ est un anneau local régulier complet (pour la topologie $\mathfrak{m}$-adique), et le théorème 
   de structure de Cohen montre que $\hat{A}$ a l'une des formes suivantes, à isomorphisme près:
   
   $\bullet$ ou bien $V[[X_1,\ldots, X_n]]$ avec 
   $V$ un corps ou un anneau de valuation discrète complet et non ramifié (i.e.\ l'idéal maximal de 
   $V$ est engendré par un nombre premier $p$). On dira alors que $\hat{A}$ est \emph{non ramifié};
   
   $\bullet$ ou bien 
   $V[[X_1,\ldots, X_n]]/(p-f)$ pour un anneau de valuation discrète complet et non ramifié $V$, de caractéristique résiduelle $p$, et un élément 
   $f$ dans $(p,X_1,\ldots, X_n)^2$ mais pas dans $pV[[X_1,\ldots, X_n]]$ (on dira que $\hat{A}$ est \emph{ramifié} dans ce cas). 
  
    
    
      

\subsection{Énoncé de la conjecture du facteur direct}
La conjecture du facteur direct de  \textcite{H2}, à laquelle cet exposé est consacré, est l'énoncé suivant,
à l'air parfaitement innocent:

      \begin{theo}\label{DSC} 
   Toute extension finie d'un anneau régulier est scindée.
   \end{theo}
   
   Précisons l'énoncé: une \emph{extension d'anneaux}
   est un morphisme injectif d'anneaux $f\colon  A\to B$, elle est dite 
   \emph{finie} 
   si $f$ fait de $B$ un $A$-module de type fini, et \emph{scindée} si $A$ est un facteur direct du $A$-module $B$, autrement dit 
  s'il existe une application
  $A$-linéaire\footnote{On ne demande pas à $r$ d'être un morphisme d'anneaux.} $r\colon  B\to A$ telle que $r(f(a))=a$ pour tout $a\in A$. 
\medskip
  
\begin{rema}
\textcite{BhattInv} revisite et simplifie la preuve d'\textcite{AndreDSC}, ce qui lui permet 
d'établir la version dérivée suivante du théorème~\ref{DSC}, conjecturée par de Jong: si 
$A$ est un anneau régulier et si $f\colon  X\to {\rm Spec}(A)$ est un morphisme propre et surjectif, alors 
le morphisme $A\to {\rm R}\Gamma(X, \mathcal{O}_X)$ est scindé dans la catégorie dérivée ${\rm D}(A)$. 
Si $A$ est une $\mathbf{Q}$-algèbre cela se déduit des travaux de  \textcite{Kovacs}, le cas ${\rm car}(A)=p$ 
avait été traité par  \textcite{Bhattdspl}.
\end{rema}

  \textcite{H2} a démontré le théorème~\ref{DSC}   
 pour les anneaux réguliers contenant un corps, et a réduit, par un argument très indirect (cf.\ th\'eor\`eme~6.1 de \cite{H4}) le cas 
      général à celui d'un anneau local régulier complet, non ramifié, de corps résiduel algébriquement clos.
        Une avancée spectaculaire est due à \textcite{Heitmann} : il a démontré la conjecture quand  
    $\dim A=3$ (le cas $\dim A\leq 2$ est une conséquence de la formule d'Auslander--Buchsbaum). 

Le lien entre les techniques perfectoïdes (plus précisément les presque mathématiques et 
le théorème de presque pureté de  \textcite{Falp,Falal}) et la conjecture du facteur direct semble avoir été remarqué depuis un certain temps\footnote{Par exemple, voici ce que m'écrit Wies{\l}awa Nizio{\l}: \og in 2001 Lorenzo Ramero visited me in Utah and gave a talk at the Number Theory seminar on his work with Gabber and their attempt to prove the almost purity conjecture. Paul Roberts was in the audience and was really surprised by the similarity of almost math techniques with the recent proof by Heitmann of the direct summand conjecture in dim $3$. Heitmann worked in the almost setting and then at some point was able to descend to the usual setting (via some finiteness properties ?). Roberts got all excited about this and we had a seminar running for a semester on almost math and commutative algebra. It did not get anywhere because, of course, we did not have the almost purity in general at that time.\fg{}}, mais ce n'est qu'en $2014$ que  \textcite{Bhattds} a obtenu le premier résultat un peu général via ces techniques, en traitant
le cas où $B[\frac{1}{p}]$ est étale sur $A[\frac{1}{p}]$ 
(et même sous des hypothèses plus faibles). 
C'est ce cercle d'idées qui mènera à la preuve de la conjecture, 
mais il a fallu attendre les travaux d'\textcite{AndreDSC,AndreIHES}
pour traiter le cas général.

\begin{rema}\label{trivia} 

a)  Une extension finie $f\colon  A\to B$ d'anneaux noethériens
    est scindée si et seulement si l'extension induite $A_{\mathfrak{m}}\to B_{\mathfrak{m}}$ l'est pour tout idéal maximal 
   $\mathfrak{m}$ de $A$: l'existence d'un scindage équivaut à la surjectivité de 
     l'application\footnote{On note ${\rm Mod}_A$ la catégorie des $A$-modules.} $${\rm ev}_1\colon  {\rm Hom}_{{\rm Mod}_A}(B,A)\to A,\,\, r\mapsto r(1),$$ qui   
    peut se tester en localisant en tout idéal maximal $\mathfrak{m}$ de $A$, or ${\rm Hom}_{{\rm Mod}_A}(B,A)_{\mathfrak{m}}\simeq {\rm Hom}_{{\rm Mod}_{A_{\mathfrak{m}}}}(B_{\mathfrak{m}}, A_{\mathfrak{m}})$
    puisque $B$ est un $A$-module de présentation finie. De même, $f$ est scindée 
   si et seulement si l'extension $C\to C\otimes_A B$ l'est pour une extension fidèlement plate $C$ de $A$, car on peut tester la surjectivité de ${\rm ev}_1$ après changement de base à $C$.
  
b) La preuve du théorème~\ref{DSC} se ramène au cas d'une extension finie 
 $f\colon  A\to B$ avec~$A$ local régulier complet et $B$~intègre (donc local et complet). En effet, par a) on peut supposer que $A$~est local, puis complet, en utilisant l'extension fidèlement plate $A\to \widehat{A}$.  Si 
   $\wp$ est un idéal premier de $B$ tel que 
   $\dim (B/\wp)=\dim B$, alors $\dim A/(\wp\cap A)=\dim A$, puis $\wp\cap A=\{0\}$ (car $A$ est local et intègre), et tout scindage de l'extension finie $A\to B/\wp$ en fournit un pour $A\to B$. 

c) Si $A$ est une $\mathbf{Q}$-algèbre intègre et normale, alors toute extension finie
$f\colon  A\to B$ est scindée. En effet, comme dans b) on peut supposer que $B$ est intègre. Si 
$K$ et $L$ sont les corps des fractions de $A$ et de $B$, par normalité de $A$ la trace 
${\rm Tr}_{L/K}\colon  L\to K$ envoie $B$ dans 
$A$, et $\frac{1}{[L:K]}{\rm Tr}_{L/K}\colon  B\to A$ fournit un
scindage. Donc
pour les $\mathbf{Q}$-algèbres  le théorème~\ref{DSC} est trivial, et pas optimal.
    
    d) Si $\dim A\leq 2$ le théorème~\ref{DSC} se déduit de la formule d'Auslander--Buchsbaum. Si 
    $A$ est de caractéristique positive on dispose de toute une variété de preuves pas (trop) difficiles du théorème~\ref{DSC}, voir l'exemple $1.3$ de \textcite{Bhattdspl} pour une preuve cohomologique, et le paragraphe~$6.2$ de \textcite{H4} pour une preuve courte.
\end{rema}

\subsection{Fragmenteurs}

  Appelons \emph{fragmenteur} (\emph{splinter} en anglais) un anneau intègre $A$ tel que toute extension finie de $A$ soit scindée. Le théorème~\ref{DSC} affirme que 
   les anneaux réguliers sont fragmenteurs.
   Tout fragmenteur est normal\footnote{Si $x\in {\rm Frac}(A)$ est entier sur $A$, alors
  $A\to A[x]$ est une extension finie et elle n'est pas scindée si $x\notin A$, puisque toute rétraction $A$-linéaire $r\colon A[x]\to A$ doit envoyer $x$ sur lui-même (si $x=\frac{a}{b}$ avec $a,b\in A$ et $b\ne 0$, alors $a=r(a)=r(bx)=br(x)$).}, et la réciproque est vraie pour les 
   $\mathbf{Q}$-algèbres intègres (remarque~\ref{trivia}). 
 La situation est nettement plus compliquée en caractéristique positive ou mixte.  \textcite{HH1,HH2}  ont montré que les fragmenteurs noethériens de caractéristique positive et localement excellents\footnote{i.e.\ dont les localisés en tout idéal maximal sont excellents.} sont des anneaux de Cohen--Macaulay, et \textcite{Bhattabs} vient de montrer, dans son travail spectaculaire, que cela reste vrai en caractéristique mixte (toujours sous des hypothèses d'excellence). 
 
 Une source importante de fragmenteurs est la théorie des représentations des groupes (linéairement) réductifs: si un tel groupe $G$ agit sur une $k$-algèbre 
 $R$ qui est un anneau régulier ($k$ étant un corps), alors l'anneau des invariants $R^G$ est un fragmenteur (car l'inclusion 
 $R^G\to R$ est scindée, via l'opérateur de Reynolds, $R$ est un fragmenteur, et un facteur direct d'un fragmenteur en est encore un).

 Voici un exemple (dû à  \cite{H2}) d'anneau normal, de Cohen--Macaulay (même intersection complète), et non fragmenteur. Soient~$k$ un corps de caractéristique~$2$ et $R=k[X,Y,Z]/(X^3+Y^3+Z^2)=k[x,y,z]$. Le morphisme 
 $R\to k[U,V]$ envoyant $x,y,z$ sur $U^2, V^2, U^3+V^3$ est fini, injectif et non scindé: s'il était scindé on aurait $R\cap Ik[U,V]=I$ pour tout idéal~$I$ de~$R$, or $z\notin (x,y)$ et $U^3+V^3\in (U^2, V^2)$. 
     
     Un schéma~$S$ est dit \emph{fragmenteur} si pour tout morphisme fini surjectif 
    $f\colon  X\to S$ le morphisme $\mathcal{O}_S\to f_*\mathcal{O}_X$ est scindé dans la catégorie ${\rm Coh}(S)$ des faisceaux cohérents sur~$S$. On dit que $S$~est un \emph{$D$-fragmenteur} si pour tout morphisme propre surjectif $f\colon  X\to S$ le morphisme 
    $\mathcal{O}_S\to {\rm R}f_*\mathcal{O}_X$ est scindé dans $D({\rm Coh}(S))$.  \textcite{Bhattdspl} a montr\'e qu'un 
    $\mathbf{F}_p$-schéma noethérien est fragmenteur si et seulement s'il est $D$-fragmenteur. Pour comparer, pour un 
     $\mathbf{Q}$-schéma le caractère fragmenteur est (plus ou moins, i.e.\ sous des hypothèses faibles) équivalent à la normalité, alors que 
    le caractère $D$-fragmenteur est (plus ou moins) équivalent, par un théorème de  \textcite{Kovacs}, au fait que les singularités de 
    $S$ sont au pire rationnelles, cf.\ exemples~$1.1$ et~$1.2$ de \textcite{Bhattdspl}.

 Voir \textcite{AndreFiorot} et \textcite{Bhattdspl} pour plus de détails et d'exemples concernant 
  les fragmenteurs.

    \subsection{Scindage et pureté}

Un morphisme d'anneaux $f\colon  A\to B$ est dit \emph{pur} s'il est \emph{universellement injectif}, i.e.\ si 
le morphisme induit $C\to B\otimes_A C$ reste injectif pour toute $A$-algèbre $C$, auquel cas il reste injectif pour tout 
$A$-module $C$. Toute extension scindée est clairement pure.
Voici deux incarnations importantes de la notion de pureté:


 $\bullet$ d'un point de vue catégorique, un morphisme $f\colon  A\to B$ est pur si et seulement si 
le foncteur $(-)\otimes_A B\colon  {\rm Mod}_A\to {\rm Mod}_B$ est fidèle. Rappelons que 
$f$ est dit \emph{plat} (resp.\ \emph{fidèlement plat}) si le foncteur $(-)\otimes_A B$ est exact (resp.\ exact et fidèle). Ainsi tout morphisme fidèlement plat est
pur; 

$\bullet$  \textcite{Olivier} a montré qu'un morphisme 
 $f\colon  A\to B$ est pur si et seulement si la théorie de la descente fonctionne bien\footnote{Autrement dit 
 le foncteur envoyant $M\in {\rm Mod}_A$ sur $B\otimes_A M$ muni de sa donnée de descente canonique 
 induit une équivalence entre ${\rm Mod}_A$ et la catégorie des $B$-modules 
 $N$ munis d'un isomorphisme $N\otimes_A B\simeq B\otimes_A N$ de $B\otimes_A B$-modules vérifiant la condition usuelle de cocycle.}
  pour les $B$-modules. Comme il a été remarqué dans \textcite{AndreFiorot}, cela permet de voir les morphismes purs comme les recouvrements pour la topologie canonique\footnote{Il s'agit de la topologie de Grothendieck la plus fine pour laquelle tous les préfaisceaux représentables sont des faisceaux.} sur la catégorie des schémas affines, ce qui en fournit une interprétation géométrique. 
     
     La conjecture du facteur direct se réincarne (cf.~théorème~\ref{bigCM})
en un énoncé de pureté grâce au résultat suivant:
     
     \begin{prop}\label{pure}
   Toute extension finie et pure d'anneaux noethériens est scindée.
           \end{prop}
     
     \begin{proof}  On peut supposer que $A$ est local et complet (remarque~\ref{trivia}). Si
$E$ est une enveloppe injective du corps résiduel de $A$, le morphisme 
$E\to E\otimes_A B$ est injectif par pureté. Puisque $E$ est un $A$-module injectif l'identité de 
   $E$ se prolonge en un morphisme de $A$-modules $u\colon  E\otimes_A B\to E$, d'où un morphisme $A$-linéaire 
     $B\to {\rm End}_A(E)$. Par dualité de Matlis on a ${\rm End}_A(E)\simeq A$, et la composée $B\to  {\rm End}_A(E)\simeq A$ est un scindage de $f$.
            \end{proof}
              
On peut utiliser les liens entre scindage et pureté pour obtenir des conséquences importantes du théorème~\ref{DSC}. Nous allons mentionner deux telles applications.
Si $f\colon  A\to B$ est une extension pure, alors $A\cap IB=I$ pour tout idéal $I$ de 
$A$ (cela ne fait que traduire l'injectivité du morphisme $A/I\to B/IB=B\otimes_A A/I$). Sous des hypothèses faibles
cette propriété de contraction d'idéaux caractérise la pureté:  
\textcite{H10}
a montré qu'une extension finie $A\to B$ d'un anneau noethérien intègre et normal est scindée si $A\cap IB=I$ pour tout idéal~$I$ de~$A$. La conjecture du facteur direct et le th\'eor\`eme suivant sont donc \'equivalents (une implication \'etant triviale, comme remarqu\'e ci-dessus).

        \begin{theo}
        Si $f\colon  A\to B$ est une extension finie d'un anneau régulier $A$, alors $A\cap IB=I$ pour tout idéal 
        $I$ de $A$.
        \end{theo}

    Pour la deuxième application, rappelons 
     qu'un morphisme d'anneaux $f\colon  A\to B$ \emph{descend la platitude} si pour tout $M\in {\rm Mod}_A$ la platitude sur $B$ de $B\otimes_A M$ force celle de $M$. Par exemple, toute extension pure descend la platitude \parencite{Olivier}. 
  \textcite{RG} ont montré, généralisant un théorème de Ferrand, que toute extension finie descend la platitude. Ils ont demandé (question 1.4.3, Seconde Partie de loc.cit.) si toute extension entière d'un anneau noethérien\footnote{L'hypothèse que $A$ soit noethérien n'est pas superflue!} $A$ descend la platitude, et ont aussi expliqué qu'il suffit de résoudre ce problème quand $A$ est régulier; or dans ce cas 
  le théorème~\ref{DSC} permet de conclure puisque 
  toute extension entière $A\to B$ est pure, en tant que colimite filtrante d'extensions finies, donc pures. 
  Donc la conjecture du facteur direct fournit une réponse positive à la question de Raynaud et Gruson. 
 \textcite{Ohi} a montré que les théorèmes~\ref{RG} et~\ref{DSC} sont en fait équivalents.
 
    \begin{theo}\label{RG}
  Toute extension entière d'un anneau noethérien descend la platitude. 
    \end{theo}

Si $A$ est un anneau et si $f\colon  M\to N$ et $g\colon  N\to P$ sont des morphismes de $A$-modules tels que 
  $g\circ f$ soit pur, alors $f$ est clairement pur. Le théorème~\ref{DSC} devient ainsi une conséquence directe du théorème fondamental suivant (qui sera discut\'e dans le \S~\ref{coma}) et de la proposition~\ref{pure}. 
    
  \begin{theo}\label{bigCM}
  Pour toute extension finie $A\to B$ d'un anneau régulier $A$ il existe un morphisme d'anneaux $B\to C$ tel que 
  $A\to C$ soit fidèlement plat.
  \end{theo}

\begin{rema}
\begin{enumerate}

\item Géométriquement, on peut (suivant \cite{AndreFiorot}) reformuler le théorème~\ref{DSC} (resp.~\ref{bigCM}) comme suit: si $Y$ est un schéma noethérien régulier, alors tout 
morphisme fini surjectif $f\colon  X\to Y$ est un recouvrement pour la topologie canonique (respectivement pour la topologie\footnote{Attention, on dit bien la topologie et non pas la pré-topologie (l'énoncé en question serait faux pour la pré-topologie fpqc). Un morphisme $f\colon X\to Y$ de sch\'emas affines est un recouvrement pour la topologie fpqc s'il existe un morphisme $X'\to X$ de sch\'emas affines tel que $X'\to Y$ soit fid\`element plat.} 
fpqc) sur la catégorie des schémas. 

\item Si l'on part d'une extension finie et pure $A\to B$, il est possible qu'un anneau $C$ comme dans le théorème 
\ref{bigCM} n'existe pas (même avec $A$ normal), i.e.\ la topologie fpqc est strictement plus faible que la topologie canonique. 
L'exemple~5.5 d'\textcite{AndreFiorot} 
  est particulièrement simple (à énoncer, pas à prouver\dots{}):  l'inclusion 
$A\coloneqq \mathbf{C}[X,Y]^{\mathbf{Z}/2\mathbf{Z}}\to B\coloneqq \mathbf{C}[X,Y]$ (l'élément non trivial de $\mathbf{Z}/2\mathbf{Z}$ agissant par $(X,Y)\mapsto (-X,-Y)$)
est pure (immédiat) mais pas un recouvrement fpqc (cela utilise les constructions de \textcite{RG} (1.4.1.1), les premiers à avoir fourni des contre-exemples).
\end{enumerate}
\end{rema}

   \subsection{Algèbres de Cohen--Macaulay pour les anneaux locaux noethériens}\label{coma}

  Contrairement à la conjecture du facteur direct, le théorème~\ref{bigCM} ci-dessus est très difficile
  même quand $A$ est une $\mathbf{Q}$-algèbre (toutes les preuves existantes passent par une réduction, via la technique des ultrafiltres, 
  au cas des anneaux de caractéristique positive, où le Frobenius fait des merveilles).
   Démontré par \textcite{HH1}
    quand $A$ contient un corps, et par \textcite{AndreDSC} (et récemment par \textcite{Bhattabs}) pour $A$ d'inégale caractéristique, ce théorème est une reformulation d'une conjecture de \textcite{H1} concernant l'existence d'une $A$-algèbre de Cohen--Macaulay pour tout 
  anneau local noethérien $A$. Pour expliquer le lien, nous devons faire quelques rappels. 
  
   Soit $(A, \mathfrak{m})$ un anneau local noethérien. Une suite $x=(x_1,\ldots, x_d)$ dans $\mathfrak{m}$ est un 
   \emph{système de paramètres} (ou \emph{suite sécante maximale}) si $d=\dim A$ et si
   $A/(x_1,\ldots, x_d)$ est de dimension $0$, autrement dit s'il existe $t\geq 1$ tel que 
   $\mathfrak{m}^t\subset (x_1,\ldots, x_d)$. D'autre part, une suite $z=(z_1,\ldots, z_k)$ d'éléments de 
   $A$ est une \emph{suite régulière} dans un $A$-module $M$ si 
   $M/(z_1,\ldots, z_k)M\ne 0$ et si la multiplication par $z_{i+1}$ est injective dans $M/(z_1,\ldots, z_i)M$ pour tout $0\leq i<k$. On dit que 
   $M$ est un $A$-\emph{module de Cohen--Macaulay} si 
  tout système de paramètres dans $A$ est une suite régulière dans $M$.
      On dit qu'une $A$-algèbre 
   $B$ est une $A$-\emph{algèbre de Cohen--Macaulay} si $B$ est un $A$-module de Cohen--Macaulay. 
   Aucune hypothèse de finitude n'étant imposée à $B$, il ne faut donc pas confondre\footnote{Les diverses appellations dans la
   littérature anglophone donnent un peu le tournis: ce que nous définissons ici correspond à une \emph{balanced (big) CM $A$-algebra}.} cette définition avec celle d'un anneau de Cohen--Macaulay qui se trouve être une $A$-algèbre\footnote{Le langage de l'algèbre commutative n'est pas vraiment commutatif\dots{}}. 

    Le résultat suivant (\cite[th\'eor\`eme 1.7]{BJ} et \cite{HH2}, 2.1.d) fait le lien avec le théorème~\ref{bigCM}. 
         
     \begin{prop}\label{CMstandard}
     
    Soit $(A, \mathfrak{m})$ un anneau local noethérien et soit $M$ un $A$-module.
    
    a) Si $A$ possède un système de paramètres qui est une suite régulière dans $M$, alors 
    le complété $\mathfrak{m}$-adique de $M$ est un $A$-module de Cohen--Macaulay.
    
    b) Si $A$ est régulier, un $A$-module est de Cohen--Macaulay si et seulement s'il est fidèlement plat sur $A$.
      
     \end{prop}

    Compte tenu de cette proposition, le théorème~\ref{bigCM} devient une conséquence\footnote{Voici l'argument, suivant \cite{AndreDSC}: 
    soit $\mathfrak{m}$ un id\'eal maximal de $A$, soit $\widehat{A}_{\mathfrak{m}}\coloneqq \varprojlim_{n} A_{\mathfrak{m}}/\mathfrak{m}^n A_{\mathfrak{m}}$ et soit 
    $\wp$ un id\'eal premier minimal de 
    $B\otimes_A \widehat{A}_{\mathfrak{m}}$ tel que $\widehat{A}_{\mathfrak{m}}\to B'\coloneqq (B\otimes_A \widehat{A}_{\mathfrak{m}})/\wp$ reste injectif et fini (cf.\ remarque~\ref{trivia}, point b)). Alors $B'$ est local, et on choisit
  une $B'$-alg\`ebre de Cohen--Macaulay $C(\mathfrak{m})$. Alors $C(\mathfrak{m})$ est fid\`element plate sur $\widehat{A}_{\mathfrak{m}}$ (proposition~\ref{CMstandard}). Comme $A$ est noeth\'erien, $C\coloneqq \prod_{\mathfrak{m}} C(\mathfrak{m})$ est plate sur 
  $A$, et m\^eme fid\`element plate puisque l'image de ${\rm Spec}(C)\to {\rm Spec}(A)$ est stable par g\'en\'erisation (par platitude) et ne contient pas de point ferm\'e par construction.}    
du résultat suivant, qui répond à la 
conjecture\footnote{À l'origine elle concernait l'existence de $A$-modules de
  Cohen--Macaulay : 
 \textcite{H1} professe un certain pessimisme 
concernant l'existence de $A$-algèbres de Cohen--Macaulay\dots{}}
    de Hochster mentionnée ci-dessus.
    \begin{theo}\label{existbig}
    Pour tout anneau local noethérien $A$ il existe une $A$-algèbre de Cohen--Macaulay.
    \end{theo}

  Les constructions d'André que l'on verra dans la suite de cet exposé font intervenir des anneaux perfectoïdes, et les 
 $A$-algèbres de Cohen--Macaulay qui en sortent ne sont presque jamais noethériennes.
  Ce n'est pas une grande surprise, puisque Hochster a déjà montré (th\'eor\`eme~6.1 de \cite{H1}) que si $A$ est un anneau local noethérien complet, normal, non Cohen--Macaulay, 
 et contenant $\mathbf{Q}$, alors $A$ n'admet pas de $A$-algèbre de Cohen--Macaulay noethérienne. 
  Voir aussi  
 \textcite{BhattVA} pour des obstructions cohomologiques à l'existence de \og petites algèbres de Cohen--Macaulay\fg{}\footnote{i.e.\ des algèbres de Cohen--Macaulay qui sont des extensions finies de l'anneau de base.} en caractéristique positive. 
 
 Le théorème~\ref{existbig} était connu pour des anneaux $A$ d'égale caractéristique ou
  quand $\dim A\leq 3$, grâce aux travaux de \textcite{HH1} et \textcite{H6}
  (le cas d'inégale caractéristique en dimension $3$ utilise de manière cruciale les travaux de  
  \cite{Heitmann}). Le cas restant est dû à  \textcite{AndreDSC} (voir le \S~\ref{clotabs} pour l'approche de Bhatt). Cependant, pour l'instant ces techniques ne semblent pas permettre des avancées vers une autre conjecture célèbre de Hochster\footnote{On m'informe que l'on conjecture d\'esormais le contraire...}: 
{\it tout anneau local noethérien complet $A$ possède un 
  $A$-module de Cohen--Macaulay de type fini}. 
Cette conjecture a un intérêt particulier puisqu'elle entraîne la conjecture de Serre (ouverte aussi depuis environ $50$ ans) 
  sur la (stricte) positivité des multiplicités d'intersection, mais en dehors du cas de la dimension $\leq 2$ très peu de choses sont connues.


\begin{rema}
 Dans une direction un peu diff\'erente mentionnons la conjecture de macaulayfication de Faltings: pour tout sch\'ema noeth\'erien quasi-excellent $X$ il existe un 
 sch\'ema de Cohen--Macaulay $\tilde{X}$ et un morphisme projectif birationnel $\pi\colon \tilde{X}\to X$ qui est un isomorphisme au-dessus du lieu (ouvert) de Cohen--Macaulay de $X$. Voir \textcite{Ces} pour une preuve, m\^eme sous des hypoth\`eses plus faibles. 
\end{rema}

      \subsection{Les conjectures homologiques}
      
      On trouvera d'excellentes présentations, par exemple dans \textcite{H7} et \textcite{Rhom}, de l'écheveau des \og conjectures homologiques\fg{} et de leurs diverses imbrications, nous nous contentons dans ce paragraphe
      de quelques extraits, qui découlent (par les travaux de \textcite{H4}, \textcite{PeskineSzpiro} et \textcite{EG}) de l'existence d'algèbres (ou même seulement de modules) de Cohen--Macaulay pour les anneaux locaux noethériens.

      \begin{theo}\label{hconj}
      
      Soit $(R, \mathfrak{m})$ un anneau local noethérien, soit $x\in R$ et soient $M,N$ des modules non nuls de type fini sur $R$.
      
     {\rm a)} {\rm (conjecture de M. Auslander)}
 Si $M$ est de dimension projective finie et sans $x$-torsion, alors $R$ est sans $x$-torsion.

     {\rm b)} {\rm (question de Bass)} Si $M$ est de dimension injective finie alors $R$ est Cohen--Macaulay. 
      
    {\rm  c)} {\rm (conjecture d'intersection de Peskine--Szpiro)}
    Si $M\otimes_R N$ est de longueur finie, alors $\dim N\leq {\rm pd}(M)$.
          
     {\rm d)} {\rm (\og improved new intersection conjecture\fg{})}
     Soit $C$ un complexe de $R$-modules libres de type fini, concentré en degrés $[0,d]$, et tel que $H_{>0}(C)$ soit de longueur finie. Si $H_0(C)$ poss\`ede un g\'en\'erateur minimal non nul $c$ tu\'e par une puissance de $\mathfrak{m}$, alors 
   $\dim R\leq d$. 
     
    {\rm e)}  {\rm (\og conjecture des syzygies d'Evans--Griffith\fg{})} Supposons que $R$ est int\`egre de Cohen--Macaulay. Si 
    $M$ est un $k$ième module de syzygies, de dimension projective finie et non libre, alors 
    $M$ est de rang $\geq k$.
          
                  \end{theo}

         Tous les résultats ci-dessus étaient connus à l'époque de l'article de \textcite{H1} quand 
      l'anneau local contient un corps 
     ou est de dimension $\leq 2$, et ouverts dans les cas restants.      
     Evans et Griffith avaient montré\footnote{Sans le dire\dots{} voir la section $2$ de \textcite{H4}.} que l'existence de modules de Cohen--Macaulay 
     implique le point d), qui implique e). 
     En combinant les travaux de  \textcite{H4} et \textcite{Dutta}, on montre que d) est équivalent à la conjecture du facteur direct, et il implique c). 
           Le point c) avait été démontré par  \textcite{PeskineSzpiro} en présence d'un corps, et par \textcite{Robint} en général. 
            Peskine et Szpiro ont montré que c) implique b) et a).

    \textcite[th\'eor\`eme 6.1]{H4} a montré  que la conjecture du facteur direct est équivalente 
à cette autre conjecture, la \emph{conjecture monomiale}, tout aussi charmante:
        
    \begin{theo}\label{Cmon}
    Si $x_1,\ldots, x_n$ est un système de paramètres d'un anneau local noethérien $(A,\mathfrak{m})$, alors 
  l'équation $$(x_1\cdots x_n)^k=y_1x_1^{k+1}+\cdots +y_nx_n^{k+1}$$
   n'a pas de solutions $(y_1,\ldots, y_n)\in A^n$ pour $k\geq 1$.
    \end{theo}

     Ce théorème découle facilement du théorème~\ref{existbig}, voici les grandes lignes de l'argument. Si 
     $x_1,\ldots, x_n$ est une suite régulière dans un $A$-module $M$ (pour un anneau $A$) et si
    $m_1,\ldots, m_n\in M$ vérifient $x_1m_1+\cdots+x_n m_n=0$ alors $m_i\in (x_1,\ldots, x_n)M$ pour tout $i$\footnote{Comme $x_n$ est non diviseur de z\'ero dans 
    $M/(x_1,\ldots, x_{n-1})M$ on a $m_n=\sum_{i=1}^{n-1} x_i m_i'$ pour certains $m_i'\in M$.
    On a donc $\sum_{i=1}^{n-1} x_i(m_i+x_nm_i')=0$, ce qui permet de conclure par r\'ecurrence.}. 
     On en déduit\footnote{Si tous les $a_i$ sont nuls, cela est \'evident, soit donc $i$ tel que $a_i\ne 0$ et posons $z=\prod_{j\ne i} x_j^{a_j}$ et $x'_j=x_j^{a_j+1}$ pour $j\ne i$. Comme $x_i^{a_i}(x_im_i-zm)+\sum_{j\ne i} x'_jm_j=0$ et $(x_1',\ldots, x_{i-1}', x_i^{a_i}, x_{i+1}',\ldots, x_n')$ est une suite r\'eguli\`ere, on obtient 
     $x_im_i-zm\in (x_1',\ldots, x_{i-1}', x_i^{a_i}, x_{i+1}',\ldots, x_n')M$, donc $zm\in 
     (x_1',\ldots, x_{i-1}', x_i, x_{i+1}',\ldots, x_n')M$. On conclut par r\'ecurrence sur le nombre de $a_j$ non nuls.}
     que si $a_i$ sont des entiers positifs et si $m,m_i\in M$ vérifient 
    $$x_1^{a_1}\cdots x_n^{a_n}m=x_1^{a_1+1}m_1+\cdots +x_n^{a_n+1} m_n,$$
    alors $m\in (x_1,\ldots, x_n)M$. En particulier, comme $M/(x_1,\ldots, x_n)M\ne 0$, on ne peut pas avoir 
    $x_1^{a_1}\cdots x_n^{a_n}\in (x_1^{a_1+1},\ldots, x_n^{a_n+1})$.
    
    \begin{rema}
    Le th\'eor\`eme ci-dessus est \`a comparer avec l'\'enonc\'e suivant, qui d\'ecoule du th\'eor\`eme de Brian\c{c}on-Skoda: si 
    $A$ est un anneau r\'egulier de dimension $\leq n$, alors $$(x_1\cdots x_{n+1})^n\in (x_1^{n+1},\ldots, x_{n+1}^{n+1})$$
    pour tous $x_1,\ldots, x_{n+1}\in A$ (ce n'est d\'ej\`a pas facile \`a d\'emontrer pour $n=2$!).
    \end{rema}
     
     \subsection{Algèbre dans la clôture intégrale absolue}\label{clotabs}
     
       Soit $A$ un anneau intègre, et fixons une clôture algébrique~$K^+$ du corps des fractions~$K$ de~$A$. La \emph{clôture intégrale 
       absolue}~$A^+$ de~$A$ est la clôture intégrale de~$A$ dans~$K^+$. Elle est bien définie à isomorphisme près, et faiblement fonctorielle en~$A$: tout morphisme $f\colon  A\to B$ d'anneaux intègres induit\footnote{Le cas d'une injection étant clair, supposons que $f$~est surjectif; on 
       peut trouver un idéal premier~$\mathfrak{q}$ de~$A^+$ au-dessus de $\wp\coloneqq \ker f$, et alors $A/\wp$ s'injecte dans $A^+/\mathfrak{q}$, ce dernier étant isomorphe à~$B^+$.}
        un morphisme $f^+\colon  A^+\to B^+$. 
        
        Pour voir le lien avec les conjectures discutées ci-dessus, mentionnons deux énoncés sympathiques. Soit 
        $(A, \mathfrak{m})$ un anneau local noethérien complet et intègre. 
        Si $\dim A\leq 2$ alors $A^+$ est une $A$-algèbre de Cohen--Macaulay (tout anneau noethérien intègre et normal de dimension $2$ est de Cohen--Macaulay). 
        \`A partir de la dimension $3$ (resp.\ $4$) et 
       en égale caractéristique nulle (resp.\ en inégale caract\'eristique) $A^+$ n'est plus une $A$-alg\`ebre de Cohen--Macaulay.

       Dans un véritable tour de force d'algèbre commutative\footnote{Voir 
  l'article de  \textcite{HL} pour une preuve plus simple.},  
  \textcite{HH1} ont démontré le résultat suivant, qui implique immédiatement le théorème~\ref{existbig} en caractéristique positive, et même une forme 
  plus forte, car la construction devient faiblement fonctorielle: 
  
  \begin{theo}
  Si $A$ est un anneau local noethérien intègre et excellent\footnote{Par exemple complet.} de caractéristique $p>0$, alors 
  $A^+$ est une $A$-algèbre de Cohen--Macaulay.
  \end{theo}
  
  Cette recette ne fonctionne plus si $A$ est une $\mathbf{Q}$-algèbre, et la preuve du théorème~\ref{existbig} (et la fonctorialité faible) passe par une réduction délicate à la caractéristique positive. 
    
    Nous finissons cette longue introduction avec le résultat spectaculaire suivant, 
dû à  \textcite{Bhattabs}, et qui fournit un analogue du théorème de Hochster--Huneke
    en inégale caractéristique. Il entraîne immédiatement le théorème~\ref{existbig} en inégale caractéristique, ainsi que la fonctorialité faible des 
    algèbres de Cohen--Macaulay. La preuve est très difficile et utilise toute la palette des développements en théorie de Hodge $p$-adique de ces dix dernières années, ainsi que la théorie prismatique de  \textcite{BS} . Elle mériterait sans doute un exposé à part entière\dots{} 
  
  \begin{theo}\label{Bhattamazing}
  Soit $(A, \mathfrak{m})$ un anneau local noethérien intègre et excellent de caractéristique résiduelle 
  $p>0$. Le complété $p$-adique de $A^+$ est une $A$-algèbre de Cohen--Macaulay.
    \end{theo}
  
  \begin{rema}
  \begin{enumerate}
  
  \item \textcite{H4} a démontré que 
   la conjecture du facteur direct est équivalente au sympathique énoncé suivant: pour tout 
   anneau local noethérien complet et intègre $(A, \mathfrak{m})$ le $A$-dual de~$A^+$ est non nul. 
   Si $A$~est de dimension~$n$, par dualité locale cela équivaut à 
   ${\rm H}^n_{\mathfrak{m}}(A^+)\ne \{0\}$. Cette non annulation se voit facilement à partir du théorème~\ref{DSC} (le point délicat est que l'on peut aller dans l'autre sens): 
   par le théorème de Cohen on peut supposer\footnote{Cela demande un petit peu de travail...} que 
   $A$~est aussi régulier, mais alors toute extension finie $A\to B$ dans~$A^+$ est scindée, donc le morphisme ${\rm H}^n_{\mathfrak{m}}(A)\to {\rm H}^n_{\mathfrak{m}}(B)$ est injectif, et il en est de même de 
      ${\rm H}^n_{\mathfrak{m}}(A)\to {\rm H}^n_{\mathfrak{m}}(A^+)$, ce qui permet de conclure puisque
      ${\rm H}^n_{\mathfrak{m}}(A)\ne \{0\}$. 
      
      \item Dans la situation du théorème~\ref{Bhattamazing}, on montre (c'est le coeur de l'article de \textcite{Bhattabs})
     que ${\rm H}^i_{\mathfrak{m}}(A^+/p)$ est nul pour $i<\dim (A/p)$ et ${\rm H}^i_{\mathfrak{m}}(A^+)$ est nul pour 
     $i<\dim A$. Même en dimension $3$ ce genre d'énoncé va beaucoup plus loin que ceux de  \textcite{Heitmann} (on passe d'une presque nullité à une vraie nullité!). Cela lui permet de montrer que si $A$ est un fragmenteur, alors $A$ est de Cohen--Macaulay: par le même argument que ci-dessus la flèche ${\rm H}^i_{\mathfrak{m}}(A/p)\to {\rm H}^i_{\mathfrak{m}}(A^+/p)$ est injective (car $A$ est un fragmenteur), donc ${\rm H}^i_{\mathfrak{m}}(A/p)$ est nul pour $i<\dim (A/p)$, et 
     $A$ est de Cohen--Macaulay. 
           
      \end{enumerate}
\end{rema}

 \section{Anneaux perfectoïdes: aspects algébriques}
 
   Le but de cette section est de mettre ensemble un certain nombre de résultats sur les anneaux perfectoïdes \og entiers\fg{}, qui sont éparpillés fa\c{c}on puzzle dans la littérature. Nous renvoyons le lecteur 
   aux articles fondamentaux de \textcite[section~3]{BMS1}, \textcite[sections~2 et~3]{BS}, \textcite[section~2]{CS}, au livre de \textcite{KL} et aux exposés de  \textcite{FonBourb} et \textcite{Morrow} dans ce Séminaire pour plus de détails.
 
    On fixe pour toute la suite un nombre premier $p$ et on note ${\rm Perf}_{\mathbf{F}_p}$ la catégorie des 
    $\mathbf{F}_p$-algèbres parfaites, i.e.\ celles dont le morphisme de Frobenius $\varphi\colon  x\mapsto x^p$ est un automorphisme (une telle algèbre est donc réduite). La réduction modulo~$p$ et 
    le foncteur 
     $W(-)$ des vecteurs de Witt $p$-typiques
          induisent des équivalences quasi-inverses entre ${\rm Perf}_{\mathbf{F}_p}$ et la catégorie des 
     anneaux $p$-complets, sans $p$-torsion, dont la réduction modulo~$p$ est parfaite.
    
    \subsection{Vecteurs de Witt}
             
     Pour tout $R\in {\rm Perf}_{\mathbf{F}_p}$ il existe une unique application multiplicative, \emph{mais pas forcément additive}
    $[\,\, \,\,]\colon  R\to W(R)$ telle que $[a]\equiv a\pmod p$ pour tout 
   $a\in R$. Tout $x\in W(R)$ s'écrit $x=\sum_{n\geq 0} [x_n]p^n$ pour une unique suite $(x_n)_{n\geq 0}$ dans $R$.
Il sera très utile de considérer $x$ comme une \og fonction holomorphe de la variable $p$\fg{}. 
 Il est donc tentant d'introduire la notation   
   $$x(0)\coloneqq x_0,\,\, x'(0)\coloneqq x_1.$$
  L'application $x\mapsto x(0)$ induit un isomorphisme d'anneaux $W(R)/p\simeq R$, et un calcul immédiat montre que 
  pour tous $x,y\in W(R)$ on a $$(xy)'(0)=x(0)y'(0)+y(0)x'(0), \eqno (1)$$
  ce qui explique la notation. L'optimisme dégagé par cette observation est tempéré par le fait que $(x+y)'(0)\ne x'(0)+y'(0)$ en général. Cependant la relation ci-dessus jouera un rôle important à plusieurs reprises.
  
  \begin{rema}\label{pder}
   L'anneau $W(R)$ est muni d'un relèvement du Frobenius $\varphi\colon  W(R)\to W(R)$, défini par $\varphi(\sum_{n\geq 0} [x_n]p^n)=\sum_{n\geq 0} [x_n^p]p^n$. L'application $\delta\colon  W(R)\to W(R)$ définie par $\delta(x)\coloneqq \frac{\varphi(x)-x^p}{p}$ est une $p$-dérivation (au sens de Buium) sur $W(R)$, i.e.\ elle vérifie
   $$\delta(xy)=x^p\delta(y)+y^p\delta(x)+p\delta(x)\delta(y),\,\, \delta(x+y)=\delta(x)+\delta(y)+\frac{x^p+y^p-(x+y)^p}{p}$$ pour $x,y\in W(R)$. Cette application 
      joue un rôle crucial dans \textcite{BS}. Notons que $x'(0)^p$ est simplement la réduction modulo $p$ de $\delta(x)$. 
       \end{rema}

   \subsection{Constructions de Fontaine}
    
    Rappelons rapidement quelques constructions fondamentales dues à Fontaine.

  \begin{defi}\label{ppuis}
  Un \emph{élément $p$-puissant} d'un anneau $A$ est une suite $(a_n)_{n\geq 0}$ d'éléments de 
  $A$ telle que $a_{n+1}^p=a_n$ pour tout $n$. On note $\varprojlim_{x\mapsto x^p} A$ l'ensemble des éléments 
  $p$-puissants de $A$. 
    On dira aussi, abusivement, qu'un élément $a\in A$ est $p$-puissant s'il est \emph{muni} d'une suite 
    $(a_n)_{n\geq 0}\in \varprojlim_{x\mapsto x^p} A$ telle que 
    $a_0=a$. 
     On écrira $a^{p^{-n}}$ ou $a^{1/p^n}$ au lieu de $a_n$ et on notera $(a^{p^{-\infty}})$ l'idéal de $A$ engendré par les $a_n$. 
  Noter que l'expression \og un élément $p$-puissant $a$ de $A$\fg{} contient implicitement la donnée d'une   
  suite de racines compatibles $(a^{p^{-n}})_{n\geq 0}$ de $a$. 

    \end{defi}      
        
    Par exemple, si $R\in {\rm Perf}_{\mathbf{F}_p}$ alors l'application $a\mapsto ([a^{1/p^n}])_{n\geq 0}$ induit une bijection entre $R$ et $\varprojlim_{x\mapsto x^p} W(R)$. 
  Plus généralement, soit $B$ un anneau $p$-complet.   
  La réduction modulo $p$ induit une bijection multiplicative de $\varprojlim_{x\mapsto x^p} B$ sur le \emph{basculé} $ B^{\flat}\coloneqq \varprojlim_{x\mapsto x^p} B/p$ de $B$, la $\mathbf{F}_p$-algèbre parfaite des suites $(x_n)_{n\geq 0}$ dans $B/p$
  telles que $ x_{n+1}^p=x_n$ pour tout $n$.
   L'inverse $\iota\colon  B^{\flat}\to \varprojlim_{x\mapsto x^p} B$
    est construit comme suit: si $(x_n)_{n\geq 0}\in B^{\flat}$ et si 
      $b_n\in B$ est un relèvement quelconque de $x_n$, alors la suite $(b_{n+k}^{p^k})_{k\geq 0}$ converge $p$-adiquement dans 
      $B$ vers un élément 
     $\tilde{b}_n$ qui ne dépend pas du choix des $b_n$, et $\iota((x_n)_{n\geq 0})=(\tilde{b}_n)_{n\geq 0}$. 
     Pour tout $x\in B^{\flat}$ on a 
     $$\iota(x)=(\sharp(x), \sharp(x^{1/p}),\ldots),$$
     où  $\sharp\colon  B^{\flat}\to B$ est 
           l'application multiplicative (mais pas additive en général)
      obtenue en composant $\iota$ avec 
la projection sur la première composante
      $\varprojlim_{x\mapsto x^p} B\to B$, i.e.\ $$\sharp(x_0,x_1,\ldots)=\tilde{b}_0=\lim_{k\to\infty} b_k^{p^k}.$$
       Un des grands classiques de la théorie de Fontaine est l'application 
       $$\theta_B\colon  W(B^{\flat})\to B,\,\, \theta_B(\sum_{n\geq 0} [a_n]p^n)=\sum_{n\geq 0} a_n^{\sharp} p^n.$$
       Il s'agit de l'unique morphisme d'anneaux tel que $\theta_B([b])=b^{\sharp}$ pour tout
      $b\in B^{\flat}$.

      Toutes ces constructions sont fonctorielles en l'anneau $p$-complet $B$
      et permettent de caractériser $W(R)$ (pour $R\in {\rm Perf}_{\mathbf{F}_p}$) dans la catégorie plus grande 
      des anneaux $p$-complets, comme suit. Soit $B$ un anneau $p$-complet et soit $f\colon  W(R)\to B$ un morphisme d'anneaux. On obtient un morphisme d'anneaux 
      $f^{\flat}\colon  R\to B^{\flat}$ en passant aux basculés, plus concrètement en 
      envoyant $a$ sur la suite des réductions modulo $p$ des  
      $f([a^{1/p^n}])$, $n\geq 0$. Réciproquement, un morphisme d'anneaux $g\colon  R\to B^{\flat}$ en induit un autre 
      $f\coloneqq \theta_B\circ W(g)\colon  W(R)\to B$, tel que $f^{\flat}=g$. On a donc 
      $$f(\sum_{n\geq 0} [a_n]p^n)=\sum_{n\geq 0} g(a_n)^{\sharp}p^n.$$
    On peut résumer cette discussion comme suit:
    
    \begin{prop}\label{adj} Soit $R\in {\rm Perf}_{\mathbf{F}_p}$.
    Les applications $f\mapsto f^{\flat}$ et $g\mapsto \theta_B\circ W(g)$ induisent une bijection fonctorielle en l'anneau $p$-complet $B$
    $${\rm Hom}(W(R), B)\simeq {\rm Hom}(R, B^{\flat}).$$
    \end{prop}

    \subsection{Anneaux perfectoïdes}
       
       Soit $R\in {\rm Perf}_{\mathbf{F}_p}$. On dit que $\xi\in W(R)$ est \emph{distingué} si 
  $\xi'(0)$ est inversible dans $R$, autrement dit si $\xi=[\xi(0)]+pu$ avec $u$ inversible dans $W(R)$. Ainsi $\xi$ est un avatar d'une \og fonction holomorphe de la variable 
  $p$, biholomorphe au voisinage de $0$\fg{}. 
  
  \begin{rema}\label{derive} Soit $R\in {\rm Perf}_{\mathbf{F}_p}$ et soit $\xi\in W(R)$ un élément distingué.
  
  a) Pour tout morphisme $R\to S$ dans ${\rm Perf}_{\mathbf{F}_p}$ l'image de 
  $\xi$ reste distinguée dans $W(S)$. On la notera encore (abusivement) $\xi$. Si 
  $R$ est $\xi(0)$-complet, la relation $(1)$ montre que $u\xi$ est distingué dans $W(R)$ pour tout élément
  inversible $u$ de $W(R)$.
    
  b) L'élément $\xi$ n'est pas un diviseur de zéro dans $W(R)$: si $\xi x=0$ alors $\xi(0)x(0)=0$ et \og en prenant la dérivée en $0$\fg{}, i.e.\ en utilisant la relation $(1)$ on obtient 
  $\xi'(0)x(0)+\xi(0)x'(0)=0$. Ainsi $\xi'(0)x(0)^2=0$, et comme $\xi'(0)$ est inversible et $R$ est parfait, on a 
  $x(0)=0$. Donc $x=px_1$ pour un $x_1\in W(R)$, et $\xi x_1=0$, puis par itération  
  $x\in\cap_{n\geq 1} p^n W(R)=\{0\}$. Le même argument montre que si $a,x\in W(R)$ vérifient $a\mid \xi x$, alors 
  $x(0)^2\in (a(0)x'(0), a(0)x(0), a'(0)x(0))$. En particulier, si $p^2\mid \xi x$ alors $p\mid x$.
  \end{rema}
    
   Introduisons maintenant, suivant le point de vue prismatique de  \textcite{BS}, la classe la plus générale (à ce jour\dots{}) d'\emph{anneaux 
perfectoïdes}:

   \begin{defi}\label{perfi}
      Un anneau $A$ est dit \emph{perfectoïde} s'il est isomorphe à un anneau 
   $W(R)/(\xi)$ avec $R\in {\rm Perf}_{\mathbf{F}_p}$ et $\xi\in W(R)$ distingué. On note ${\rm Perf}$ la catégorie des anneaux perfectoïdes, les morphismes étant ceux d'anneaux.
   \end{defi}
   
   \begin{rema}\label{histoire}
Les premiers êtres du monde perfectoïde sont les corps perfectoïdes\footnote{Attention: un corps perfectoïde n'est pas la même chose qu'un anneau perfectoïde (au sens de la définition~\ref{perfi}) qui est aussi un corps. Tout n'est pas parfait dans le monde perfectoïde\dots{}}, par exemple le complété $\mathbf{C}_p$ d'une clôture algébrique de $\mathbf{Q}_p$. Puis ont vu le jour \parencite{Scholzethese} les algèbres de Banach perfectoïdes sur un tel corps (certaines étaient déjà bien présentes dans les articles de \textcite{CF} et \textcite{BC}\dots{}). \textcite{FonBourb} a introduit dans son exposé 
une classe plus large 
d'anneaux perfectoïdes. Enfin,  \textcite{BMS1} 
 introduisent la classe la plus générale 
 d'anneaux perfectoïdes, équivalente à celle introduite ci-dessus. Voir aussi la proposition~\ref{310} pour le lien avec les définitions plus anciennes (et plus restreintes).
   \end{rema}
   
   \begin{exem}\label{experf}
   \begin{enumerate}
  
  \item Il est évident que ${\rm Perf}_{\mathbf{F}_p}$ s'identifie à la sous-catégorie de 
    ${\rm Perf}$ des objets tués par $p$ (noter que $p\in W(R)$ est distingué). En particulier l'anneau nul, 
    $\mathbf{F}_p$, $\mathbf{F}_p[T^{1/p^{\infty}}]\coloneqq \varinjlim_{n} \mathbf{F}_p[T^{1/p^n}]$ sont perfectoïdes.
   
\item Soit $A=W(R)/(\xi)$, avec $R\in {\rm Perf}_{\mathbf{F}_p}$ et $\xi\in W(R)$ distingué. 
       L'image $\pi$ de $[\xi(0)^{1/p}]$ dans 
     $A$, munie de la suite de racines donnée par les images des $[\xi(0)^{1/p^{n+1}}]$ ($n\geq 0$), est un élément $p$-puissant (d\'efinition~\ref{ppuis}) de $A$. On a une égalité $(\pi^p)=(p)$ d'idéaux de $A$, puisque $\xi=\xi(0)+pu$ avec $u$ inversible. Ainsi non seulement $\mathbf{Z}_p$ n'est pas perfectoïde, il n'y a même pas de morphisme d'un anneau perfectoïde vers $\mathbf{Z}_p$. Noter que si $\xi(0)=0$ alors $\pi=0$ et $A\in {\rm Perf}_{\mathbf{F}_p}$.  
  
  \item Gardons le contexte ci-dessus. Le Frobenius $A/\pi\to A/\pi^p$ 
     s'identifie au Frobenius $R/\xi(0)^{1/p}\to R/\xi(0)$, et c'est un isomorphisme.  Ainsi pour tout anneau perfectoïde~$A$ il existe un élément 
       $p$-puissant $\pi\in A$
      tel que $(\pi^p)=(p)$ (la proposition~\ref{pcomplet} ci-dessous montre que  
      $A$~est $\pi$-complet) et tel que le Frobenius $A/\pi\to A/\pi^p$ soit un isomorphisme.
  Voir la proposition~\ref{310} pour une réciproque partielle.
        
\item Pour $R=\mathbf{F}_p[T^{1/p^{\infty}}]$ l'anneau $W(R)$ est isomorphe au complété 
   $p$-adique de $\mathbf{Z}_p[T^{1/p^{\infty}}]$ (réduire modulo $p$ pour s'en convaincre). L'élément 
   $\xi=[T]-p$ est distingué et $W(R)/(\xi)$ est isomorphe au complété 
   $p$-adique de $\mathbf{Z}_p[p^{1/p^{\infty}}]$.    
   Le même genre d'argument montre que si $A$ est un anneau perfectoïde, alors le complété $p$-adique $A\langle T^{1/p^{\infty}}\rangle$ de 
    $A[T^{1/p^{\infty}}]=\varinjlim_{n}A[T^{1/p^n}]$ l'est aussi.

\item Comme $W$ commute aux produits, on voit facilement qu'un produit quelconque d'anneaux perfectoïdes est perfectoïde. Il n'est pas vrai (cf.\ remarque~\ref{sen}), et ceci posera de sérieux soucis plus tard, qu'une limite projective (dans la catégorie des anneaux) d'anneaux perfectoïdes est perfectoïde.
 \end{enumerate}
   \end{exem}
   
   \subsection{Basculement}
   
     Le résultat suivant sera constamment utilisé par la suite.
            
    \begin{prop}\label{pcomplet}
    Tout anneau perfectoïde est $p$-complet et sa torsion $p$-primaire est tuée par $p$.
    \end{prop}
    
    \begin{proof}
    On peut supposer que l'anneau est de la forme $A=W(R)/(\xi)$ avec $R\in {\rm Perf}_{\mathbf{F}_p}$ et $\xi\in W(R)$ distingué. 
   Pour montrer que $p$ annule $A[p^{\infty}]$
            il suffit de montrer que $A[p^2]=A[p]$, autrement dit que si $x\in W(R)$ vérifie $\xi\mid p^2x$ alors 
      $\xi\mid px$, ce qui se déduit de la dernière phrase de la remarque~\ref{derive}.
                Montrons ensuite que $A$ est $p$-complet\footnote{Une manière savante est de remarquer que $A$ est $p$-complet au sens dérivé et 
       sa torsion $p$-primaire étant bornée le mot \og dérivé\fg{} est superflu\dots{}}.        
       Il suffit de voir que $(\xi)\subset W(R)$ est fermé pour la topologie $p$-adique. Mais on vient d'établir l'égalité  
 $(\xi)\cap p^2W(R)=p((\xi)\cap pW(R))$, et une récurrence immédiate donne
       $(\xi)\cap p^{n+1}W(R)=p^n((\xi)\cap pW(R))$.
    \end{proof}

     Soit $R\in {\rm Perf}_{\mathbf{F}_p}$ et soit $\xi\in W(R)$ un élément distingué. En posant $A=W(R)/(\xi)$, 
    l'anneau $A/p$ s'identifie à $R/\xi(0)$, en particulier \emph{le Frobenius est surjectif sur $A/p$}. Cela implique, par réduction modulo $p$ et le caractère 
     $p$-complet des anneaux en présence, la surjectivité de l'application de Fontaine $\theta_A\colon  W(A^{\flat})\to A$. Pour étudier son noyau on commence par décrire $A^{\flat}$.
     Puisque $R$ est parfait et $A/p\simeq R/\xi(0)$, les morphismes $x\mapsto x^{p^n}$ induisent un isomorphisme entre les systèmes projectifs
     $\{A/p\}$ et $\{R/\xi(0)^{p^n}\}$, les transitions 
     étant induites par le Frobenius pour le premier et les projections canoniques pour le second. Donc 
     $A^{\flat}=\varprojlim_{x\mapsto x^p} A/p$ est isomorphe à la $\xi(0)$-complétion $\hat{R}$ de $R$, en particulier \emph{$A^{\flat}$ est $\xi(0)$-complet}. Pour aller plus loin nous avons besoin du: 
     
     \begin{lemm}
      Le morphisme $R\to \hat{R}$ induit un isomorphisme 
     $A\simeq W(\hat{R})/(\xi)$.
     \end{lemm}
     
     \begin{proof}
     Puisque $A$ et $W(\hat{R})/(\xi)$ sont $p$-complets (proposition~\ref{pcomplet}), il suffit de voir que 
     le morphisme induit $A/p^n\to W(\hat{R})/(\xi, p^n)$ est un isomorphisme pour tout $n\geq 1$. Mais 
     $\xi-[\xi(0)]$ engendre le même idéal que $p$, donc $$A/p^n\simeq W(R)/(p^n, \xi)=W(R)/([\xi(0)]^n, \xi)\simeq (W(R)/[\xi(0)]^n)/(\xi).$$
          Comme $R/\xi(0)^n\simeq \hat{R}/\xi(0)^n$, le morphisme $W(R)/[\xi(0)]^n\to W(\hat{R})/[\xi(0)]^n$ est un isomorphisme, ce qui finit la preuve (en reprenant la chaîne d'isomorphismes ci-dessus avec~$R$ remplacé par~$\hat{R}$).
                              \end{proof}
                              
                           L'adjonction entre basculement et vecteurs de Witt (proposition~\ref{adj}) 
     montre que le morphisme 
          $\theta_A\colon  W(A^{\flat})\to A$ s'identifie, via les isomorphismes $A^{\flat}\simeq \hat{R}$ et $A\simeq W(\hat{R})/(\xi)$, 
           à la projection canonique 
          $W(\hat{R})\to W(\hat{R})/(\xi)$, en particulier son noyau est engendré par~$\xi$. On obtient ainsi l'isomorphisme 
          $$\theta_A\colon  W(A^{\flat})/(\xi)\simeq A, \eqno(2)$$
         d'où le premier résultat fondamental de la théorie:
              
          \begin{theo}\label{tilt}
          Si $A$ est un anneau perfectoïde et si $\xi\in W(A^{\flat})$ est distingué et engendre $\ker(\theta_A)$, alors 
           la catégorie des 
          $A$-algèbres perfectoïdes est équivalente, via $B\mapsto B^{\flat}$ et $R\mapsto W(R)/(\xi)$, à celle des 
          $A^{\flat}$-algèbres parfaites $\xi(0)$-complètes. 
          \end{theo}

   \begin{proof}
   Il suffit de montrer que $\theta_B\colon  W(B^{\flat})/(\xi)\simeq B$ et que 
   $B^{\flat}$ est $\xi(0)$-complète
    pour toute $A$-algèbre perfectoïde $B$.
   On vient de voir que  
          $\theta_B$ est surjective et qu'il existe un élément distingué $\xi_1\in W(B^{\flat})$ tel que $\ker(\theta_B)=(\xi_1)$, $B^{\flat}$ étant $\xi_1(0)$-complet. Comme $B$ est une $A$-algèbre, la fonctorialité de la construction de l'application~$\theta$ montre que 
          $\xi\in \ker(\theta_B)$, donc $\xi=\xi_1v$ pour un $v\in W(B^{\flat})$. En \og dérivant\fg{} on obtient (cf.\ relation $(1)$)
          $\xi'(0)=\xi_1'(0)v(0)+\xi_1(0)v'(0)$, donc $\xi_1'(0)v(0)=\xi'(0)-\xi_1(0)v'(0)$ est inversible (puisque $\xi'(0)$ l'est et que 
          $B^{\flat}$ est $\xi_1(0)$-complet). Comme $\xi_1'(0)$ est inversible, $v(0)$ l'est tout autant et donc il en est de même de $v$, et 
          $(\xi)=(\xi_1)$. En particulier $B^{\flat}$ est aussi $\xi(0)$-complète, ce qui finit la preuve.
   \end{proof}
                                          
                    La propriété de stabilité suivante jouera un rôle crucial par la suite.
                 
              \begin{coro}\label{tensor}
              Soit $A$ un anneau perfectoïde.
 Si $B$ et $C$ sont des $A$-algèbres perfectoïdes, alors le complété $p$-adique 
      $B\widehat{\otimes}_A C$ de $B\otimes_A C$ est un anneau perfectoïde. Plus précisément, si 
      $\xi$ engendre le noyau de $\theta_A\colon  W(A^{\flat})\to A$, alors
      $$B\widehat{\otimes}_A C\simeq W(B^{\flat}\otimes_{A^{\flat}} C^{\flat})/(\xi).$$
              \end{coro}
                    
                    \begin{proof}
                    La $A^{\flat}$-algèbre
    $R\coloneqq B^{\flat}\otimes_{A^{\flat}} C^{\flat}$ est parfaite puisque 
    $A^{\flat}, B^{\flat}, C^{\flat}$ le sont. Comme $\xi$ reste distingué dans $W(R)$, 
    l'anneau $T=W(R)/(\xi)$ est perfectoïde, donc $p$-complet. Les morphismes 
    $B^{\flat}\to R$ et $C^{\flat}\to R$ combinés avec les isomorphismes 
    canoniques (th\'eor\`eme~\ref{tilt}) $W(A^{\flat})/(\xi)\simeq A,\, W(B^{\flat})/(\xi)\simeq B,\,W(C^{\flat})/(\xi)\simeq C$
     induisent un morphisme $B\otimes_A C\to T$, qui se prolonge en un morphisme $B\widehat{\otimes}_A C\to T$ puisque $T$ est $p$-complet. Pour montrer que c'est un isomorphisme il suffit de voir qu'il induit une bijection ${\rm Hom}(T, S)\to {\rm Hom}(B\widehat{\otimes}_A C, S)$ pour tout anneau $p$-complet $S$. Comme 
     \begin{multline*}
       {\rm Hom}(B\widehat{\otimes}_A C, S)={\rm Hom}(B\otimes_A C, S)=\\{\rm Hom}\Bigl(\frac{W(B^{\flat})}{(\xi)}, S\Bigr)\times_{{\rm Hom}\bigl(\frac{W(A^{\flat})}{(\xi)}, S\bigr)}
      {\rm Hom}\Bigl(\frac{W(C^{\flat})}{(\xi)}, S\Bigr),
     \end{multline*}
      le résultat suit formellement de l'adjonction fournie par la proposition~\ref{adj}. 
      \end{proof}

       \subsection{Anneaux perfectoïdes et Frobenius}\label{Frob}

       Le résultat fondamental suivant 
      (lemmes 3.9 et 3.10 de \textcite{BMS1})
      fournit 
      le critère le plus simple pour tester le caractère perfectoïde d'un anneau, sous des hypothèses faibles. Cela permet de fabriquer des tas d'anneaux
           perfectoïdes sans être obligé de fournir une présentation $W(R)/(\xi)$, et fait aussi le lien avec les définitions plus anciennes
           (cf.\ remarque~\ref{histoire}). 


      \begin{prop}\label{310}
      Soit $A$ un anneau contenant un non diviseur de zéro $\pi\in A$ tel que 
      $\pi^p\mid p$, $A$ est $\pi$-complet et le Frobenius induit un isomorphisme $A/\pi\simeq A/\pi^p$. Alors 
      $A$ est perfectoïde. 
      \end{prop}
      
      \begin{proof}       Notons que $A$ est $p$-complet car $\pi^p\mid p$ et $A$ est $\pi$-complet. Pour montrer la surjectivité 
      de $\theta_A\colon  W(A^{\flat})\to A$ il suffit donc de montrer celle du Frobenius sur 
      $A/p$. Pour tout $x\in A$ la surjectivité du Frobenius modulo $\pi^p$ et la $\pi$-complétude de $A$ permettent d'écrire 
      $x=x_0^p+\pi^p x_1^p+\pi^{2p} x_2^p+\cdots$ pour certains $x_n\in A$, et alors $x\equiv (x_0+\pi x_1+\pi^2 x_2+\cdots)^p\pmod p$.
      
           Ensuite, on construit un élément distingué dans le noyau de $\theta\coloneqq \theta_{A}$. 
   La surjectivité du Frobenius $A/\pi\to A/\pi^p$ et la $\pi$-complétude de $A$ fournissent\footnote{Prendre une suite $(a_n)_{n\geq 0}$ dans $A$ telle que 
      $a_0=\pi$ et $a_{n+1}^p\equiv a_n\pmod {\pi^p}$ et poser $u=\lim_{n\to\infty} a_n^{p^n}$, la suite $(a_n^{p^n})_{n\geq 0}$ étant de Cauchy 
   pour la topologie $\pi$-adique.} un élément $p$-puissant (d\'efinition~\ref{ppuis}) $u\in A$ tel que $\pi\equiv u\pmod {\pi^p}$. Comme $A$ est $\pi$-complet, 
      $\pi$ et $u$ engendrent le même idéal de $A$, donc on peut remplacer $\pi$ par $u$ et supposer que $\pi=f^{\sharp}$ pour un $f\in A^{\flat}$. 
    Comme $\theta$ est surjective et $\pi^p\mid p$, il existe 
    $x\in W(A^{\flat})$ tel que $p=\pi^p \theta(-x)=\theta(-x[f]^p)$. En posant $\xi=p+x[f]^p$, on a $\xi'(0)=1+x'(0)f^p$. 
    On vérifie ensuite que $A^{\flat}$ est $f$-complet, donc $\xi$ est distingué, et que 
    $\ker(\theta)=(\xi)$. 
         \end{proof}
      
      \begin{rema} \label{sen} On déduit facilement de la proposition~\ref{310} les résultats suivants: 
      
      \begin{enumerate}
            
      \item Soit $A$ un anneau intègre, sans $p$-torsion, $p$-adiquement séparé et tel que $p\in {\rm Rad}(A)$. Alors le complété $p$-adique 
      $\widehat{A^+}$ de la clôture intégrale absolue $A^+$ de $A$ est perfectoïde. 
      

      \item L'anneau $\mathbf{Z}_p$ n'est pas perfectoïde (exemple~\ref{experf}), mais le théorème d'Ax--Sen--Tate l'exhibe 
    comme l'anneau des invariants de $\widehat{\mathbf{Z}_p^+}$ sous l'action du groupe de Galois absolu de 
      $\mathbf{Q}_p$. On voit donc qu'une limite projective d'anneaux perfectoïdes ne l'est plus forcément (un autre exemple 
      est fourni par les sous-anneaux $\widehat{\mathbf{Z}_p[\mu_{p^{\infty}}]}$ et $\widehat{\mathbf{Z}_p[p^{1/p^{\infty}}]}$ de $\widehat{\mathbf{Z}_p^+}$, qui sont  
      perfectoïdes, mais dont l'intersection, \'egale \`a $\mathbf{Z}_p$, ne l'est pas). 
  \end{enumerate}
      \end{rema}
      
      \begin{rema}
      Le lien entre anneaux perfectoïdes et Frobenius est rendu plus clair par la th\'eorie prismatique de \textcite{BS}. Soit $A$ un anneau 
      muni d'une $p$-d\'erivation $\delta$ (cf.\ remarque~\ref{pder}), et soit $\varphi(x)=x^p+p\delta(x)$ le rel\`evement du Frobenius induit par 
      $\delta$. Si $I$ est un id\'eal de $A$, on dit que la paire $(A,I)$ est un 
       \emph{prisme} si $I$ d\'efinit un diviseur de Cartier dans ${\rm Spec}(A)$, $A$ est $(p,I)$-complet (au sens d\'eriv\'e) et 
      $p\in I+\varphi(I)A$. Bhatt et Scholze montrent (th\'eor\`eme $3.10$ de loc.cit.) que la cat\'egorie des prismes parfaits (i.e.\ pour lesquels $\varphi\colon A\to A$ est un isomorphisme) est \'equivalente \`a celle des anneaux perfectoïdes, via les foncteurs $(A,I)\mapsto A/I$ et $R\mapsto (W(R^{\flat}), \ker(\theta_R))$.
      \end{rema}

      \subsection{Anneaux perfectoïdes et $p$-clôtures intégrales}\label{pclot}
      
     Les articles 
de \textcite{Rob} et d'\textcite{AndreIHES}  mettent en avant une relation fondamentale entre la notion de $p$-clôture intégrale (ou simplement $p$-clôture, pour raccourcir)  
et celle d'anneau perfectoïde.

       \begin{defi}\label{pclos}
Soit $A$ un sous-anneau d'un anneau $B$. On dit que $A$ est 
       \emph{$p$-clos dans $B$} si 
pour tout $x\in B$ vérifiant $x^p\in A$ on a $x\in A$. Il existe un plus petit sous-anneau 
$p$-clos de $B$ contenant $A$, que l'on appelle la \emph{$p$-clôture de $A$ dans $B$}. 
\end{defi}

 Les observations suivantes sont dues à  \textcite{Rob}.
  On fixe dans la suite de ce paragraphe un anneau $A$ muni d'un non diviseur de zéro $\pi$ tel que
  $\pi^p\mid p$. On note $\varphi\colon  A/\pi\to A/\pi^p$ le Frobenius. 

\begin{prop}\label{Andrepclos}
a) La 
$p$-clôture de $A$ dans $A[\frac{1}{\pi}]$ est 
  $$R=\Bigl\{x\in A\Bigl[\frac{1}{\pi}\Bigr]\,|\, \exists n\geq 0, x^{p^n}\in A\Bigr\}.$$

  b)  L'anneau $A$ est $p$-clos dans $A[\frac{1}{\pi}]$ si et seulement si $\varphi\colon  A/\pi\to A/\pi^p$ est injectif.

  \end{prop}

\begin{proof} a) La seule difficulté est de montrer que 
$R$ est bien un sous-anneau de $A[\frac{1}{\pi}]$, et plus précisément qu'il est stable par addition. 
Prenons $x,y\in R$ et $n,k\geq 1$ tels que 
  $x^{p^n}, y^{p^n},\pi^kx, \pi^k y$ soient tous dans $A$, et montrons qu'en posant 
  $N=2p^nk+n$ on a $(x+y)^{p^N}\in A$, ce qui permettra de conclure que $x+y\in R$. Il suffit de voir que 
  $\binom{p^N}{i}x^{p^N-i}y^i\in A$ pour tout $0\leq i\leq p^N$. Cela est évident si $p^n\mid i$, supposons donc que ce n'est pas le cas. Puisque $v_p(i)<n$ on a $p^{N-n}\mid \binom{p^N}{i}$, donc $\pi^{2p^n k}\mid \binom{p^N}{i}$. Comme $x^{p^n}, y^{p^n}\in A$, il suffit de voir que 
  $\pi^{2p^n k}x^uy^v\in A$ pour $0\leq u,v<p^n$, ce qui est clair puisque $\pi^k x, \pi^k y\in A$. 
  
  b)  Il est clair que $\varphi$ est injectif si $A$ est $p$-clos dans $A[\frac{1}{\pi}]$. Dans l'autre sens, soit
  $x\in A[\frac{1}{\pi}]$ tel que $x^p\in A$ et soit $n\geq 1$ tel que $\pi^n x\in A$. Alors $(\pi^n x)^p\in \pi^p A$, donc 
  $\pi^n x\in \pi A$ et $\pi^{n-1}x\in A$. En itérant on obtient $x\in A$.
  \end{proof}

   Nous aurons besoin du résultat suivant dans la preuve du théorème~\ref{nopresque}. 
  
\begin{prop}\label{p-closed} Supposons que le Frobenius $A/\pi\to A/\pi^p$ est un isomorphisme, que $g\in A$ et~$\pi$ sont $p$-puissants (déf.~\ref{ppuis}) et que $g$~est non diviseur de zéro modulo~$\pi$. Pour tout $n\geq 1$ 
la $p$-clôture de $A[\frac{g}{\pi^n}]$ dans $A[\frac{1}{\pi}]$ est 
$A[\frac{g^{1/p^j}}{\pi^{n/p^j}}\,, j\geq 0]$, et son 
 complété $\pi$-adique est perfectoïde.
\end{prop}

 \begin{proof}  L'algèbre $C\coloneqq A[\frac{g^{1/p^j}}{\pi^{n/p^j}}\,, j\geq 0]$ est la réunion croissante de ses sous-algèbres $A[\frac{g^{1/p^j}}{\pi^{n/p^j}}]$, et clairement contenue dans la $p$-clôture de $A[\frac{g}{\pi^n}]$ dans $A[\frac{1}{\pi}]$. 
 
 Notons $\varphi_R\colon  R/\pi\to R/\pi^p$ le Frobenius
 d'une $A$-algèbre $R$. Il suffit de montrer que 
 $\varphi_C$ est bijectif: le complété $\pi$-adique de $C$ sera perfectoïde par la proposition~\ref{310}, et 
  $C$ sera $p$-close dans $C[\frac{1}{\pi}]=A[\frac{1}{\pi}]$ (proposition~\ref{Andrepclos}), donc 
égale à la $p$-clôture de $A[\frac{g}{\pi^n}]$ dans $A[\frac{1}{\pi}]$. 

  Posons $u_j=\pi^{n/p^j}X^{1/p^j}-g^{1/p^j}\in A[X^{1/p^j}]$. 
   Puisque $g$ n'est pas un diviseur de zéro modulo $\pi$, la remarque~\ref{nondiv} ci-dessous fournit des isomorphismes  $A[\frac{g^{1/p^j}}{\pi^{n/p^j}}]\simeq A[X^{1/p^j}]/(u_j)$ compatibles avec la variation de $j$, d'où un isomorphisme
 $C\simeq B/I$ avec $B=A[X^{1/p^{\infty}}]$ et $I=(u_0, u_1,\dots{})$. Puisque 
 $\varphi_{A}$ est bijectif, il en est de même de $\varphi_{B}$. En utilisant les congruences 
   $u_{j+1}^p\equiv u_j\pmod{\pi^p B}$, on en déduit facilement que $\varphi_{B/I}$ est un isomorphisme. 
 \end{proof}

\begin{rema}\label{nondiv}
Soient $R$ un anneau et $r\in R$ non diviseur de zéro. Si $s\in R$ est non diviseur de zéro modulo 
$r$, alors le morphisme naturel $R[X]/(rX-s)\to R[\frac{s}{r}]\subset R[\frac{1}{r}]$ est un isomorphisme (exercice!).
\end{rema}

\subsection{Deux résultats techniques importants}
          
   Nous travaillerons souvent avec des anneaux sans $p$-torsion, et le caractère perfectoïde d'un anneau a le bon goût de 
   se propager à son quotient maximal sans $p$-torsion:
      
      \begin{prop}\label{perfide} 
      Si $A\simeq W(R)/(\xi)$ (avec $R\in {\rm Perf}_{\mathbf{F}_p}$ et $\xi$ distingué) est un 
   anneau perfectoïde alors $A/A[p^{\infty}]\simeq W(R/{\rm Ann}(\xi(0)))/\xi$ l'est aussi.
          \end{prop}

\begin{proof}
 Rappelons que $A[p^{\infty}]=A[p]$ (proposition~\ref{pcomplet}). 
 Écrivons $\xi=[\xi(0)]+pu$, avec $u$ inversible dans $W(R)$ et
 notons que   
   $S\coloneqq R/{\rm Ann}(\xi(0))$ est parfait (en effet, si $x^p\xi(0)=0$ alors $(x\xi(0))^p=0$, donc $x\xi(0)=0$). 
 La projection canonique $f\colon  R\to S$ induit un morphisme $f\colon  W(R)\to W(S)$. Si 
 $a\in A[p]$ se relève en $x\in W(R)$, il existe $y\in W(R)$ tel que $px=\xi y$. On a $f(y)\in pW(S)$ puisque 
 $\xi(0)y(0)=0$ (donc $f(y(0))=0$). On en déduit que $f(x)\in (\xi)W(S)$ et que le morphisme $A\to W(S)/(\xi)$ se factorise en un morphisme $A/A[p]\to W(S)/(\xi)$, clairement surjectif.
  Il est aussi injectif: si $a\in A$ se relève en $x\in W(R)$ et a une image nulle dans $W(S)/(\xi)$, alors $f(x)=\xi y$ pour un $y\in W(S)$. Soit 
      $z\in W(R)$ un relèvement de $y$, alors $[\xi(0)](x-\xi z)=0$, puis $(\xi-pu)(x-\xi z)=0$ et $\xi\mid px$, donc $a\in A[p]$.                   
       \end{proof}
       
       Le résultat suivant sera systématiquement utilisé par la suite. Sa preuve est assez astucieuse. 
             
      \begin{prop}\label{reduced}
      Tout anneau perfectoïde $A$ est réduit et, pour tout élément $p$-puissant (déf.~\ref{ppuis}) $a\in A$, la torsion $a$-primaire $A[a^{\infty}]$ de $A$ est tuée par 
$(a^{p^{-\infty}})$.
      \end{prop}
      
      \begin{proof}
Le second point est une conséquence formelle du premier: si $a^n x=0$ pour un $x\in A$, alors $(a^{n/p}x)^p=0$, donc $a^{n/p}x=0$ puisque $A$ est réduit. Soit $\pi\in A$ un élément $p$-puissant tel que $(\pi^p)=(p)$ (exemple~\ref{experf}).
Supposons d'abord que $A$ est sans $p$-torsion. Si 
      $a\in A$ vérifie $a^p=0$, l'isomorphisme $A/\pi\simeq A/\pi^p$ montre que $a=\pi b$ pour un
      $b\in A$. Comme $(\pi^p)=(p)$ et $A$ est sans $p$-torsion, on a $b^p=0$. Par itération cela force   
      $a\in \cap_{n} \pi^n A=\cap_{n} p^n A=\{0\}$, ce qui permet de conclure. 
      
      Dans le cas général, par la proposition~\ref{perfide} et le caractère réduit (même parfait!) de $A/(\pi^{p^{-\infty}})$, il suffit de prouver 
      l'injectivité  
   du morphisme naturel $A\to A/A[p]\times A/(\pi^{p^{-\infty}})$. Soit $a\in A[p]\cap (\pi^{p^{-\infty}})$ et soit $x\in W(R)$ un représentant de $a$. Comme 
   $a\in (\pi^{p^{-\infty}})$, on a 
   $x(0)\in (\xi(0)^{p^{-\infty}})$. Soit $R=A^{\flat}$, donc $A\simeq W(R)/(\xi)$. Puisque $pa=0$, il existe $y\in W(R)$ tel que $px=\xi y$. On a donc 
      $\xi(0)y(0)=0$ et $\xi'(0)y(0)+\xi(0)y'(0)=x(0)$, donc 
          $y(0)\in (\xi(0)^{p^{-\infty}})$. Mais $\xi(0)\in {\rm Ann}(y(0))$ et $R$ est parfait, donc $(\xi(0)^{p^{-\infty}})\subset {\rm Ann}(y(0))$, ce qui force $y(0)^2=0$, puis $y(0)=0$. On a donc $\xi\mid x$ et $a=0$. 
      \end{proof}

           

       \section{Algèbres perfectoïdes: aspects analytiques}\label{perf}
       
       On fixe dans cette section (qui emprunte les notations de la précédente) un nombre premier~$p$ (d'où une notion d'anneau perfectoïde), ainsi qu'un corps~$K$ muni 
  d'une valeur absolue non archimédienne $\lvert\cdot\rvert\colon  K\to \mathbf{R}_{\geq 0}$, dont
 la boule unité~$K^0$ 
  est un anneau perfectoïde. Comme $K^0$~est $p$-complet (proposition~\ref{pcomplet}), on a 
  $p\in {\rm Rad}(K^0)$ et donc $|p|<1$. On fixe 
   un élément $\pi\in K^0$ comme suit:
   
   $\bullet$ si ${\rm car}(K)=p$ on prend n'importe quel élément non nul $\pi\in K^0$ tel que 
  $|\pi|<1$.
  
  $\bullet$ sinon, on choisit 
  un générateur distingué $\xi\in W((K^0)^{\flat})$ de $\ker(\theta_{K^0})$ et on note 
  $$\pi=\theta_{K^0}([\xi(0)^{1/p}])=(\xi(0)^{1/p})^{\sharp}\in K^0.$$ Ainsi $\pi^p K^0=pK^0$ (exemple~\ref{experf}), donc $|\pi|<1$. L'élément~$\pi$ muni du système de racines $(\pi^{1/p^n}\coloneqq 
  (\xi(0)^{1/p^{n+1}})^{\sharp})_{n\geq 0}$ est $p$-puissant (d\'efinition~\ref{ppuis}). 
  
  Dans les deux cas,
  $K^0/\pi\simeq K^0/\pi^p$ via le Frobenius, $\pi$ est $p$-puissant et $|\pi|<1$.

        Si
    $A$ est une $K^0$-algèbre plate (i.e.\ sans $\pi$-torsion) on note
    $$A_*=\pi^{-1/p^{\infty}}A\coloneqq \{f\in A[\frac{1}{\pi}]\,|\,\, \pi^{1/p^n}f\in A\,\, \text{pour tout}\,\, n\geq 0\},$$
    une sous $K^0$-algèbre de $A[\frac{1}{\pi}]$ contenant $A$ et contenue dans $\pi^{-1}A$. 
    
    \begin{rema}\label{almost} On voit facilement que  
    $(A_*)_*=A_*$ et que si $A$ est $\pi$-complète (respectivement $p$-close (d\'efinition~\ref{pclos}) dans $A[\frac{1}{\pi}]$), alors $A_*$ l'est aussi. Si $(A_i)_{i\in I}$ est un système projectif de 
    $K^0$-algèbres plates, alors $(\varprojlim_{i\in I} A_i)_*\simeq \varprojlim_{i\in I} (A_i)_*$.
\end{rema}

\subsection{Algèbres de Banach uniformes}\label{algunif}

 On note ${\rm Ban}_K$ la catégorie des \emph{$K$-algèbres de Banach}, i.e.\ celle des $K$-algèbres $A$ munies d'une norme ultramétrique
 $\lvert\cdot\rvert \colon  A\to \mathbf{R}_{\geq 0}$ sur le $K$-espace vectoriel $A$ telle que $|ab|\leq |a||b|$ pour tous $a,b\in A$ (et $|1|=1$ si $A\ne 0$) et qui fait de $A$ un 
 espace métrique complet pour la distance induite. Les morphismes dans ${\rm Ban}_K$ sont les applications lipschitziennes qui sont aussi des morphismes de $K$-algèbres.
 
   Berkovich\footnote{On pourrait (ou devrait\dots{}) remplacer les espaces de Berkovich par ceux de Huber dans ce qui suit. Nous avons fait ce choix pour épargner au lecteur les multiples définitions intervenant dans la théorie de Huber, qui ne joueront pas de rôle sérieux dans cet exposé.} a construit un foncteur de ${\rm Ban}_K$ vers les espaces topologiques compacts, en associant à 
   $(A, \lvert\cdot\rvert )\in {\rm Ban}_K$ son \emph{spectre de Berkovich} $\mathcal{M}(A)$, i.e.\ l'ensemble des semi-normes multiplicatives\footnote{On demande donc que $x(f+g)\leq \max(x(f), x(g)), x(fg)=x(f)x(g)$ et $x(1)=1$ si $A\ne 0$.}
      $x\colon  A\to \mathbf{R}_{\geq 0}$ telles que $x(f)\leq |f|$ pour tout $f\in A$, muni de la topologie la plus faible rendant continues les évaluations $x\mapsto x(f)$ pour
   $f\in A$. On écrira $|f(x)|$ au lieu de 
   $x(f)$.    
   
   Soit $(A, \lvert\cdot\rvert )\in {\rm Ban}_K$. 
     Le sous-ensemble 
   $$A^0=\{f\in A\, |\,\, \sup_{n\geq 1} |f^n|<\infty\}$$
est une sous $K^0$-algèbre de $A$, contenant la boule unité de $A$, en particulier $A^0[\frac{1}{\pi}]=A$. On écrira $A_*^0$ au lieu de $(A^0)_*$. Si l'on note $$|f|_{\rm sp}=\lim_{n\to\infty}|f^n|^{1/n}=\inf_{n\geq 1} |f^n|^{1/n}$$ 
   la \emph{semi-norme spectrale} d'un élément $f\in A$, on dispose des formules fondamentales\footnote{   La première est la \emph{formule du rayon spectral} de Berkovich; la seconde 
    découle facilement des inclusions évidentes 
    $\{f\in A\, |\, |f|_{\rm sp}< 1\}\subset A^0\subset \{f\in A\, |\, |f|_{\rm sp}\leq 1\}$.
    }
     $$|f|_{\rm sp}=\max_{x\in \mathcal{M}(A)} |f(x)|\,\,\text{et}\,\, A_*^0=\{f\in A\, |\, |f|_{\rm sp}\leq 1\}.$$

   On dit que 
   $(A, \lvert\cdot\rvert )\in {\rm Ban}_K$ est \emph{uniforme} si $A^0$ est une partie bornée de $A$. Cela arrive si et seulement si $\lvert\cdot\rvert $ est équivalente à $\lvert\cdot\rvert _{\rm sp}$ (utiliser les inclusions $\{f\in A\, |\, |f|_{\rm sp}< 1\}\subset A^0\subset \{f\in A\, |\, |f|_{\rm sp}\leq 1\}$), auquel cas\footnote{Pour la première égalité noter que si
   $f\in A_*^0$ alors $f^{p^n}$ reste dans le borné $p^{-1}A^0$ pour tout $n$.} 
   $$A^0=A_*^0=\{f\in A\,|\, |f|_{\rm sp}\leq 1\}=\{f\in A\,|\, |f(x)|\leq 1,\,\, \text{pour tout}\,\, x\in \mathcal{M}(A)\} \eqno(3).$$ 
    On note 
   ${\rm Ban}^u_K$ la sous-catégorie pleine de ${\rm Ban}_K$ des $K$-algèbres de Banach uniformes.

    On dit que $A$ est \emph{spectrale} si $|f|=|f|_{\rm sp}$ pour tout $f\in A$, ce qui arrive si et seulement si $A^0$ est la boule unité de $A$, auquel cas 
    $A$ est uniforme.    
 
     \subsection{Dictionnaire analyse-algèbre, interprétation catégorique}\label{dico}
     
     Le lecteur trouvera un dictionnaire exhaustif analyse-algèbre dans le paragraphe $2.3$ de l'article
     d'\textcite{AndreIHES}.
     
     \begin{prop}\label{unif}

     Soit $(A, \lvert\cdot\rvert )\in {\rm Ban}_K^u$. La $K^0$-algèbre 
     $A^0$ est plate, $\pi$-complète et $p$-close (déf.~\ref{pclos}) dans 
     $A=A^0[\frac{1}{\pi}]$. Ainsi le Frobenius $A^0/\pi\to A^0/\pi^p$ est injectif.
          \end{prop}
     
     \begin{proof}
     Le seul point non évident est le fait que $A^0$ est $p$-close, mais il suffit de contempler l'\'egalit\'e $(3)$ pour s'en convaincre.
     \end{proof}
     
       Soit $\mathcal{C}$ la catégorie des $K^0$-algèbres plates et $\pi$-compl\`etes 
     $S$, telles que $S$ soit $p$-close dans $S[\frac{1}{\pi}]$ (d\'efinition~\ref{pclos}) et $S_*=S$. Pour 
     $S\in\mathcal{C}$ 
     et $f\in A\coloneqq S[\frac{1}{\pi}]$ on pose
       $$|f|\coloneqq \inf\{|k|\,|\,\, k\in K\setminus \{0\},\, \frac{f}{k}\in S\}.$$
       Alors\footnote{Le fait que $S$ est plate, $\pi$-complète et vérifie 
     $S=S_*$ est suffisant pour montrer que $\lvert\cdot\rvert $ est une norme de $K$-algèbre de Banach sur 
     $A$, de boule unité $S$.
     On déduit de la relation $S=\{f\in A|\, f^p\in S\}$ que $|f^p|=|f|^p$, puis que 
     $|f|=\lim_{n\to\infty}|f^{p^n}|^{1/p^n}=|f|_{\rm sp}$ pour tout $f\in A$, donc $A$ est bien spectrale.}
     $(A, \lvert\cdot\rvert )$ devient 
      une 
      $K$-algèbre de Banach spectrale (donc uniforme) telle que $A^0=S$.
     On obtient ainsi:
     
     \begin{prop}
     Le foncteur $(A, \lvert\cdot\rvert )\mapsto A^0$ induit une équivalence de catégories 
     $${\rm Ban}_K^u\simeq \mathcal{C}.$$ 
     \end{prop}
     
     La remarque~\ref{almost} montre que 
               $\mathcal{C}$ possède toutes les petites limites (qui se calculent comme dans la catégorie des $K^0$-algèbres) et toutes les colimites filtrantes: si
     $(S_i)_{i\in I}$ est un système inductif filtrant dans $\mathcal{C}$, sa colimite dans $\mathcal{C}$ est 
     $S_*$, où $S$ est le complété $\pi$-adique de la colimite dans la catégorie des
     anneaux $\varinjlim_{i\in I} S_i$.
     
     Ainsi la catégorie  ${\rm Ban}^u_K$ possède toutes les petites limites, et aussi toutes les colimites filtrantes. 
     L'inclusion de 
  ${\rm Ban}_K^u$ dans ${\rm Ban}_K$ admet un adjoint à gauche $A\mapsto A^u$, l'uniformis\'e 
  $A^u$ de $A$ étant le séparé complété de 
  $A$ pour la semi-norme spectrale. Comme 
  ${\rm Ban}_K$ possède des pushouts (donnés par le produit tensoriel complété), 
  il en est de même de ${\rm Ban}_K^u$: si $A\to B$ et $A\to C$ sont des morphismes dans 
  ${\rm Ban}^u_K$, le pushout correspondant est $B\widehat{\otimes}^u_{A} C\coloneqq (B\widehat{\otimes}_A C)^u$. Le calcul de 
  $(B\widehat{\otimes}^u_{A} C)^0$ n'est pas aisé, mais on verra qu'il est (presque) faisable dans le monde parfait des algèbres de Banach 
  perfectoïdes.
  
  \subsection{Algèbres de Banach perfectoïdes}\label{An}
  
   Il convient de garder en tête qu'une \emph{$K$-algèbre de Banach perfectoïde} n'est \emph{pas} la même chose qu'une $K$-algèbre 
    qui est un anneau perfectoïde: 
  
  \begin{defi} Une \emph{$K$-algèbre de Banach perfectoïde} est un objet 
  $(A,\lvert\cdot\rvert )$ de ${\rm Ban}_K^u$ tel que $A^0$ soit un anneau perfectoïde. On note 
  ${\rm Perf}_K$ la sous-catégorie pleine de ${\rm Ban}_K^u$ dont les objets sont les 
  $K$-algèbres de Banach perfectoïdes.
  \end{defi}
  
  Vérifions la compatibilité avec la définition originelle dans \textcite{Scholzethese}:

  \begin{prop}\label{banperf}
 Une $K^0$-algèbre plate et $\pi$-complète $B$ 
  est un anneau perfectoïde si et seulement si le Frobenius 
  $B/\pi\to B/\pi^p$ est un isomorphisme, auquel cas $B_*$ est un anneau perfectoïde et $B$ est $p$-close dans 
  $B[\frac{1}{\pi}]$ (déf.~\ref{pclos}).
   En particulier $A\in {\rm Ban}_K^u$ est dans ${\rm Perf}_K$ si et seulement si le Frobenius 
  $A^0/\pi\to A^0/\pi^p$ est un isomorphisme ou, de manière équivalente\footnote{Puisque $A$ est uniforme, $A^0$ est $p$-close dans $A$, donc le Frobenius 
  $A^0/\pi\to A^0/\pi^p$ est injectif.}, surjectif. 

     \end{prop}
  
  \begin{proof} On peut supposer que 
  ${\rm car}(K)=0$. 
  Si $B$ est un anneau perfectoïde, le théorème~\ref{tilt} montre que $B\simeq W(B^{\flat})/(\xi)$. Puisque $\pi=\theta_{K^0}([\xi(0)^{1/p}])$, le Frobenius $B/\pi\to B/\pi^p$ est un isomorphisme et donc (proposition~\ref{Andrepclos}) 
  $B$ est $p$-close dans $B[\frac{1}{\pi}]$. La réciproque découle de la proposition~\ref{310}.
  Supposons que $B$ est un anneau perfectoïde. Comme $B_*$ est plate et $\pi$-complète, il suffit de voir (d'après ce que l'on vient de faire) que 
  le Frobenius $B_*/\pi\to B_*/\pi^p$ est un isomorphisme. Mais $B$ étant $p$-close dans $B[\frac{1}{\pi}]$, 
  $B_*$ l'est dans $B_*[\frac{1}{\pi}]$, donc $B_*/\pi\to B_*/\pi^p$ est injectif. Soit $x\in B_*$, alors $\pi x\in B$ donc $\pi x=a^p+\pi^p b$ avec 
  $a,b\in A$. Puisque $x=(\pi^{-1/p} a)^p+\pi^{p-1}b$ et $B_*$ est $p$-close dans $B[\frac{1}{\pi}]$, on a $\pi^{-1/p}a\in B_*$, donc 
  $B_*=B_*^p+\pi B_*$. Par itération et en utilisant le fait que $\pi^p\mid p$, on obtient $B_*=B_*^p+\pi^p B_*$, ce qui permet de conclure.
  \end{proof}

  \begin{rema}\label{perfcat}
  
  \begin{enumerate}
  
  \item Le caractère perfectoïde de $B_*$ n'est pas totalement gratuit: si 
  $A\in  {\rm Perf}_K$ et si $g\in A^0$ n'est pas un diviseur de zéro, il n'est pas connu (cf.\ la question~3.5.1 d'\textcite{AndreIHES}) si 
   l'algèbre $g^{-1/p^{\infty}}A=\cap_{n\geq 1} g^{-1/p^n}A\subset A[\frac{1}{g}]$ est encore dans ${\rm Perf}_K$.
   Cette algèbre interviendra naturellement par la suite (cf.\ théorème~\ref{Riem} par exemple).
  
\item Supposons que ${\rm car}(K)=0$ et posons $K^{\flat}\coloneqq (K^0)^{\flat}[\frac{1}{\xi(0)}]$, un corps perfectoïde de caractéristique $p$. 
  En combinant la proposition ci-dessus et le théorème~\ref{tilt}, on obtient facilement l'équivalence de catégories 
  ${\rm Perf}_K\simeq {\rm Perf}_{K^{\flat}}$ de  \textcite{Scholzethese}, induite par les foncteurs $A\mapsto A^{\flat}\coloneqq (A^0)^{\flat}[\frac{1}{\xi(0)}]$ et $R\mapsto W(R)[\frac{1}{\pi}]/(\xi)$ (cette approche est celle de  \textcite{FonBourb} et du livre de  \textcite{KL}).
  
\item La catégorie ${\rm Perf}_K$ possède toutes les colimites filtrantes, qui sont les mêmes que celles dans 
     ${\rm Ban}_K^u$: si $(B_i)_{i\in I}$ est un système inductif filtrant dans ${\rm Perf}_K$ alors 
     $A\coloneqq (\widehat{\varinjlim_{i\in I} B_i^0})_*$ est perfectoïde (proposition~\ref{banperf}).
L'équivalence ${\rm Perf}_K\simeq {\rm Perf}_{K^{\flat}}$ et la description explicite\footnote{Il s'agit de la sous-catégorie pleine de ${\rm Ban}_{K^{\flat}}^u$ des $K^{\flat}$-algèbres de Banach parfaites.} de ${\rm Perf}_{K^{\flat}}$
montrent que ${\rm Perf}_K$ possède aussi toutes les petites limites, \emph{mais si ${\rm car}(K)=0$ elles ne sont pas les mêmes que celles dans 
${\rm Ban}^u_K$, et cela même pour les limites finies.} Voir les exemples
prophylactiques $3.8.2$ et $3.8.5$ d'\textcite{AndreIHES}. Le point $1)$ est un autre exemple 
  de difficulté posée par la différence entre les limites dans ${\rm Ban}_K^u$ et celles dans 
  ${\rm Perf}_K$ quand ${\rm car}(K)=0$.
\end{enumerate}

   \end{rema}

    \subsection{Le lemme d'approximation de Scholze}\label{appro}
    
    Nous aurons besoin du
     résultat technique mais fondamental ci-dessous.
 La preuve originelle (\cite{Scholzethese} dans un contexte un peu différent) est assez technique, voir aussi la preuve du lemme~2.3.1 dans \textcite{CS}, qui reprend un argument de Kedlaya.  
  
   Si $A$ est un anneau et $\pi\in A$, on note 
  ${\rm Spa}_{\pi}(A)$ l'ensemble des valeurs absolues\footnote{On demande que $|ab|=|a||b|$, $|a+b|\leq \max(|a|,|b|)$, 
  pour $a,b\in A$, $|0|=0$ et $|1|=1$ si $A\ne \{0\}$.} $\lvert\cdot\rvert \colon  A\to \Gamma\cup \{0\}$, où $\Gamma$ est un groupe abélien totalement ordonné (noté multiplicativement), 
telles que $|f|\leq 1$ pour tout $f\in A$ et $\lim_{n\to\infty}|\pi^n|=0$. Si $\pi$ est non diviseur de zéro on note 
${\rm Spa}(A[\frac{1}{\pi}], A)$ le sous-ensemble des $\lvert\cdot\rvert \in {\rm Spa}_{\pi}(A)$ qui se prolongent en une valeur absolue sur 
$A[\frac{1}{\pi}]$.

 \begin{theo}[lemme d'approximation]\label{approx}
 
 Soit $A$ un anneau perfectoïde et soit 
  $\pi\in A$ tel que $A$ soit $\pi$-complet et $\pi^p\mid p$. Pour tous
  $f\in A$ et $n\in\mathbf{N}$, il existe  
$g\in A$ $p$-puissant (déf.~\ref{ppuis}) tel que 
  $$|f-g|\leq |p|\cdot \max(|g|, |\pi|^n),\,\,\text{pour tout}\,\, \lvert\cdot\rvert \in {\rm Spa}_{\pi}(A);$$
  en particulier $|f|\leq |\pi|^n$ si et seulement si $|g|\leq |\pi|^n$.
  \end{theo}

  \begin{prop}\label{topnilp} Soient $A$ un anneau, $\pi\in A$, non diviseur de zéro et $x\in A[\frac{1}{\pi}]$.
  
  a)  Si  $|x|<1$ pour tout
 $|\cdot |\in {\rm Spa}(A[\frac{1}{\pi}], A)$, alors 
     il existe $n\geq 1$ tel que $x^n\in \pi A$.
     
     b) Si  $y\in A[\frac{1}{\pi}]$ vérifie
       $|x-y|<1$ pour tout $\lvert\cdot\rvert \in {\rm Spa}(A[\frac{1}{\pi}]=A[x][\frac{1}{\pi}], A[x])$, alors $y$ est dans
         la $p$-clôture (déf.~\ref{pclos}) de $A[x]$ dans $A[\frac{1}{\pi}]$.
      \end{prop}

       \begin{proof} Le point a) est standard,
 voir le lemme 2.3.2 de \textcite{CS} pour la preuve. Le b) s'en déduit\footnote{Par a) il existe 
       $d\geq 0$ tel que $(x-y)^{p^d}\in \pi A[x]$.}.     \end{proof}
 
\subsection{Presque rappels de presque algèbre}\label{alpresque}

 On trouve dans le livre de \textcite{GRal} un exposé exhaustif de 
 la théorie des presque mathématiques\footnote{Qui semble être une manière d'écrire des $o(1)$ sans l'avouer\dots{}} de Faltings, et un excellent résumé des points essentiels
 dans la section $1$ de l'article d'\textcite{AndreIHES}.
 
    On fixe un anneau $V$ muni d'un élément $p$-puissant (d\'efinition~\ref{ppuis}) $g$, d'où un idéal idempotent 
    $\mathfrak{m}=(g^{p^{-\infty}})$ (on dira aussi, simplement et abusivement,
    que l'on se place \emph{dans le cadre $g^{1/p^{\infty}}$}). 
     La proposition $2.1.7$ de \textcite{GRal} montre que 
    $\mathfrak{m}\otimes_V \mathfrak{m}$ est un $V$-module plat (cela est évident si $V$ est sans $g$-torsion), donc les résultats de \textcite{GRal} s'appliquent.

 \begin{defi}
  Soit $f\colon  M\to N$ un morphisme de $V$-modules. On dit que 
 
 $\bullet$  $M$ est \emph{presque nul} si 
 $\mathfrak{m}M=0$; 
 
 $\bullet$ $f$ est \emph{presque injectif (resp.\ presque surjectif, resp.\ un presque isomorphisme)} 
 si $\ker f$ (resp.\ ${\rm coker} f$, resp.\ les deux) est presque nul;

 $\bullet$ $M$ est \emph{presque plat} si $M\otimes_V N\to M\otimes_V P$ est presque injectif pour tout morphisme injectif 
 $N\to P$ de $V$-modules, ce qui équivaut à la presque nullité de 
  ${\rm Tor}_i^A(M, N)$ pour tous $i>0$ et $N\in {\rm Mod}_V$. On définit de manière analogue la notion de module 
 \emph{presque projectif};
 
 $\bullet$ $M$ est \emph{presque fidèlement plat} si 
 $M$ est presque plat et si la presque nullité de $M\otimes_V N$ entraîne celle de $N$ pour tout $N\in {\rm Mod}_V$;
 
 $\bullet$ $M$ est \emph{presque de type fini} si pour 
 tout $n$ il existe un sous $V$-module de type fini $N\subset M$ tel que $g^{1/p^n}M\subset N$. 
  
 \end{defi}
   
    Si $W$ est une $V$-algèbre, $g$ reste $p$-puissant dans $W$, donc on peut appliquer les définitions 
    ci-dessus aux $W$-modules, d'où une notion de $W$-module presque plat, etc. 
    Cette remarque est utilisée dans la définition suivante:
    
    \begin{defi}
     Soit $g\colon  A\to B$ un morphisme de $V$-algèbres. On dit que 
     
     $\bullet$ 
     $g$ est \emph{presque étale} si $B$ est un $A$-module presque plat 
     et un 
     $B\otimes_A B$-module (via la multiplication $B\otimes_A B\to B$) presque projectif;

    $\bullet$ $g$ est \emph{presque fini étale} si $g$ est presque étale et si $B$ est un $A$-module presque de type fini et presque projectif.  

\end{defi}

  \subsection{Miracles perfectoïdes}\label{christmas}
  
    Les résultats suivants de  \textcite{Scholzethese} sont pour le moins surprenants: leurs analogues dans le monde des algèbres affinoïdes usuelles sont totalement faux. Dans le monde perfectoïde c'est (presque) tous les jours le printemps \dots{} 
   
  \begin{theo}\label{coprod} Soit $A\in {\rm Perf}_K$ et soient $A\to B$ et $A\to C$ deux morphismes dans ${\rm Perf}_K$. Alors 
  $B\widehat{\otimes}_A C\in {\rm Perf}_K$ et le morphisme naturel\footnote{La complétion à gauche est $\pi$-adique.}
  $$B^0\widehat{\otimes}_{A^0} C^0\to  (B\widehat{\otimes}_A C)^0$$
  est un presque isomorphisme (dans le cadre $\pi^{1/p^{\infty}}$). 
  \end{theo}
  
  \begin{proof} On peut supposer que ${\rm car}(K)=0$.
   Les anneaux $R=B^0\widehat{\otimes}_{A^0} C^0$ et 
     $S\coloneqq R/R[p^{\infty}]$ sont
   perfectoïdes (corollaire~\ref{tensor} et proposition~\ref{perfide}). La proposition~\ref{banperf} montre que 
   $S_*\in \mathcal{C}$ (cf.\ \S~\ref{dico}), donc
      $T\coloneqq S[1/\pi]=S_*[1/\pi]$ a une structure canonique de $K$-algèbre de Banach spectrale telle que $T^0=S_*$, et on a $T\in {\rm Perf}_K$.
   Comme $B$ et $C$ sont uniformes, les morphismes évidents $B^0\to R\to S$ et $C^0\to R\to S$
   s'étendent en des morphismes (dans 
   ${\rm Ban}_K$) $B\to T$ et $C\to T$, d'où un morphisme
   $B\widehat{\otimes}_A C\to T$. On construit un inverse comme suit. Soit $X$ la boule unité de $B\widehat{\otimes}_A C$.
   Comme $B$ et $C$ sont uniformes, il existe 
   $c\geq 1$ tel que les images de $B^0$ et $C^0$ dans $B\widehat{\otimes}_A C$ soient contenues dans $p^{-c}X$. Le morphisme induit $S\to B\widehat{\otimes}_A C$ a une image contenue dans $p^{-2c}X$ et s'étend ainsi en un morphisme
   $T\to B\widehat{\otimes}_A C$, inverse du morphisme $B\widehat{\otimes}_A C\to T$. 
   
   On en déduit que 
   $B\widehat{\otimes}_A C\in {\rm Perf}_K$ et que 
   $(B\widehat{\otimes}_A C)^0$ est isomorphe à $T^0=S_*$. Le morphisme $R\to S\to S_*$ est un 
  presque isomorphisme, puisque l'idéal $(\pi^{p^{-\infty}})$ annule le noyau et le conoyau 
   de $R\to S$ (proposition~\ref{reduced}) et aussi de $S\to S_*$.
  \end{proof}

    Si $B\in {\rm Ban}_K$ et si $f_1,\ldots,f_n,g\in B$ engendrent l'idéal unité on dispose d'une $B$-algèbre de Banach universelle
     $B\langle \frac{f_1,\ldots, f_n}{g}\rangle$ dans laquelle $g$ est inversible et les $\frac{f_i}{g}$ sont à puissances bornées. Explicitement, 
          $B\langle \frac{f_1,\ldots, f_n}{g}\rangle$ est le quotient de l'algèbre de Banach $B\langle T_1,\ldots, T_n\rangle=B\widehat{\otimes}_K K\langle T_1,\ldots, T_n\rangle$ par \emph{l'adhérence} de l'idéal engendré par les $gT_i-f_i$.
            
 \begin{theo}\label{oplus} Soit $A\in {\rm Perf}_K$ et 
  soient $f_1,\ldots, f_n, g\in A$ qui engendrent l'idéal unité de $A$.
  
  a) On a $ A\langle \frac{f_1,\ldots, f_n}{g}\rangle\in {\rm Perf}_K$.
  
  b) Si $f_1,\ldots, f_n, g$ sont $p$-puissants (déf.~\ref{ppuis}) dans $A^0$ et si 
  $A^0\langle (\frac{f_1}{g})^{1/p^{\infty}},\ldots, (\frac{f_n}{g})^{1/p^{\infty}}\rangle$ est le complété 
  $\pi$-adique de $A^0[\frac{f_i^{1/p^n}}{g^{1/p^n}},\,\, n\geq 0]$, alors 
   le morphisme naturel
  $$A^0\Bigl\langle \bigl(\frac{f_1}{g}\bigr)^{1/p^{\infty}},\ldots, \bigl(\frac{f_n}{g}\bigr)^{1/p^{\infty}}\Bigr\rangle\to   A\bigl\langle \frac{f_1,\ldots, f_n}{g}\bigr\rangle^0$$
  est un presque isomorphisme (dans le cadre $\pi^{1/p^{\infty}}$). 
  \end{theo}
  
  \begin{proof} Voir la section $6$ de \textcite{Scholzethese} pour les détails. Le très joli argument de \og réduction au cas universel\fg{} ci-dessous 
 est dû à  \textcite{AndreIHES}. 
  Le point a) se déduit du point b) et du théorème~\ref{approx}. Puisque $f_i,g\in A^0$ sont $p$-puissants, on dispose d'un morphisme $S\coloneqq K^0\langle T_1^{1/p^{\infty}},\ldots, T_n^{1/p^{\infty}}, U^{1/p^{\infty}}\rangle[\frac{1}{\pi}]\to A$ dans 
  ${\rm Ban}_K$, envoyant $T^{1/p^j}$ sur $f_i^{1/p^j}$ et $U^{1/p^j}$ sur $g^{1/p^j}$.
 Soit $N\geq 1$ tel que $\pi^N\in (f_1,\ldots, f_n, g)\subset A^0$. Alors $A\widehat{\otimes}_S S\langle \frac{\pi^N, T_1,\ldots, T_n}{U}\rangle $ et $A\langle \frac{f_1,\ldots, f_n}{g}\rangle$ partagent la même propriété universelle, donc ces algèbres de Banach sont isomorphes. 
  Le théorème~\ref{coprod} ramène donc la preuve au \og cas universel\fg{}, i.e.\ 
à vérifier que $S, S\langle \frac{\pi^N, T_1,\ldots, T_n}{U}\rangle \in {\rm Perf}_K$
  et que les morphismes  $K^0\langle T_1^{1/p^{\infty}},\ldots, T_n^{1/p^{\infty}}, U^{1/p^{\infty}}\rangle \to S^0$ et  
  $S^0\langle (\frac{T_1}{U})^{1/p^{\infty}},\ldots, (\frac{T_n}{U})^{1/p^{\infty}},  (\frac{\pi^N}{U})^{1/p^{\infty}}\rangle\to S\langle \frac{\pi^N, T_1,\ldots, T_n}{U}\rangle^0$
   sont des presque isomorphismes, ce qui se fait 
   par un calcul direct (assez pénible\dots{}).     \end{proof}
 
   Le résultat fondamental suivant, sur lequel tout repose dans la section~\ref{Aby}, est le \emph{théorème de presque pureté} de Faltings, étendu par  \textcite{Scholzethese} et  \textcite{KL}. Voir la section $7$ de \textcite{Scholzethese} pour la preuve (fort délicate) et le paragraphe $3.4$ d'\textcite{AndreIHES} pour des compléments.
   
   \begin{theo}\label{FalSch}
   Soit $(A, \lvert\cdot\rvert )\in {\rm Perf}_K$. Pour toute
   $A$-algèbre finie étale $B$ il existe une norme de $K$-algèbre de Banach
   $|\lvert\cdot\rvert |$ sur $B$ telle que le morphisme $(A, \lvert\cdot\rvert )\to (B, |\lvert\cdot\rvert |)$ soit continu,  
   $B\in {\rm Perf}_K$ et $B^0$ soit presque fini étale de $A^0$ (dans le cadre 
   $\pi^{1/p^{\infty}}$). Si $A\to B$ est injectif, $B^0$ est presque fidèlement plate sur $A^0$.
   \end{theo}

\section{Le lemme de platitude d'André}

    Cette section est consacrée au premier pilier des travaux d'André, son lemme de platitude, un énoncé fondamental qui, loin d'être une platitude, sera systématiquement utilisé dans la preuve des résultats principaux de ce rapport. On trouvera des raffinements et des généralisations, ainsi que des applications spectaculaires dans les travaux de  \textcite{BS},  \textcite{CS}, \textcite{Dine} et dans le livre de \textcite{GR}. En particulier, on trouvera dans le paragraphe $7.3$ de l'article de \textcite{BS} une preuve purement alg\'ebrique (mais pas facile), qui n'utilise pas la théorie des espaces perfectoïdes. 
    
   On garde les notations et les conventions introduites au début de la section~\ref{perf}.

    \subsection{Le lemme de presque platitude}
    
     Le problème qui nous occupera dans ce paragraphe est le suivant:
    on se donne 
   $A\in {\rm Perf}_K$ et  
   $g\in A^0$, et on cherche à construire une $A$-algèbre $B\in {\rm Perf}_K$ dans laquelle 
   $g$ est $p$-puissant (d\'efinition~\ref{ppuis}) et telle que $B^0/\pi$ soit presque fidèlement plate sur $A^0/\pi$ (dans le cadre 
   $\pi^{1/p^{\infty}}$ (\S~\ref{alpresque})). 
   
   L'idée naïve est de regarder la 
   $A$-algèbre de Banach universelle munie d'un système compatible de $p^{\infty}$-racines de 
   $g$, i.e.\ $A\langle T^{1/p^{\infty}}\rangle/\overline{(T-g)}$. L'algèbre $A\langle T^{1/p^{\infty}}\rangle$ est perfectoïde
   (noter que $A\langle T^{1/p^{\infty}}\rangle^0=A^0\langle T^{1/p^{\infty}}\rangle$ et utiliser la remarque~\ref{experf}), donc uniforme, mais 
   en général l'uniformité ne se propage pas aux quotients. On insiste et on la remplace par son uniformisé (\S~\ref{dico})
   $$A\langle g^{1/p^{\infty}}\rangle\coloneqq \left(A\langle T^{1/p^{\infty}}\rangle/\overline{(T-g)}\right)^u.$$
 Le résultat un peu miraculeux d'\textcite{AndreDSC}\footnote{L'énoncé suivant est tiré de l'article de  \textcite{BhattInv}.}  est que cela résout notre problème:   
   
\begin{theo}\label{fflemma}  On a $A\langle g^{1/p^{\infty}}\rangle\in {\rm Perf}_K$ et 
 $A\langle g^{1/p^{\infty}}\rangle^0/\pi$ est presque fidèlement plate sur $A^0/\pi$ (dans le cadre $\pi^{1/p^{\infty}}$).
\end{theo}

  On a un énoncé analogue modulo $\pi^d$ pour tout $d\geq 1$, se déduisant formellement du théorème. 
  La preuve se fait en deux étapes, détaillées ci-dessous: 
  on relie $A\langle g^{1/p^{\infty}}\rangle$ à des localisés de l'algèbre  perfectoïde $A\langle T^{1/p^{\infty}}\rangle$, puis on analyse ces localisés en utilisant de manière intensive les résultats des \S~\ref{appro} et~\ref{christmas}.
  
  \subsubsection{Une autre description de $A\langle g^{1/p^{\infty}}\rangle$}
  
  Soient 
  $B\in {\rm Ban}^u_K$, $f\in B$ et $I=\overline{fB}\subset B$.
   Il convient de voir l'algèbre 
   $(B/I)^u\in {\rm Ban}^u_K$ comme celle des fonctions analytiques 
  sur le fermé Zariski $V(I)$ de $\mathcal{M}(B)$. Comme $x\in \mathcal{M}(B)$ est annulé par 
  $f$ si et seulement si 
  $|f(x)|\leq |\pi^n(x)|$ pour tout $n\geq 1$, $(B/I)^u$ n'est rien d'autre que 
   la limite inductive $\varinjlim_{n} B_n$ (dans la catégorie ${\rm Ban}_K^u$, cf.\ \S~\ref{dico})
  des algèbres 
  $$B_n\coloneqq (B\langle f/\pi^n\rangle)^u.$$
  En effet, pour $S\in {\rm Ban}^u_K$ la donnée d'un morphisme (dans ${\rm Ban}_K$) de $\varinjlim_{n} B_n$ dans $S$ équivaut à celle d'un morphisme $\varphi\colon  B\to S$
  tel que $\varphi(f/\pi^n)\in S^0$ pour tout $n\geq 1$, or cela signifie précisément que 
  $\varphi$ se factorise par $(B/I)^u$ (puisque $S$ est uniforme). 
  
  Prenons maintenant
  $B=A\langle T^{1/p^{\infty}}\rangle$ et $f=T-g$ dans la discussion ci-dessus. Puisque 
  $B$ est perfectoïde, chacune des algèbres $B\langle f/\pi^n\rangle$ l'est (th\'eor\`eme~\ref{oplus}), donc 
 $B_n=B\langle f/\pi^n\rangle$ et
 $C\coloneqq A\langle g^{1/p^{\infty}}\rangle\in {\rm Perf}_K$, en tant que colimite filtrante (dans ${\rm Ban}_K^u$) d'algèbres 
 perfectoïdes (remarque~\ref{perfcat}). Comme 
 $C^0=(\widehat{\varinjlim_{n} B_n^0})_*$ est presque 
 isomorphe à $\widehat{\varinjlim_{n} B_n^0}$, l'algèbre 
  $C^0/\pi$ est presque isomorphe à 
  $\varinjlim_{n} B_n^0/\pi$. Il suffit donc de voir que chaque 
  $B_n^0/\pi$ est presque fidèlement plate sur $A^0/\pi$. On fixe par la suite $n\geq 1$.

  \subsubsection{Analyse des localisés de $A\langle T^{1/p^{\infty}}\rangle$ et fin de la preuve}

   Par le lemme d'approximation (th\'eor\`eme~\ref{approx}) il existe 
   un élément $p$-puissant 
   $h\in B^0$ tel que $|h-f|\leq |p|\cdot \max(|h|, |\pi^n|)$ pour tout 
   $\lvert\cdot\rvert \in {\rm Spa}_{\pi}(B^0)$. Comme 
   $|\frac{h-f}{\pi}|<1$ pour tout $\lvert\cdot\rvert \in {\rm Spa}(B^0[\frac{1}{\pi}], B^0)$ et
   $B^0$ est $p$-close (d\'efinition~\ref{pclos}) dans $B^0[\frac{1}{\pi}]$ (proposition~\ref{unif}), 
   la proposition~\ref{topnilp} montre que 
   $$h\equiv f\pmod {\pi B^0}.$$
   L'inégalité $|h-f|\leq |p|\cdot \max(|h|, |\pi^n|)$ montre aussi que 
   $B_n\simeq B\langle \frac{h}{\pi^n}\rangle$, et on déduit du 
  théorème~\ref{oplus} que 
  $B_n^0/\pi$ est presque isomorphe à $\varinjlim_{j} C_j/\pi$, avec 
 $C_j= B^0[ \frac{h^{1/p^j}}{\pi^{n/p^j}}]\subset B$. Nous allons montrer 
 que chaque $C_j/\pi$ est fidèlement plate sur $A^0/\pi$, ce qui permettra de conclure.
 Par dévissage il suffit de montrer que $C_j/\pi^{1/p^j}$ l'est sur $A^0/\pi^{1/p^j}$.

 On a $B^0=A^0\langle T^{1/p^{\infty}}\rangle$, donc $B^0/\pi\simeq (A^0/\pi)[T^{1/p^{\infty}}]$, en particulier 
 $B^0/\pi^{1/p^j}\simeq B^0/\pi$ via $x\mapsto x^{p^j}$ (puisqu'il en est de même du morphisme 
 $A^0/\pi^{1/p^j}\to A^0/\pi$, cf.\ proposition~\ref{banperf})
 et  
 $f=T-g$ n'est pas un diviseur de zéro modulo $\pi$. 
 Puisque $h\equiv f\pmod {\pi B^0}$, il en est de même de $h$ et la remarque~\ref{nondiv} fournit 
un isomorphisme $C_j\simeq B^0[X^{1/p^j}]/(u_j)$, avec
$u_j=\pi^{n/p^j}X^{1/p^j}-h^{1/p^j}$.
Ainsi le morphisme $x\mapsto x^{p^j}$ et la congruence $h\equiv f\pmod {\pi B^0}$ induisent des isomorphismes 
$$C_j/\pi^{1/p^j}\simeq B^0[X^{1/p^j}]/(\pi^{1/p^j}, h^{1/p^j})\simeq $$
$$ B^0[X]/(\pi, h)\simeq B^0[X]/(\pi, f)\simeq ((A^0/\pi)[T^{1/p^{\infty}}]/(T-g))[X],$$ exhibant $C_j/\pi^{1/p^j}$ comme une colimite filtrante de 
 $A^0/\pi^{1/p^j}\simeq A^0/\pi$-modules libres de type fini, ce qui permet de conclure.
\subsection{Le lemme de platitude}

    Le résultat suivant de  \textcite{GR} et  \textcite{BS}
  est une vaste généralisation et un raffinement du théorème~\ref{fflemma}. Il a des nombreuses applications, par exemple à l'étude des sous-espaces Zariski fermés d'un espace perfectoïde, à la cohomologie prismatique \parencite{BS}, etc.

\begin{theo}[lemme de platitude d'Andr\'e]
Soit $A$ un anneau perfectoïde. Il existe une $A$-algèbre perfectoïde $B$ telle que 
$B/p$ soit fidèlement plate sur $A/p$ et telle que tout $x\in B$ soit $p$-puissant (déf.~\ref{ppuis}). On peut même
choisir $B$ telle que tout polynôme unitaire $P\in B[X]$ possède une racine dans $B$.
\end{theo}

  Nous ferons usage seulement\footnote{On pourrait même se dispenser de ce paragraphe et utiliser seulement le théorème~\ref{fflemma}.} du théorème~\ref{nopresque} ci-dessous, qui entraîne celui ci-dessus par une itération un peu pénible (cf.\ la preuve du th\'eor\`eme $2.3.4$ de \textcite{CS}, pour les détails). 
 Pour la preuve, on suit la présentation par  \textcite{CS} des arguments de Gabber et Ramero. Essentiellement, le rôle des localisés dans la preuve du théorème~\ref{fflemma} est pris par les $p$-clôtures intégrales (d\'efinition~\ref{pclos}).

   \begin{theo}\label{nopresque}
                          Soit $R$ un anneau muni d'un élément $p$-puissant (déf.~\ref{ppuis}) $\pi$, non diviseur de zéro, tel que 
                          $\pi^p\mid p$ et $R/\pi\simeq R/\pi^p$ via le Frobenius. 
                          Soit $P\in R[X]\setminus R$ unitaire et soit $B$ la $p$-clôture (déf.~\ref{pclos}) de $R[X^{1/p^{\infty}}]/P$ dans $(R[X^{1/p^{\infty}}]/P)[1/\pi]$. La $\pi$-complétion de $B$ est perfectoïde, sans $\pi$-torsion, et $B/\pi$ est fidèlement plate sur $R/\pi$.                           \end{theo}

 Notons que, par construction, $P$ possède une racine $p$-puissante dans $B$, et que $R[X^{1/p^{\infty}}]/P$ est sans $\pi$-torsion car $P$ est unitaire, donc non diviseur de zéro modulo $\pi$.

\begin{proof} Soit $A=R[X^{1/p^{\infty}}]$. Comme le Frobenius 
$A/\pi\to A/\pi^p$ est un isomorphisme, $A$ est $p$-close dans $A[\frac{1}{\pi}]$ et le complété $\pi$-adique 
$\hat{A}$ est perfectoïde (proposition~\ref{310}). 
 Soit $I=PA\left[\frac{1}{\pi}\right]$ et soit $u\colon  A[\frac{1}{\pi}]\to (A/P)[\frac{1}{\pi}]\simeq A[\frac{1}{\pi}]/I$ la projection canonique. 

     Pour $f\in A$ et $n\geq 1$, notons $C_{n,f}$ la $p$-clôture de $A[\frac{f}{\pi^n}]$ dans 
$A[\frac{1}{\pi}]$, et $C_{\infty,f}=\cup_{n\geq 1} C_{n,f}$. Alors 
$I=\bigcup_{n\geq 1} \frac{P}{\pi^n} A=\bigcup_{n\geq 1} \frac{P}{\pi^n} C_{\infty, P}$ est un idéal 
$\pi$-divisible de $C_{\infty, P}$, et
   la projection $u$ induit un isomorphisme 
$$C_{\infty, P}/I\simeq B,$$
et donc des isomorphismes
$B/\pi^n\simeq C_{\infty, P}/\pi^n$ pour tout $n\geq 1$.
En effet, un
 élément $x\in A[\frac{1}{\pi}]$ vérifie 
 $u(x)\in B$ si et seulement s'il existe 
   $d\geq 0$ tel que $x^{p^d}\in A+I$ (proposition~\ref{Andrepclos}), ce qui revient à dire que $x\in C_{\infty, P}$.
 Il suffit donc de montrer que le complété $\pi$-adique de $C_{n,P}$ est perfectoïde et que 
     $C_{n,P}/\pi$ est fidèlement plat sur $R/\pi$ pour tout $n\geq 1$.

  Puisque $\hat{A}$ est perfectoïde (proposition~\ref{310}), le théorème~\ref{approx} fournit un élément $p$-puissant 
       $g\in \widehat{A}$ tel que $|P-g|\leq |p|\max(|g|, |\pi^n|)$ pour tout $|\cdot |\in {\rm Spa}_{\pi}(\widehat{A})\simeq {\rm Spa}_{\pi}(A)$. Soit 
       $(q_j)_{j\geq 0}$ une suite dans $A$ telle que
       $q_j\equiv g^{1/p^j}\pmod {\pi^{n+p} \widehat{A}}$ pour tout $j\geq 0$. On a 
       $$|P-q_0|\leq |\pi^p|\max(|q_0|, |\pi^n|),\,\, \text{pour tout}\, |\cdot |\in {\rm Spa}_{\pi}(A). \eqno(4).$$

 \begin{lemm}\label{epsi}
 Il existe $d\geq 1$ tel que $q_j^{p^j}\equiv P\pmod {\pi^{1/p^d} A}$ pour tout $j$, et 
 $q_j$ n'est pas un diviseur de zéro modulo $\pi$.
 \end{lemm}
 
 \begin{proof} La relation $(4)$ et la proposition~\ref{topnilp} montrent qu'il existe $d\geq 1$ tel que $(P-q_0)^{p^d}\in \pi A$, donc $P-q_0\in \pi^{1/p^d}A$ (car $A$ est $p$-clos dans $A[\frac{1}{\pi}]$). On conclut en remarquant que $q_j^{p^j}\equiv g\equiv q_0\pmod {\pi \hat{A}}$ et que $P$ n'est pas un diviseur de zéro modulo~$\pi$.       
\end{proof}

 \begin{lemm}\label{crucial}
  On a $C_{n, P}=\cup_{j\geq 0} A[\frac{q_j}{\pi^{n/p^j}}]$ et sa $\pi$-complétion est perfectoïde.

   \end{lemm}
   \begin{proof} 
      Notons $x_j=\frac{q_j}{\pi^{n/p^j}}\in A[\frac{1}{\pi}]$ et $y_j=\frac{g^{1/p^j}}{\pi^{n/p^j}}\in \widehat{A}[\frac{1}{\pi}]$. Alors $x_{j+1}^p- x_j\in{\pi A}$
      (puisque  $q_{j+1}^p-q_j\in\pi^{n+p}A$) et $y_{j+1}^p=y_j$, donc  
          $T\coloneqq A[x_0,x_1,\ldots]\subset A[\frac{1}{\pi}]$ (resp.\ $S\coloneqq \widehat{A}[y_0,y_1,\ldots]\subset  \widehat{A}[\frac{1}{\pi}]$) 
      est la réunion croissante des sous-algèbres $A[x_j]$ (resp.\ $\widehat{A}[y_j]$), et 
      $T\subset C_{n, q_0}$. 
         Comme 
       $x_j-y_j\in \pi \hat{A}$ et $q_j$ n'est pas un diviseur de zéro modulo $\pi$ (lemme~\ref{epsi}), on a des isomorphismes 
       naturels $$A[x_j]/\pi\simeq A[X]/(\pi, \pi^{n/p^j}X-q_j)\simeq \widehat{A}[X]/(\pi, \pi^{n/p^j}X-g^{1/p^j})\simeq \widehat{A}[y_j]/\pi,$$  qui se propagent en un isomorphisme
          $T/\pi\simeq S/\pi$. Mais $S$ est perfectoïde
 (proposition~\ref{p-closed}), donc le Frobenius $T/\pi^{1/p}\to T/\pi$ est un isomorphisme. Il s'ensuit que 
 $T$ est $p$-close dans $T[\frac{1}{\pi}]=A[\frac{1}{\pi}]$, donc $T=C_{n, q_0}$, et que 
 le complété $\pi$-adique de $T$ est perfectoïde. On conclut en remarquant que $C_{n,P}=C_{n, q_0}$, grâce à la relation $(4)$ et au point b) de la proposition~\ref{topnilp}.
                            \end{proof}
 
  Pour finir la preuve du théorème~\ref{nopresque} il suffit de voir que chaque $$A[q_j/\pi^{n/p^j}]\simeq R[T, X^{1/p^{\infty}}]/(\pi^{n/p^j}T-q_j)$$ est fidèlement plate sur $R$ modulo $\pi$. Compte tenu du lemme~\ref{epsi}, il suffit de recopier la fin de la preuve du théorème~\ref{fflemma}.
  \end{proof}
  

  \section{Le lemme d'Abhyankar perfectoïde d'André}\label{Aby}

 Cette section est consacrée à la preuve de l'autre pilier de la stratégie d'André: son difficile lemme d'Abhyankar perfectoïde \parencite{AndreIHES}. Rappelons le contexte classique: on se donne une extension finie et plate $B$ d'un 
 anneau local régulier $A$ d'inégale caractéristique $(0,p)$. On suppose que l'extension est ramifiée le long d'un diviseur à croisements normaux défini par une équation $f=0$, avec $f$ multiple de $p$, et que les indices de ramification sont premiers à $p$. On peut alors rendre l'extension 
 étale en adjoignant des racines de $f$ d'ordre divisible par tous les indices de ramification, puis en passant à la clôture intégrale (mais \emph{sans inverser $f$}).
 Dans le cadre perfectoïde les hypothèses concernant la ramification sont superflues, mais la conclusion reste valable seulement dans le cadre $f^{1/p^{\infty}}$ des presque mathématiques (\S~\ref{alpresque}). Le lecteur trouvera des g\'en\'eralisations et raffinements importants dans le livre de \textcite{GR} (\S~$16.9)$ et dans l'article de \textcite{BS} (\S~$10.2$).

  On peut voir le lemme d'Abhyankar perfectoïde comme une généralisation au cas ramifié du théorème de presque pureté de Faltings, Scholze et Kedlaya--Liu (th\'eor\`eme~\ref{FalSch}). Ce théorème jouera d'ailleurs un rôle capital dans la preuve.
    
 On garde les notations du début de la section~\ref{perf}. 
  
  \subsection{Le théorème d'extension de Riemann}
 
  Le résultat suivant d'André\footnote{La démonstration, ainsi qu'une version plus forte du théorème principal, qui sera cruciale par la suite, sont empruntées de  \textcite{BhattInv}.} (inspiré par ceux de la section II de \textcite{Scholzetorsion}) est une version perfectoïde d'un théorème
  de  \textcite{Barten} (lui-même une version rigide analytique du théorème classique d'extension de Riemann pour les variétés complexes): si 
  $X$ est un affinoïde rigide analytique \emph{normal}, toute fonction analytique bornée sur le complémentaire d'un fermé analytique 
  nulle part dense de $X$ s'étend de manière unique en une fonction analytique bornée sur $X$. 
  On verra que dans le monde perfectoïde l'hypothèse de normalité est superflue.

Soit $B\in {\rm Perf}_K$ et soit $g\in B^0$ un élément $p$-puissant (d\'efinition~\ref{ppuis}), non diviseur de z\'ero.
  Soit $B_n=B\langle \frac{\pi^n}{g}\rangle\in {\rm Perf}_K$ l'algèbre perfectoïde (th\'eor\`eme~\ref{oplus}) des fonctions analytiques sur le lieu 
  $|\pi^n|\leq |g|$ de $\mathcal{M}(B)$, et soit $B_{\infty}=\varprojlim_{n} B_n$ la 
  $K$-algèbre de Fréchet des fonctions analytiques sur le lieu 
 de non-annulation de $g$. On note 
 $$g^{-1/p^{\infty}} B^0\coloneqq \bigcap_{n\geq 1} g^{-1/p^n} B^0$$
 le sous-anneau des $f\in B[\frac{1}{g}]$ tels que $g^{1/p^n}f\in B^0$ pour tout $n\geq 1$. 
  
 
 \begin{theo}\label{Riem}
 
 a) Le sous-anneau 
 $g^{-1/p^{\infty}} B^0$ de $B[\frac{1}{g}]$
 est la clôture intégrale complète\footnote{C'est-à-dire l'ensemble des
 $f\in B[\frac{1}{g}]$ dont les puissances sont contenues dans un sous $B^0$-module de type fini de $B[\frac{1}{g}]$.} de $B^0$ dans 
 $B[\frac{1}{g}]$, et $B[\frac{1}{g}]$ est intégralement clos dans 
 $B_{\infty}$.
 
 b) Le morphisme $B[\frac{1}{g}]\to B_{\infty}$ est injectif et 
  l'image de $g^{-1/p^{\infty}} B^0$ est
 $B_{\infty}^{0}\coloneqq \varprojlim_{n} B_n^0$. \end{theo}

 L'injectivité du morphisme $B[\frac{1}{g}]\to B_{\infty}$ se ramène à celle de $B\to B_{\infty}$. Si $f\in B$ a une image nulle dans 
 $B_{\infty}$, alors $fg$ est nulle sur $\mathcal{M}(B)$, donc de norme spectrale nulle; comme $B$ est uniforme, on a $fg=0$ puis $f=0$. 
 Soit $C$ la clôture intégrale complète de $B^0$ dans $B[\frac{1}{g}]$. Si 
 $f\in g^{-1/p^{\infty}}B^0$ alors $f^n\in g^{-1}B^0$ pour tout $n$, donc $f\in C$. 
 Dans l'autre sens, soit $f\in C$. Il existe $N$ tel que 
 $f^n\in \frac{1}{(\pi g)^N} B^0$ pour tout $n\geq 1$. Comme
 $B^0$ est $p$-close (d\'efinition~\ref{pclos}) dans $B$ (proposition~\ref{unif}), on a 
 $(\pi g)^{N/p^k}f\in B^0$ pour tout $k\geq 1$, donc $\pi^{1/p^{\ell}} g^{N/p^k}f\in B^0$ pour tous 
 $k,\ell\geq 1$. Ainsi $g^{N/p^k}f\in B^0_*=B^0$ pour tout $k$ et $f\in g^{-1/p^{\infty}}B^0$, montrant bien que 
 $C=g^{-1/p^{\infty}}B^0$. 
  
  Si $f\in B_{\infty}$ est entier sur $B[\frac{1}{g}]$ alors il existe $N$ tel que $(\pi g)^Nf$ soit entier sur 
  $B^0$. Mais un élément $x$ de $B_n$ entier sur $B^0$ est dans $B_n^0$ car $B_n$ est uniforme (en effet, les puissances de $x$
  restent dans un sous $B_n^0$-module de type fini de $B_n$, qui est borné dans~$B_n$ par uniformité).
Donc $(\pi g)^N f\in B_{\infty}^0$. Pour finir la preuve du théorème~\ref{Riem} il suffit donc de montrer (et c'est le coeur de l'affaire) que $B[\frac{1}{g}]\to B_{\infty}$
  identifie $g^{-1/p^{\infty}}B^0$ et $B_{\infty}^0$. Comme $B^0=B^0_*$, il suffit de vérifier que $(\pi g)^{1/p^k} f\in B^0$ pour tous $f\in B_{\infty}^0$ et $k\geq 1$. Cela découle par passage à la limite du résultat plus fin suivant, dû à  \textcite{BhattInv}, et qui demande quelques préliminaires.
  
  On se place 
 dans le cadre 
  $(\pi g)^{1/p^{\infty}}$ des presque mathématiques (\S~\ref{alpresque}). 
 On dit qu'un système projectif $\{M_n\}_{n\geq 1}$ de $B^0$-modules est \emph{presque nul} si 
 pour tous $k\geq 0$ et $n\geq 1$ il existe $m\geq n$ tel que $(\pi g)^{1/p^k}$ annule l'image de 
 $M_m$ dans $M_n$. Il est équivalent de dire que pour tout $k\geq 1$ le morphisme naturel 
 $\{M_n[(\pi g)^{1/p^k}]\}_{n\geq 1}\to \{M_n\}_{n\geq 1}$ est un isomorphisme de pro-$B^0$-modules. 
 On dira qu'un morphisme 
 $f\colon  \{M_n\}_{n\geq 1}\to \{N_n\}_{n\geq 1}$ de pro-$B^0$-modules est un \emph{presque isomorphisme} si 
 le noyau et le conoyau de $f$ sont des systèmes projectifs presque nuls. 

  \begin{theo}\label{bhatt}
  Pour tout $d\geq 1$ le morphisme de systèmes projectifs 
  $$\{B^0/\pi^d B^0\}_{n\geq 1}\to \{B_n^0/\pi^d B_n^0\}_{n\geq 1}$$
 est un presque isomorphisme dans le cadre $(\pi g)^{1/p^{\infty}}$.
     \end{theo}
  
  \begin{proof} Par dévissage on peut supposer que $d=1$. Supposons d'abord que 
  $g$ n'est pas un diviseur de zéro modulo $\pi$. Alors chaque 
  $f_n\colon  B^0/\pi B^0\to B_n^0/\pi B_n^0$ est injectif: si $F\in B^0$ a une image $\pi G$ dans $B_n^0$, avec 
  $G\in B_n^0$, alors $u\coloneqq \frac{Fg}{\pi }\in B$ satisfait \`a \footnote{En effet, soit $x\in \mathcal{M}(B)$, alors ou bien 
  $|\pi(x)|^n\leq |g(x)|$, auquel cas $|u(x)|=|G(x)|\cdot |g(x)|\leq 1\cdot 1=1$, ou bien $|\pi(x)|^n>|g(x)|$, auquel cas
  $|u(x)|\leq |\pi(x)|^{n-1}|F(x)|\leq 1\cdot 1=1$.}  $|u|_{\rm sp}\leq 1$, donc 
 $Fg\in \pi B^0$, puis $F\in \pi B^0$.
    
  Ensuite, fixons $k,n\geq 1$ et montrons que $(\pi g)^{1/p^k}$ annule l'image de ${\rm coker}(f_{n+p^{k}})$ dans ${\rm coker}(f_n)$, ce qui permettra de conclure dans le cas où $g$ n'est pas un diviseur de zéro modulo $\pi$.
 Soit $i_n\colon  B^0\to B_n^0$ le morphisme canonique. 
  On veut montrer que si $f\in B_{n+p^{k}}^0$ a une image $F$ dans $B_n^0$, alors $(\pi g)^{1/p^k}F\in \pi B_n^0+i_n(B^0)$. En posant
  $S=B^0[ (\frac{\pi^{n+p^{k}}}{g})^{1/p^j}\,, j\geq 0]\subset B_{n+p^{k}}^0$, 
  le morphisme naturel $\widehat{S}\coloneqq \varprojlim_{d} S/\pi^d\to B_{n+p^{k}}^0$ est un presque isomorphisme dans le cadre $\pi^{1/p^{\infty}}$ (th\'eor\`eme~\ref{oplus}), donc 
  $\pi^{1/p^k}f\in S+\pi B_{n+p^{k}}^0$. Il suffit donc de montrer que 
  $g^{1/p^k} H\in \pi B_n^0+i_n(B^0)$ quand $H$ est l'image de $h=(\frac{\pi^{n+p^{k}}}{g})^e$ pour un $e\in \mathbf{Z}_{\geq 0}[1/p]$ arbitraire, mais cela est clair puisque $\pi^{p^k e}$ est un multiple de $\pi$ pour $e\geq 1/p^k$ et
  $$g^{1/p^k}H=\begin{cases} i_n(g^{\frac{1}{p^k}-e} \pi^{(n+p^k)e})\in i_n(B^0),\,\, \text{si}\,\, e<1/p^k \\ i_n(g^{1/p^k})\pi^{p^k e} (\frac{\pi^n}{g})^{e}\in \pi B_n^0,\,\, \text{sinon}.\end{cases}$$

   Il reste à s'affranchir de l'hypothèse sur $g$. Comme $g\in B^0$ est $p$-puissant, on dispose d'un morphisme
   $R\coloneqq K\langle T^{1/p^{\infty}}\rangle \to B$, envoyant $T^{1/p^j}$ sur $g^{1/p^j}$, et 
   $B_n\simeq B\widehat{\otimes}_R R_n$, avec $R_n=R\langle \frac{\pi^n}{T}\rangle$. Le théorème~\ref{coprod}
    fournit un presque isomorphisme 
    $$B^0/\pi\otimes_{R^0/\pi} R_n^0/\pi\simeq B_n^0/\pi,$$ et comme $T$ n'est pas diviseur de zéro modulo $\pi$ dans 
    $R^0$, ce qui précède montre que le morphisme $\{R^0/\pi\}\to \{R_n^0/\pi\}_{n\geq 1}$ est un presque isomorphisme de systèmes projectifs, ce qui permet de conclure facilement que 
    $\{B^0/\pi\}\to \{B_n^0/\pi\}_{n\geq 1}$ est un presque isomorphisme.
     \end{proof}
  
  \subsection{Le lemme d'Abhyankar perfectoïde}
  
  On garde les notations introduites juste avant le théorème~\ref{Riem}.
  Soit 
  $C$ une algèbre finie étale sur $B[1/g]$. \emph{Si $A\to B$ est un morphisme d'anneaux, on note 
  ${\rm fi}(A,B)\subset B$ la fermeture 
  intégrale de $A$ dans $B$.} Posons 
  $$\tilde{C}^0={\rm fi}(g^{-1/p^{\infty}}B^0, C).$$
     On verra (proposition~\ref{trickyyy}) que l'algèbre $\tilde{C}\coloneqq \tilde{C}^0[\frac{1}{p}]={\rm fi}(g^{-1/p^{\infty}}B, C)$ possède une structure naturelle de $K$-algèbre de Banach uniforme pour laquelle 
  $\tilde{C}^0=(\tilde{C})^0$ (en particulier $\tilde{C}^0$ est $p$-complète).
  
  \emph{On suppose que $K$ est de caractéristique nulle et 
on se place dans le cadre $(\pi g)^{1/p^{\infty}}$ des presque mathématiques (\S~\ref{alpresque}).}
  L'un des résultats principaux de l'article d'\textcite{AndreIHES} s'énonce (rappelons que $\xi$ est un générateur distingué de $\ker(\theta_{K^0})$):

  \begin{theo}[lemme d'Abhyankar perfectoïde]\label{downtownabbey}
    \hspace*{0cm}\par
 a) Le morphisme $\theta\colon  W((\tilde{C}^0)^{\flat})/\xi\to \tilde{C}^0$ est injectif et un presque isomorphisme.
 En particulier, le
  Frobenius est presque surjectif sur $\tilde{C}^0/p$.
 
 b) Pour tout $n\geq 1$ la $B^0/p^n$-algèbre $\tilde{C}^0/p^n$ est presque finie étale, et elle est presque fidèlement plate si 
 $C$ l'est sur $B[\frac{1}{g}]$. 
 
  \end{theo}
  
  \begin{rema}
  Contrairement au théorème de presque pureté, il n'est pas connu si 
  $\tilde{C}^0$ est presque de type fini sur $B^0$. Si c'était le cas, on pourrait en déduire que $\tilde{C}^0$ est presque finie étale sur 
  $B^0$. On peut montrer (proposition $5.2.3$ d'\textcite{AndreIHES}) que ce dernier énoncé est tout de même vrai après inversion de $g$.
  \end{rema}
  
  La preuve occupe les prochains paragraphes et fait jouer un rôle important aux algèbres 
    $$C_n\coloneqq C\otimes_{B[\frac{1}{g}]} B_n,\,\, C_{\infty}\coloneqq \varprojlim_{n} C_n,\,\, C_{\infty}^0\coloneqq \varprojlim_{n} C_n^0.$$ 
    On dispose d'isomorphismes canoniques\footnote{Le premier est dû au fait que $C_{n+1}$ est un $B_{n+1}$-module projectif de type fini.}
        $$C_{n+1}\widehat{\otimes}_{B_{n+1}} B_n\simeq C_{n+1}\otimes_{B_{n+1}} B_n= (C\otimes_{B[\frac{1}{g}]} B_{n+1})\otimes_{B_{n+1}} B_n\simeq C\otimes_{B[\frac{1}{g}]} B_n=C_n.$$

   \begin{rema}\label{almostpure}
   
   a) Comme $C$ est localement libre de rang fini sur $B[\frac{1}{g}]$, le morphisme injectif 
  (théorème~\ref{Riem})
    $B[\frac{1}{g}]\to B_{\infty}$ induit une injection $C\to C\otimes_{B[\frac{1}{g}]} B_{\infty}$, et le morphisme naturel
  $C\otimes_{B[\frac{1}{g}]} B_{\infty}\to C_{\infty}$ est un isomorphisme. On peut donc identifier $C$ à une sous-algèbre de 
  $C_{\infty}$, via $c\mapsto (c\otimes 1)_{n\geq 1}$.

  b)   Puisque $C_n$ est finie étale sur $B_n\in {\rm Perf}_K$, on a $C_n\in {\rm Perf}_K$
  et $C_n^0$ est presque finie étale sur $B_n^0$
(théorème de presque pureté~\ref{FalSch}).  Les isomorphismes $C_{n+1}\widehat{\otimes}_{B_{n+1}} B_n\simeq C_n$
          combinés au théorème~\ref{coprod} montrent que le morphisme naturel $C_{n+1}^0\widehat{\otimes}_{B_{n+1}^0} B_n^0\to C_n^0$ est un presque isomorphisme (à priori dans le cadre $\pi^{1/p^{\infty}}$, mais donc aussi dans le cadre 
          $(\pi g)^{1/p^{\infty}}$).

\end{rema}
        
    \subsubsection{Identification de la clôture intégrale}
    
    Dans ce paragraphe on montre le résultat suivant. 
    
    \begin{prop}\label{trickyyy}
    L'inclusion de $C$ dans $C_{\infty}$ identifie $\tilde{C}^0$ \`a 
    $C_{\infty}^0$. 
    \end{prop}
    
    Ainsi $\tilde{C}\coloneqq \tilde{C}^0[\frac{1}{p}]$ s'identifie à la limite uniforme 
    des algèbres de Banach uniformes $C_n$ et $(\tilde{C})^0=\tilde{C}^0$.
        La preuve de la proposition demande quelques préliminaires. 
   
   \begin{lemm}\label{key}
   On a   $C_{\infty}^0={\rm fi}(B_{\infty}^0, C_{\infty})$.
   \end{lemm}
   
   \begin{proof} Si $x=(x_n)_{n\geq 0}\in {\rm fi}(B_{\infty}^0, C_{\infty})$, alors $x_n$ est entier sur $B_n^0$, donc ses puissances restent dans un 
   sous $C_n^0$-module de type fini de $C_n$, et  
    $x_n\in C_n^0$ car $C_n$ est uniforme. Ainsi $x\in C_{\infty}^0$. 
  Pour l'autre inclusion, soit $x=(x_n)_{n\geq 0}\in C_{\infty}^0$ et soit $\chi_n\in B_n[X]$ le polynôme caractéristique de la multiplication par $x_n$ dans $C_n$. Puisque $C_{n+1}\otimes_{B_{n+1}} B_n\simeq C_n$, il existe $\chi\in B_{\infty}[X]$ induisant tous les 
  $\chi_n$. Par Cayley--Hamilton, on a $\chi_n(x_n)=0$ pour tout $n$, donc 
  $\chi(x)=0$. Le lemme~\ref{charpol} ci-dessous montre que  $\chi_n\in B_n^0[X]$ pour tout $n$, donc  
  $\chi\in B_{\infty}^0[X]$ et $x\in {\rm fi}(B_{\infty}^0, C_{\infty})$, ce qui finit la preuve.
           \end{proof}

  \begin{lemm}\label{charpol}
  Soit $R\to S$ un morphisme fini étale de $K$-algèbres de Banach uniformes. Pour tout 
  $s\in S^0$ le polynôme caractéristique de $s$ est à coefficients dans $R^0$.
  \end{lemm}
  
  \begin{proof} Soient $a_0,\ldots, a_n$ les coefficients de ce polynôme. 
  On veut montrer que $\max_{0\leq i\leq n} |a_i|_{\rm sp}\leq 1$, ou encore (cf.\ \S~\ref{algunif}) 
  que $\max_{0\leq i\leq n} |a_i(x)|\leq 1$ pour tout $x\in \mathcal{M}(S)$. 
  Si $L$ est le corps résiduel complété de $x$, 
  on veut montrer que les $a_i(x)$ appartiennent à $L^0$. 
  On peut remplacer $R\to S$ par $L\to S\otimes_R L$ et donc supposer que 
  $R=L$, auquel cas  
  $S$ est un produit d'extensions finies de $R$; le r\'esultat s'en d\'eduit.
  \end{proof}

           \begin{lemm}\label{equal}
           On a $C=C_{\infty}^0[\frac{1}{pg}]$ à l'intérieur de 
           $C_{\infty}$.
                  \end{lemm}
   
        \begin{proof}
        Notons $T=C_{\infty}^0[\frac{1}{pg}]$. Comme $C$ est entier sur $B[\frac{1}{g}]=B^0[\frac{1}{pg}]$, pour tout $x\in C$ il existe 
        $n\geq 0$ tel que $(pg)^nx$ soit entier sur $B^0\subset B_{\infty}^0$, et donc dans ${\rm fi}(B_{\infty}^0, C_{\infty})=C_{\infty}^0$ (lemme~\ref{key}). Ainsi 
        $C\subset T$.
             Montrons que $T\subset C$.        
        Par le lemme~\ref{key} et l'isomorphisme  (th\'eor\`eme~\ref{Riem}) $B[\frac{1}{g}]\simeq B_{\infty}^0[\frac{1}{pg}]$ on a 
        $T\subset {\rm fi}(B[\frac{1}{g}], C\otimes_{B[\frac{1}{g}]} B_{\infty})$. Puisque le morphisme $B[\frac{1}{g}]\to C$ est entier on a 
        ${\rm fi}(B[\frac{1}{g}], C\otimes_{B[\frac{1}{g}]} B_{\infty})={\rm fi}(C, C\otimes_{B[\frac{1}{g}]} B_{\infty})$. Comme 
        $B[\frac{1}{g}]\to C$ est étale, le morphisme naturel $ C\otimes_{B[\frac{1}{g}]} {\rm fi}(B[\frac{1}{g}], B_{\infty})\to {\rm fi}(C, C\otimes_{B[\frac{1}{g}]} B_{\infty})$ est un isomorphisme. 
        Enfin, ${\rm fi}(B[\frac{1}{g}], B_{\infty})=B[\frac{1}{g}]$ par le théorème~\ref{Riem}, donc $T\subset C$. 
        \end{proof}
        
        On finit la preuve de la proposition~\ref{trickyyy} en utilisant les lemmes~\ref{key} et~\ref{equal} et l'égalité (dans $B_{\infty}$, cf.\ th\'eor\`eme~\ref{Riem}) $g^{-1/p^{\infty}}B^0=B_{\infty}^0$: 
        $$\tilde{C}^0={\rm fi}(g^{-1/p^{\infty}} B^0, C)={\rm fi}(B_{\infty}^0, C_{\infty}^0[\frac{1}{pg}])=C_{\infty}^0.$$
           
         Le lemme~\ref{equal} a une autre conséquence importante: 
        
        \begin{coro}\label{loc}
        Le morphisme $f\colon  C_{\infty}^0[\frac{1}{p}]\widehat{\otimes}_B B_n\to C_n$ induit par la projection canonique $C_{\infty}^0[\frac{1}{p}]\to C_n$ est un isomorphisme.
        \end{coro}
        
        \begin{proof}
         On a (lemme~\ref{equal}) $$C_n=C\otimes_{B[\frac{1}{g}]} B_n= C_{\infty}^0[\frac{1}{pg}]\otimes_{B[\frac{1}{g}]} B_n\simeq C_{\infty}^0[\frac{1}{p}]\otimes_B B_n,$$ d'où un morphisme $h\colon  C_n\to C_{\infty}^0[\frac{1}{p}]\widehat{\otimes}_B B_n$ de $B_n$-modules (automatiquement continu puisque $C_n$~est projectif de type fini sur~$B_n$), tel que  $f\circ h$ soit l'identité. 
              On vérifie que $h\circ f$ est l'identité sur 
      l'image de $C_{\infty}^0[\frac{1}{p}]\otimes_B B_n$ dans $ C_{\infty}^0[\frac{1}{p}]\widehat{\otimes}_B B_n$, donc aussi sur 
      $ C_{\infty}^0[\frac{1}{p}]\widehat{\otimes}_B B_n$ par continuité et densité. Donc $f$~est un 
      isomorphisme.
\end{proof}
      
      \subsubsection{Etude du système projectif $\{C_n^0/\pi^d\}$ et fin de la preuve}
      La différence entre les limites dans  
      ${\rm Ban}^u_K$ et celles dans 
      ${\rm Perf}_K$ (remarque~\ref{perfcat}) se fera pleinement sentir dans ce paragraphe.

        Rappelons que $\xi$ est un générateur distingué de $\ker(\theta_{K^0})$ et que 
          $\pi$ est l'image de $[\xi(0)^{1/p}]$. En posant 
           $R_n=(C_n^0)^{\flat}$, 
           on obtient (th\'eor\`eme~\ref{tilt}) des isomorphismes $\theta_{C_n^0}\colon  C_n^0\simeq W(R_n)/\xi$ compatibles avec la variation de $n$. Comme $C_n^0$ est sans 
           $p$-torsion, les $R_n$ n'ont pas de $\xi(0)$-torsion, donc 
            $R=\varprojlim_{n} R_n\simeq (C_{\infty}^0)^{\flat}\in {\rm Perf}_{\mathbf{F}_p}$ n'en a pas non plus et  
            $S\coloneqq W(R)/\xi$ est une $K^0$-algèbre perfectoïde plate. 
      
      \begin{prop}\label{pain}
     Les morphismes naturels
       $C_{\infty}^0/\pi\to \varprojlim_{n} (C_n^0/\pi)$ et $S\to C_{\infty}^0$
      sont des presque isomorphismes dans le cadre $(\pi g)^{1/p^{\infty}}$, et $S\to C_{\infty}^0$ est injectif.
            \end{prop}
      
      \begin{proof} 
      Le morphisme naturel  
          $\varprojlim_{n} (C_n^0/\pi^d)\to \varprojlim_{n} (C_n^0/\pi^e)$ est presque surjectif pour tous $d>e$:
          par $K^0$-platitude des $C_n$ il suffit de voir que $R^1\varprojlim_{n} (C_n^0/\pi)$ est presque nul; or 
          le système projectif $\{B_n^0/\pi\}_{n\geq 1}$ est presque Mittag-Leffler\footnote{Un système projectif $\{M_n\}_{n\geq 1}$ de $B^0$-modules 
          est dit presque Mittag-Leffler si pour tous $k,n$ il existe $m\geq n$ tel que $(\pi g)^{1/p^k} {\rm Im}(M_m\to M_n)\subset {\rm Im}(M_{t}\to M_n)$ pour tout $t\geq m$.}
           (conséquence directe du théorème~\ref{bhatt}) et le presque isomorphisme 
          $C_{n+1}^0/\pi \otimes_{B_{n+1}^0/\pi} B_n^0/\pi\to C_n^0/\pi$ (remarque~\ref{almostpure}) montre qu'il en est de même du système projectif 
          $\{C_n^0/\pi\}_{n\geq 1}$.

         Ainsi les transitions du système projectif $\{\varprojlim_{n} (C_n^0/\pi^d)\}_{d\geq 1}$ sont presque surjectives, ayant pour conséquence la presque 
          surjectivité du morphisme $$C_{\infty}^0\simeq \varprojlim_{d} (\varprojlim_{n} (C_n^0/\pi^d))\to \varprojlim_{n} (C_n^0/\pi).$$ 
          Puisque $C_{\infty}^0/\pi\to\varprojlim_{n} (C_n^0/\pi)$ est clairement injectif, c'est un presque isomorphisme.

                              Pour montrer que $S\to C_{\infty}^0$ est injectif et un presque isomorphisme 
                               il suffit de le vérifier modulo $\pi$ (les 
                              deux algèbres étant plates sur 
          $K^0$ et $\pi$-complètes). Mais $S/\pi\simeq R/\xi(0)^{1/p}$, $C_n^0/\pi\simeq R_n/\xi(0)^{1/p}$ et nous avons vu que 
          $C_{\infty}^0/\pi\to \varprojlim_{n} (C_n^0/\pi)$ 
          est injectif et un presque isomorphisme. Il suffit donc de vérifier que  
          $R/\xi(0)^{1/p}\to \varprojlim_{n} (R_n/\xi(0)^{1/p})$ est injectif et un presque isomorphisme. L'injectivité est claire puisque les $R_n$ n'ont pas de 
          $\xi(0)^{1/p}$-torsion. Le fait que c'est un presque isomorphisme découle (comme ci-dessus, par
           $\xi(0)$-complétude des 
          $R_n$) de la presque surjectivité des applications canoniques $\varprojlim_{n} (R_n/\xi(0)^{p^k})\to \varprojlim_{n} (R_n/\xi(0)^{p^{k-1}})$. Celle-ci
         s'obtient en utilisant la perfection des $R_n$ et en appliquant successivement 
          Frobenius au morphisme presque surjectif 
          (cf.\ premier paragraphe) $\varprojlim_{n} (C_n^0/\pi^p)\to  \varprojlim_{n} (C_n^0/\pi)$, qui s'identifie à
          $\varprojlim_{n} (R_n/\xi(0))\to \varprojlim_{n} (R_n/\xi(0)^{1/p})$.                         
      \end{proof}
                  
      Posons $T=S[\frac{1}{p}]=S_*[\frac{1}{p}]\in {\rm Perf}_K$ (proposition~\ref{banperf}). La proposition~\ref{pain} fournit un presque isomorphisme 
      $T\simeq C_{\infty}^0[\frac{1}{p}]$. Comme $\pi g$ est inversible dans $B_n$, cela induit un isomorphisme 
      $T\widehat{\otimes}_B B_n\simeq C_{\infty}^0[\frac{1}{p}]\widehat{\otimes}_B B_n$, d'où (corollaire~\ref{loc}) un isomorphisme 
      $T\widehat{\otimes}_B B_n\simeq C_n$. 
      Puisque
      $T, B, B_n$ sont dans ${\rm Perf}_K$, le morphisme naturel 
      $T^0\widehat{\otimes}_{B^0} B_n^0\to (T\widehat{\otimes}_B B_n)^0$ est un presque isomorphisme (th\'eor\`eme~\ref{coprod}). 
      On obtient donc des presque isomorphismes, compatibles avec la variation de $n$ et $d$
      $$T^0/\pi^d\otimes_{B^0/\pi^d} B_n^0/\pi^d\simeq C_n^0/\pi^d.$$
      
      Par le théorème de presque pureté~\ref{FalSch}, la $B_n^0/\pi^d$-algèbre 
       $C_n^0/\pi^d$ est presque finie étale (et fidèlement plate, si $C$ l'est sur $B[\frac{1}{g}]$).
       Les isomorphismes ci-dessus combinés avec le th\'eor\`eme~\ref{bhatt} et le        
        lemme~\ref{descent} ci-dessous montrent que 
       $T^0/\pi^d$ est presque finie étale (et fidèlement plate, si $C$ l'est sur $B[\frac{1}{g}]$) sur $B^0/\pi^d$. Mais 
       $T^0=S_*$ est presque isomorphe à $S$, qui est presque isomorphe à $C_{\infty}^0$ (proposition~\ref{pain}), qui s'identifie enfin à 
       $\tilde{C}^0$ par la proposition~\ref{trickyyy}. Cela finit la preuve du théorème~\ref{downtownabbey}.
       
       Le lemme suivant a jou\'e un r\^ole important ci-dessus, permettant de transf\'erer les propri\'et\'es des $B_n^0/\pi^d$-alg\`ebres $C_n^0/\pi^d$ (fournies par le th\'eor\`eme de presque puret\'e) \`a la $B^0/\pi^d$-alg\`ebre $T^0/\pi^d$. On trouve dans la section $14.2$ de \textcite{GR} des r\'esultats beaucoup plus pr\'ecis et g\'en\'eraux que l'\'enonc\'e ci-dessous.
       
      \begin{lemm}\label{descent}
      Soit $\mathcal{P}\in \{ \text{plat, fidèlement plat, projectif, de type fini}\}$. Soit $M$ un 
    $B^0/\pi^d$-module et $M_n\coloneqq M\otimes_{B^0/\pi^d} B_n^0/\pi^d$. 
    
    a) Si le $B_n^0/\pi^d$-module $M_n$ est
      presque $\mathcal{P}$ pour tout $n$, alors 
      $M$ est presque $\mathcal{P}$. 
      
      b) Si $M$ est une $B^0/\pi^d$-algèbre et si $M_n$ est une 
       $B_n^0/\pi^d$-algèbre
      presque finie étale pour tout $n$, alors $M$ est presque finie étale sur $B^0/\pi^d$ (idem avec presque finie étale et fidèlement plate). 
      \end{lemm}
      
      \begin{proof}
      Il s'agit d'une conséquence formelle du fait que le système projectif $\{B_n^0/\pi^d\}$ est \og presque constant de valeur $B^0/\pi^d$\fg{} (th\'eor\`eme~\ref{bhatt}) et de l'interprétation catégorique (à la Yoneda) de $\mathcal{P}$. Voir le th\'eor\`eme $14.2.39$ de \textcite{GR} pour les détails.
      \end{proof}

\section{Existence d'algèbres de Cohen--Macaulay}

 \emph{On fixe par la suite de cette section, qui est le coeur de l'expos\'e, un nombre premier $p$ (d'où une notion d'anneau perfectoïde). On suppose que la caractéristique résiduelle de tous les anneaux locaux noethériens rencontrés ci-dessous est égale à $p$. Pour éviter des longues listes d'épithètes, introduisons: }

 \begin{defi}\label{symp}
 La cat\'egorie\footnote{ ${\rm CLI}$ est l'abr\'eviation de \og complet local int\`egre \fg{}, l'indice~$p$
 indique la caract\'eristique r\'esiduelle, ainsi que le caract\`ere parfait du corps r\'esiduel.}
 ${\rm CLI}_p$ a pour objets les anneaux locaux noeth\'eriens complets, int\`egres, d'in\'egale 
 caractéristique et de corps résiduel parfait (de caractéristique~$p$), et pour morphismes les morphismes locaux d'anneaux locaux.
 \end{defi}

  Le but de cette section est d'expliquer la preuve du théorème suivant. Combiné avec 
  les travaux de \textcite{HH1}, il implique le théorème~\ref{existbig} de l'introduction.
    
  \begin{theo}\label{CMexistent}
  Pour tout $A\in {\rm CLI}_p$ il existe une $A$-algèbre de Cohen--Macaulay.
  \end{theo}


Voir les th\'eor\`emes~\ref{Shimo} et~\ref{weak} pour des r\'esultats bien plus forts et pr\'ecis. 
\textcite{Ma} a remarqué qu'une construction d'\textcite{AndreJAMS}
fournit une preuve relativement directe\footnote{Modulo le lemme de platitude d'André, qui est tout sauf une platitude!} du théorème ci-dessus, n'utilisant pas le difficile 
 lemme d'Abhyankar 
perfectoïde (comme dans \cite{AndreDSC}): on traite d'abord le cas d'un anneau régulier, en construisant dans ce cas une $A$-algèbre de Cohen--Macaulay \emph{perfectoïde} $C$ (ceci est standard), puis on \'ecrit l'anneau $A\in {\rm CLI}_p$ comme un quotient d'un anneau régulier $A_0$ et on fabrique à partir 
de $C$ et du lemme de platitude d'André une $A$-algèbre \emph{presque perfectoïde et presque de Cohen--Macaulay} $C'$ (cf.\ proposition~\ref{alCMexist} pour l'énoncé précis, mais peu ragoûtant). Nous adaptons la m\'ethode de Ma pour obtenir aussi une forme faible de fonctorialit\'e de la construction, retrouvant ainsi l'un des r\'esultats principaux d'\textcite{AndreJAMS} (on trouvera cependant dans loc.cit. des r\'esultats bien plus forts que ceux expos\'es ici). 
 Une autre construction de $C'$ (celle d'\cite{AndreDSC}) utilise plutôt le fait que $A$ est une \emph{extension} finie d'un
anneau local régulier complet $A_0$, et le lemme d'Abhyankar perfectoïde combiné avec le lemme de platitude d'André pour construire $C'$. Les deux preuves produisent donc seulement une \og presque algèbre de Cohen--Macaulay (presque perfectoïde)\fg{}, mais il était connu, grâce aux travaux de  \textcite{H6}, que cela suffit pour conclure la preuve du théorème~\ref{CMexistent}. Gabber a découvert un autre moyen, plus canonique et direct, de passer d'une presque algèbre de Cohen--Macaulay à une vraie telle algèbre, il est exposé dans le paragraphe $17.5$ du livre de \textcite{GR}. On trouvera dans cette section une version encore plus simple et directe.


  \subsection{Algèbres perfectoïdes CM}

On fixe une fois pour toutes une suite compatible 
$(p^{1/p^n})_{n\geq 0}$ de racines de $p$ dans une clôture algébrique de $\mathbf{Q}_p$ et on définit le corps 
perfectoïde $$ K\coloneqq \widehat{\mathbf{Z}_p[p^{1/p^{\infty}}]}[\frac{1}{p}],$$
    pour lequel $K^0=\widehat{\mathbf{Z}_p[p^{1/p^{\infty}}]}$. Soit $p^{\flat}\coloneqq (p, p^{1/p}, p^{1/p^2},\ldots)\in K^{0,\flat}$ et 
    $\xi\coloneqq p-[p^{\flat}]$, un élément distingué de $W(K^{0,\flat})$.
    On note 
    ${\rm Perf}_{K^0}^{\rm tf}$ la catégorie des $K^0$-algèbres perfectoïdes \emph{sans $p$-torsion}.
    

         \begin{defi}\label{perfCM} Soit
         $A$ un anneau local noeth\'erien. Une \emph{$A$-algèbre CM} (resp.\  \emph{$A$-algèbre perfectoïde CM}) est une $A$-algèbre de Cohen--Macaulay
         (resp.\ une 
         $A$-algèbre de Cohen--Macaulay qui est aussi dans ${\rm Perf}_{K^0}^{\rm tf}$). 
         
         \end{defi}
         
        Nous aurons besoin des résultats suivants par la suite:

    \begin{lemm}\label{221}
    Soit 
    $(x_1,\ldots, x_d)$ une suite régulière dans une algèbre $C\in {\rm Perf}^{\rm tf}_{K^0}$, avec $x_1=p$. Le complété 
    $(x_1,\ldots, x_d)$-adique $\hat{C}$ de $C$ est dans ${\rm Perf}^{\rm tf}_{K^0}$.
    \end{lemm}
    
    \begin{proof} Voir la proposition 2.2.1 d'\textcite{AndreJAMS} pour les détails. Puisque $I\coloneqq (x_1,\ldots, x_d)$ est de type fini, $\widehat{C}$ est 
    $I$-complet (donc aussi $p$-complet). 
    Si $z\in C$ vérifie $pz\in I^{nd}C$, alors
     $z\in I^{n-1}C$. En effet, on a 
     $I^{nd}\subset (p^n, x_2^n,\ldots, x_d^n)$, donc il existe $y\in C$ tel que $p(z-p^{n-1}y)\in (x_2^n,\ldots, x_d^n)C$, puis par régularité de $(p, x_2^n,\ldots, x_d^n)$ on obtient  $z-p^{n-1}y\in (x_2^n,\ldots, x_d^n)C$ et $z\in I^{n-1}C$.
     On en déduit immédiatement que 
     $\widehat{C}$ est sans $p$-torsion et que les id\'eaux $p^{1/p}\widehat{C}$ et $p\widehat{C}$ sont fermés 
     dans $\widehat{C}$ pour la topologie $I$-adique, donc 
    $\widehat{C}/p\simeq \widehat{C/p}$, $\hat{C}/p^{1/p}\simeq \widehat{C/p^{1/p}}$. On conclut en utilisant la proposition~\ref{310}.
    \end{proof}

 \begin{lemm}\label{CMcover}
 Soit $(A, \mathfrak{m})\in {\rm CLI}_p$ (d\'ef.~\ref{symp}) et soit $C$ une $A$-algèbre perfectoïde {\rm CM}, $\mathfrak{m}$-complète.
 Pour tous $z_1,\ldots, z_n\in A$ il existe une $C$-algèbre perfectoïde {\rm CM}, $\mathfrak{m}$-complète $C'$ dans laquelle 
 $z_1,\ldots, z_n$ sont $p$-puissants (déf.~\ref{ppuis}).  
 \end{lemm}

\begin{proof}
Par le théorème~\ref{nopresque} il existe une $C$-algèbre $P\in {\rm Perf}_{K^0}^{\rm tf}$ fidèlement plate sur 
$C$ modulo $p$ et dans laquelle $z_1,\ldots, z_n$ sont $p$-puissants. Si $(x_1,\ldots, x_d)$ est un système de paramètres de 
$A$ avec $x_1=p$, alors $(x_1,\ldots, x_d)$ est une suite régulière dans $C$, et elle le reste dans $P$ par fidèle platitude modulo $p$ de celui-ci sur $C$. Par la proposition~\ref{CMstandard} le complété $(x_1,\ldots, x_d)$-adique (qui est aussi le complété 
$\mathfrak{m}$-adique) $C'$ de $P$ est une $A$-algèbre CM $\mathfrak{m}$-complète, qui est 
dans ${\rm Perf}_{K^0}^{\rm tf}$ par le lemme~\ref{221}. 
\end{proof}

\subsection{Algèbres perfectoïdes CM sur un anneau local régulier}

Les théorèmes~\ref{perfcover} et~\ref{kunz} ci-dessous sont des analogues en inégale caractéristique du célèbre résultat de  \textcite{Kunz}: pour un anneau noethérien $A$ de caractéristique $p$ chacune des assertions suivantes est équivalente à la régularité de $A$:

$\bullet$ le Frobenius $\varphi\colon  A\to A$ de $A$ est un morphisme plat.

$\bullet$ il existe une $A$-algèbre parfaite 
et fidèlement plate.

 Rappelons (proposition~\ref{CMstandard}) que pour  
  un anneau local régulier~$A$ une $A$-algèbre est CM si et seulement si elle est fidèlement plate sur~$A$. En utilisant la partie facile du théorème de Kunz et le résultat suivant, on
  en déduit facilement que pour tout anneau régulier~$A$ avec 
    $p\in {\rm Rad}(A)$ il existe une $A$-algèbre perfectoïde fidèlement plate.

   \begin{theo}\label{perfcover}
Pour tout anneau local régulier $(A, \mathfrak{m})$ d'inégale caractéristique il existe une $A$-algèbre perfectoïde {\rm CM}, $\mathfrak{m}$-complète. 
   \end{theo}
   
     Dans les applications nous aurons uniquement
besoin du théorème pour un anneau de la forme $W(k)[[T_1,\ldots, T_d]]$, $k$ étant un corps parfait de caractéristique $p$, auquel cas la preuve est parfaitement élémentaire.

   \begin{proof} On peut supposer que $A$ est complet, car le complété $\hat{A}$ de $A$ est fidèlement plat sur $A$.  
 Par le théorème de structure de Cohen et la régularité de $A$ il existe un anneau de valuation $V$, complet et absolument non ramifié\footnote{
   Cela veut dire que l'idéal maximal de $V$ est $pV$.}, ainsi qu'un élément 
   $f\in V[[X_1,\ldots, X_n]]$, que l'on peut choisir de la forme 
$f=X_1$ si $p\in \mathfrak{m}\setminus \mathfrak{m}^2$
 et dans l'idéal $(p,X_1,\ldots, X_n)^2$ sinon, tels que
 $$A\simeq V[[X_1,\ldots, X_n]]/(p-f).$$

   Supposons pour commencer que $k$ est parfait. Posons  
       $$A_j=V[[X_1^{1/p^j},\ldots, X_n^{1/p^j}]]/(p-f),\,\, A_{\infty}=\varinjlim_{j} A_j$$
et montrons que 
  le complété $p$-adique $\widehat{A}_{\infty}$ de $A_{\infty}$ est une $A$-algèbre perfectoïde fidèlement plate.
  Les morphismes $A\to A_j$ sont locaux, injectifs, finis et plats, 
   donc fidèlement plats. Ainsi
  $A_{\infty}$ est une $A$-algèbre fidèlement plate, et le lemme ci-dessous montre que 
  $\widehat{A}_{\infty}$ reste fidèlement plate sur $A$.
    Il est évident que $\widehat{A}_{\infty}$ est sans $p$-torsion et que le Frobenius sur 
$A_{\infty}/p$ est surjectif. Montrons qu'il existe $\pi\in \widehat{A}_{\infty}$ tel que $(\pi^p)=(p)$. Cela est évident si $f=X_1$, supposons donc que $f\in (p,X_1,\ldots, X_n)^2$. Puisque 
$A_{\infty}=A_{\infty}^p+pA_{\infty}$ (donc $\mathfrak{m}_{A_{\infty}}=\mathfrak{m}_{A_{\infty}}^p+pA_{\infty}$) et $p\in \mathfrak{m}_{A_{\infty}}^2$, on obtient 
$p=\pi^p+py$ pour certains $\pi,y\in \mathfrak{m}_{A_{\infty}}$, donc $(p)=(\pi^p)$. 
On conclut alors en utilisant la proposition~\ref{310} et en remarquant que le Frobenius $A_{\infty}/\pi\to A_{\infty}/\pi^p$ est injectif, puisque 
$A_{\infty}$ est normal (tous les $A_j$ sont réguliers, donc normaux), donc $p$-clos dans $A_{\infty}[\frac{1}{\pi}]$.

      Dans le cas général soit~$\bar{k}$ une clôture algébrique de~$k$ et soit
   $V\to W$ une extension fidèlement plate 
     d'anneaux de valuation absolument non ramifiés, telle que le corps résiduel de~$W$ soit~$\bar{k}$. Puisque $W[[X_1,\ldots, X_n]]$ est une 
    $V[[X_1,\ldots, X_n]]$-algèbre fidèlement plate (par exemple par le lemme ci-dessous), 
    $W[[X_1,\ldots, X_n]]/(p-f)$ est une extension fidèlement plate de~$A$, avec un corps résiduel algébriquement clos. Mais $W[[X_1,\ldots, X_n]]/(p-f)$ possède une algèbre perfectoïde fidèlement plate~$S$ d'après ce qui précède, et elle répond à l'appel.
    
    Ce qui précède montre qu'il existe une $A$-algèbre perfectoïde {\rm CM}. Son complété 
    $\mathfrak{m}$-adique répond à l'appel par les lemmes~\ref{221} et~\ref{compflat}. 
    \end{proof}

Nous avons utilisé le résultat suivant (cf.\ lemme 1.1.1 d'\textcite{AndreDSC}): 

\begin{lemm}\label{compflat}
Soit $I$ un idéal d'un anneau noethérien $A$ et soit $B$ une $A$-algèbre plate. On note $\hat{X}=\varprojlim_{n} X/I^nX$ pour tout 
$A$-module $X$.

a) Pour tout $A$-module de type fini $M$ le morphisme naturel $M\otimes_A \hat{B}\to \widehat{M\otimes_A B}$ est un isomorphisme.

b) La $A$-algèbre $\hat{B}$ est plate, et même fidèlement plate si
$I\subset {\rm Rad}(A)$.
\end{lemm}

\begin{proof}
 a) Une application directe du lemme d'Artin--Rees montre que le foncteur $\mathcal{F}\colon  M\mapsto \widehat{M\otimes_A B}$ est exact sur les 
 $A$-modules de type fini. Puisque le morphisme $f_M\colon  M\otimes_A \hat{B}\to \widehat{M\otimes_A B}$ est un isomorphisme si 
 $M$ est libre de type fini sur $A$, il l'est pour tout $M$ de type fini sur $A$
 par un argument standard.
 
 b) La platitude découle directement du point a) et de sa preuve. Supposons que $I\subset {\rm Rad}(A)$ et que 
 $B$ est fidèlement plat sur $A$. Soit $M$ un $A$-module de type fini tel que $M\otimes_A \hat{B}=0$. Alors 
 $M\otimes_A B/IB=0$ et donc $M/IM\otimes_A B=0$, puis $M/IM=0$ et (lemme de Nakayama) $M=0$.
 \end{proof}

    La preuve du résultat suivant est nettement plus délicate et nous renvoyons le lecteur à \textcite{BIMa} car nous n'en ferons pas usage.

\begin{theo}\label{kunz} 
 Soit $A$ un anneau noethérien tel que 
 $p\in {\rm Rad}(A)$. S'il existe une 
 $A$-algèbre perfectoïde et fidèlement plate, alors $A$ est régulier. 
\end{theo}

 \subsection{Un résultat technique crucial}
 
     \emph{On fixe dans ce paragraphe un anneau $(A,\mathfrak{m})\in {\rm CLI}_p$ (d\'ef.~\ref{symp}) et $\wp\in {\rm Spec} A[\frac{1}{p}]$, de hauteur $c\geq 1$. On a donc $A/\wp\in {\rm CLI}_p$.}
 
   \begin{prop}\label{choosesop}
    Il existe 
   $g\in A \setminus\wp$ multiple de $p$, ainsi qu'un système de paramètres $(f_1,\ldots, f_c, x_1,\ldots, x_d)$ de $A$ avec 
   $x_1=p$ et 
   $$g\wp\subset \sqrt{(f_1,\ldots, f_c)}\subset \wp.$$
   La suite $(x_1,\ldots, x_d)$ est un système de paramètres de $A/\wp$.
   \end{prop}
   
   \begin{proof}
   Par le théorème de l'idéal principal de Krull  (et un argument standard d'évitement des idéaux premiers) il existe $f_1,\ldots, f_c\in \wp$ tels que ${\rm ht}(p, f_1,\ldots, f_c)=c+1$. Ainsi 
   $(p,f_1,\ldots, f_c)$ s'étend en un système de paramètres $(f_1,\ldots, f_c, x_1,\ldots, x_d)$ de $A$ avec 
   $x_1=p$. Puisque $\wp$~est un idéal premier minimal de l'anneau réduit $A'\coloneqq A/\sqrt{(f_1,\ldots, f_c)}$, tout élément de~$\wp$ est un diviseur de zéro dans~$A'$, d'où l'existence de~$g$ (que l'on peut choisir multiple de~$p$ car $p\notin \wp$).
   \end{proof}
        
                \begin{defi}\label{prreg}
  Soit $R$ un anneau  
       muni d'un élément $p$-puissant (déf.~\ref{ppuis}) $g$, et soit $S$ une $R$-algèbre. 
            Une suite $x=(x_1,\ldots, x_d)$ d'éléments de $S$ est dite \emph{presque régulière (dans le cadre $g^{1/p^{\infty}}$)} si 
     l'idéal $(g^{p^{-\infty}})$ annule\footnote{Si $I,J$ sont deux idéaux d'un anneau $S$, on note $I: J=\{s\in S|\, sJ\subset I\}$.}
       $\frac{(x_1,\ldots, x_i): (x_{i+1})}{(x_1,\ldots, x_i)}$ pour tout $i$, mais n'annule pas 
     $S/(x_1,\ldots, x_d)$ (en particulier $g\ne 0$).
     
        \end{defi}
        
        On fera attention au fait qu'une suite régulière dans $S$ n'a aucune raison d'être presque régulière (!): même si $S/(x_1,\ldots, x_d)$ est non nul, il pourrait très bien être annulé par 
        $(g^{p^{-\infty}})$. Le résultat d\'elicat suivant (tiré et adapté d'\textcite{AndreJAMS}) montre que ce genre de canular ne se produit pas dans notre situation. 
             
     \begin{prop}\label{presentquot}
     Soit 
     $(f_1,\ldots, f_c, x_1,\ldots, x_d)$ un système de paramètres de $A$, avec $x_1=p$.
     Soit $C$ une $A$-algèbre {\rm CM}, $\mathfrak{m}$-complète, dans laquelle $p, f_1,\ldots, f_c$ sont $p$-puissants (déf.~\ref{ppuis}). Posons 
    $$\bar{C}\coloneqq C/(f_1^{p^{-\infty}},\ldots, f_c^{p^{-\infty}})C,\,\, \widehat{\overline{C}}=\varprojlim_{n} \overline{C}/p^n.$$  
       
     a) La suite $(x_1,\ldots, x_d)$ est régulière dans $\bar{C}$ et dans $\widehat{\overline{C}}$.

     b) Si $C\in {\rm Perf}_{K^0}^{\rm tf}$, il en est de m\^eme de $\widehat{\overline{C}}$.
     
      c) Soit $\wp\in V(f_1,\ldots, f_c)$ et soit
     $g\in A\setminus \wp$ un élément qui devient $p$-puissant dans $C$. La suite $(x_1,\ldots, x_d)$ est presque régulière (dans le cadre $g^{1/p^{\infty}}$, déf.~\ref{prreg}) dans $\overline{C}$ et $\widehat{\overline{C}}$.

                    \end{prop}
          
          \begin{proof} 
          
          a) Comme $x_1=p$, il suffit de voir que $(x_1,\ldots, x_d)$ est régulière dans $\bar{C}=\varinjlim_{n} C/(f_1^{1/p^n},\ldots, f_c^{1/p^n})$, ou encore que $(x_1,\ldots, x_d)$ est régulière dans $C/(f_1^{1/p^n},\ldots, f_c^{1/p^n})$ pour tout $n$, or cela est clair puisque la suite 
          $(f_1^{1/p^n},\ldots, f_c^{1/p^n}, x_1,\ldots, x_d)$ est régulière\footnote{Rappelons que 
          si $M$ est un module sur un anneau $R$ et $x_1,\ldots, x_n\in R$, $e_1,\ldots, e_n\in \mathbf{Z}_{>0}$, alors 
          la suite $(x_1,\ldots, x_n)$ est régulière dans $M$ si et seulement si la suite $(x_1^{e_1},\ldots, x_n^{e_n})$ l'est.}
           dans $C$.

      
      b) Par a) $\widehat{\overline{C}}$ n'a pas de $p$-torsion. En utilisant la proposition~\ref{310}, il suffit de vérifier que le Frobenius $\overline{C}/p^{1/p}\to \overline{C}/p$ est un isomorphisme, or cela découle directement des définitions et de l'isomorphisme $C/p^{1/p}\simeq C/p$ induit par le Frobenius.
                  
                  c) Ceci co\^ute nettement plus cher.
                      Par a) il suffit de vérifier que 
      $\bar{C}/(x_1,\ldots, x_d)$ n'est pas annulé par $(g^{p^{-\infty}})$. Supposons que ce n'est pas le cas, donc
      $$(g^{p^{-\infty}})C\subset (f_1^{p^{-\infty}},\ldots, f_c^{p^{-\infty}})C+(x_1,\ldots, x_d)C,$$
  en particulier $g\in \mathfrak{m}^n C+\sqrt{\wp C}$ pour tout $n$.

  
  

  

  
  Soit $\bar{C}=C/\wp C$, $\bar{A}=A/\wp$ et notons 
  $\bar{x}$ l'image de $x$ dans 
 $\bar{C}$ (resp. $\bar{A}$) si $x\in C$ (resp. $x\in A$). 
  Soit $f\colon  \bar{C}\to \bar{A}$ un morphisme de $\bar{A}$-modules. Pour tout $n\geq 1$ il existe 
$\alpha\in \mathfrak{m}^nC$ et $N\geq 1$ tel que $(g-\alpha)^N\in \wp C$, donc $(\bar{g}-\overline{\alpha})^N=0$ dans $\bar{C}$. En appliquant $f$ on obtient 
$\bar{g}^N f(1)-\binom{N}{1}\bar{g}^{N-1} f(\overline{\alpha})+\cdots+(-1)^N
f(\overline{\alpha}^N)=0$, avec $f(\overline{\alpha}^k)\in
\mathfrak{m}_{\bar{C}}^{nk}$. Ainsi $x\coloneqq f(1)\overline{g}\in \bar{A}$ 
vérifie une équation de la forme $x^N+a_{1}x^{N-1}+\cdots+a_N=0$ avec 
$a_i\in \mathfrak{m}_{\bar{A}}^{ni}$ (dépendant de $n$, bien entendu). 

Montrons que cela force $x=0$. La clôture intégrale 
de $\bar{A}$ dans son corps des fractions est un anneau noethérien normal et intègre (théorème de Nagata), donc une intersection d'anneaux de valuation discrète $(V_i)_{i\in I}$ (les localisés de cette clôture intégrale en ses idéaux premiers de hauteur $1$) tels que le morphisme $\bar{A}\to V_i$ soit local. L'image 
$y$ de $x$ dans un tel $V_i$ vérifie des équations du type $y^N+b_1y^{N-1}+\cdots+b_N=0$ avec $b_j\in \mathfrak{m}_{V_i}^{jn}$ (dépendant de $n$). Cela force
$v(y^N)\geq \min_{1\leq j\leq N} (v(b_j)+(N-j)v(y))$, puis $v(y)\geq n$ pour tout $n$, et enfin $y=0$ et $x=0$.

  On a donc $f(1)\bar{g}=0$ dans $\bar{A}=A/\wp$, puis $f(1)=0$ car $g\notin \wp$. En appliquant ceci au morphisme $c\mapsto f(tc)$ pour $t\in \bar{C}$ quelconque, on voit que $f=0$, autrement dit 
  ${\rm Hom}_{{\rm Mod}_{A/\wp}}(C/\wp C, A/\wp)=\{0\}$, contredisant le lemme ci-dessous.
\end{proof}

\begin{lemm}
Soit $A$ un anneau local noethérien complet et intègre et soit $\wp$ un idéal premier de $A$. Si 
$C$ est une $A$-algèbre ${\rm CM}$, alors ${\rm Hom}_{{\rm Mod}_A}(C, A)$ et ${\rm Hom}_{{\rm Mod}_{A/\wp}}(C/\wp C, A/\wp)$ sont non nuls.
\end{lemm}

\begin{proof} Le très joli argument suivant est dû à  \textcite{H5}. Soit $A_0\to A$ une extension finie, avec $A_0$ local régulier et complet, de corps résiduel $k$, et soit 
$q=A_0\cap \wp$. Si
$E$ (resp.\ $E'$) est une enveloppe injective du $A_0$-module (resp.\ $A_0/q$-module) $k$, 
par dualité de Matlis on a ${\rm Hom}_{{\rm Mod}_{A_0}}(C, A_0)\simeq {\rm Hom}_{{\rm Mod}_{A_0}} (C\otimes_{A_0} E, E)$ et ${\rm Hom}_{{\rm Mod}_{A_0/q}}(C/q C, A_0/q)\simeq {\rm Hom}_{{\rm Mod}_{A_0/q}}(C/q C\otimes_{A_0/q} E', E')$. Ces modules sont non nuls car $C\otimes_{A_0} E$ et $C/q C\otimes_{A_0/q} E$ le sont, puisque 
$C$ (resp.\ $C/q C$) est fidèlement plat sur $A_0$ (resp.\ $A_0/q$), en tant que $A$-algèbre CM (proposition~\ref{CMstandard}). 

  Pour conclure, il faut passer d'un $A_0$-dual à un $A$-dual (l'argument donné ci-dessous pour ${\rm Hom}_{{\rm Mod}_A}(C, A)$ est identique pour ${\rm Hom}_{{\rm Mod}_{A/\wp}}(C/\wp C, A/\wp)$). Soit $f\colon  C\to A_0$ un morphisme $A_0$-linéaire non nul. Pour tout $c\in C$ on dispose d'un morphisme $A_0$-linéaire 
  $\varphi_c\colon  A\to A_0, a\mapsto f(ac)$. L'application $F\colon  C\to X\coloneqq {\rm Hom}_{{\rm Mod}_{A_0}}(A, A_0), c\mapsto \varphi_c$ est non nulle et $A$-linéaire. Puisque le morphisme $A_0\to A$ est fini, le 
  $A$-module $X$ est de type fini sur $A$ et sans torsion\footnote{Si $a\in A\setminus \{0\}$ et $f\in X$ vérifient $a.f=0$ alors 
  $f(P(a)x)=0$ pour tout $P\in XA_0[X]$ et tout $x\in A$; il suffit de prendre $P$ tel que $P(a)\in A_0\setminus \{0\}$ pour conclure que $f(x)=0$.}, donc se plonge  
  dans un $A$-module libre de type fini $A^j$. En projetant sur l'un des facteurs de $A^j$ et en composant avec $F$ on obtient 
  un élément non nul de ${\rm Hom}_{{\rm Mod}_A}(C, A)$.
\end{proof}

\subsection{Construction d'une algèbre presque perfectoïde et presque CM}

 Soit $P\in {\rm Perf}_{K^0}^{\rm tf}$ et soit $g\in P\setminus \{0\}$ un élément $p$-puissant (d\'efinition~\ref{ppuis}) et multiple de $p$. On ne suppose \emph{pas} que 
       $P$ est sans $g$-torsion.
        Soit $\iota\colon  P\to P[\frac{1}{g}]$ le morphisme canonique et considérons la $P$-algèbre
     $$Q\coloneqq g^{-1/p^{\infty}} P=\bigcap_{n\geq 0}\Bigl\{x\in P\bigl[\frac{1}{g}\bigr]|\,\, g^{1/p^n}x\in \iota(P)\Bigr\}\subset P\bigl[\frac{1}{g}\bigr].$$
   L'image de $g$ dans $Q$ est $p$-puissante, multiple de $p$ et \emph{non diviseur de zéro}. Le morphisme $\iota\colon  P\to P[\frac{1}{g}]$ induit un presque isomorphisme $\iota\colon  P\to Q$ (dans le cadre $g^{1/p^{\infty}}$), car 
     $(g^{p^{-\infty}})$ annule $P[g^{\infty}]$ (lemme~\ref{reduced}). Si $P,g$ sont comme ci-dessus, on dira que 
     $Q$ est une \emph{algèbre presque perfectoïde}. 

   Le résultat peu ragoûtant suivant s'obtient en combinant ceux des paragraphes ci-dessus. 
   
  
  \begin{prop}\label{UGLY}
   Soit $(A, \mathfrak{m})\in {\rm CLI}_p$ (d\'ef.~\ref{symp}), $\wp\in {\rm Spec} A[\frac{1}{p}]$ et soient
   $g\in pA\setminus \wp$ et $(f_1,\ldots, f_c, x_1,\ldots, x_d)$ comme dans la proposition~\ref{choosesop}. Pour toute
   $A$-algèbre perfectoïde {\rm CM} $\mathfrak{m}$-complète $C$ il existe 
une $C$-algèbre $P\in {\rm Perf}_{K^0}^{\rm tf}$ dans laquelle 
$g$ devient $p$-puissant (déf.~\ref{ppuis}) et non nul, s'insérant dans
un diagramme commutatif    
     $$\xymatrix@R=.6cm@C=.8cm{A\ar[r]^-{}\ar[d]^{}&
A/\wp \ar[d]\\
P\ar[r]^-{}& Q\coloneqq g^{-1/p^{\infty}}P}$$
et telle que l'image de 
    $(x_1,\ldots, x_d)$  
   dans $P$ soit une suite régulière et presque régulière (dans le cadre $g^{1/p^{\infty}}$, déf.~\ref{prreg}) d'éléments $p$-puissants. 
  \end{prop}
  
  \begin{proof} 
   Par le lemme~\ref{CMcover} on peut trouver une $A$-algèbre perfectoïde CM, $\mathfrak{m}$-complète $C_1$ qui est une 
   $C$-algèbre dans laquelle $g,f_1,\ldots, f_c, x_1,\ldots, x_d$ deviennent $p$-puissants. 
       Soit 
   $\tilde{C}_1=C_1/(f_1^{p^{-\infty}},\ldots, f_c^{p^{-\infty}})C_1$. La proposition~\ref{presentquot} montre que 
   $P\coloneqq \varprojlim_{n}\tilde{C}_1/p^n\in {\rm Perf}_{K^0}^{\rm tf}$ et 
  que $(x_1,\ldots, x_d)$ est une suite régulière et 
   presque régulière (dans le cadre $g^{1/p^{\infty}}$) dans $P$. En particulier $g\ne 0$ dans $P$, sinon $g^{1/p^n}=0$ pour tout~$n$ (car $P$~est réduit, cf.\ proposition~\ref{reduced}) et donc $(g^{p^{-\infty}})$ annule $P/(x_1,\ldots, x_d)$, une contradiction. 
   
   
   Il reste à vérifier que le morphisme $A\to P$ se factorise $A/\wp\to g^{-1/p^{\infty}}P$. Si 
   $f\in \wp$ alors $gf\in \sqrt{(f_1,\ldots, f_c)}$, donc $gf$ devient nilpotent dans $\tilde{C}_1$ et donc aussi dans 
   $P$. Comme $P$ est réduit (proposition~\ref{reduced}), l'image de $f$ dans $P$ arrive dans $P[g]$ et 
   $A\to P$ se factorise $A/\wp\to P/P[g^{\infty}]\to g^{-1/p^{\infty}}P$.   
  \end{proof}

   \begin{prop}\label{alCMexist}
   Pour tout $A\in {\rm CLI}_p$ (d\'ef.~\ref{symp}) il existe un système de paramètres 
   $(x_1,\ldots, x_d)$ de $A$ avec $x_1=p$, un élément $g\in A$ et un objet $P\in {\rm Perf}_{K^0}^{\rm tf}$ tels que 
   $g$ soit $p$-puissant (déf.~\ref{ppuis}) et non nul dans $P$
   et que $Q\coloneqq g^{-1/p^{\infty}}P$ soit une $A$-algèbre dans laquelle 
  $(x_1,\ldots, x_d)$ est une 
   suite presque régulière (dans le cadre $g^{1/p^{\infty}}$, déf.~\ref{prreg}).
            \end{prop}
         
             \begin{proof}
        Soit $k$ le corps résiduel de $A$. Il existe $n$ et une surjection $A_0\coloneqq W(k)[[T_1,\ldots, T_n]]\to A$. Soit $\wp=\ker(A_0\to A)$, donc $A\simeq A_0/\wp$. Le théorème~\ref{perfcover} fournit une $A_0$-algèbre CM perfectoïde et $\mathfrak{m}_{A_0}$-complète $C$. On conclut en appliquant la proposition ci-dessus à ces données (avec $A_0$ à la place de $A$), et en remarquant que $P$ et $Q$ sont presque isomorphes dans le cadre 
        $g^{1/p^{\infty}}$.     
         \end{proof}

\subsection{La construction de Gabber}

         Pour finir la preuve du théorème~\ref{CMexistent} il faut passer d'une presque algèbre de Cohen--Macaulay à une vraie. Voir la proposition $4.1.2$ d'\textcite{AndreDSC} 
         pour la méthode des \emph{modifications partielles} de Hochster, nous allons présenter la construction de Gabber (un peu modifi\'ee), qui est miraculeusement \'el\'ementaire et directe. Voir aussi la section $17.5$ du livre de \textcite{GR} pour des compl\'ements.

  Soit $A$ un anneau et soit $g\in A$ un \'el\'ement $p$-puissant (d\'ef.~\ref{ppuis}) non nul.
  Soit $$S=\{(g^{a_n})_{n\geq 0}\in A^{\mathbf{N}}|\,\, a_n\in \mathbf{N}[\frac{1}{p}],\, \lim_{n\to\infty} a_n=0\}.$$
  Il est \'evident que $S$ est une partie multiplicative de $A^{\mathbf{N}}$,
  qui ne contient pas $0\coloneqq (0,0,\ldots)$. Si $M$ est un $A$-module, alors $M^{\mathbf{N}}$ est un $A^{\mathbf{N}}$-module, et on d\'efinit 
$$\mathcal{G}(M)\coloneqq S^{-1}M^{\mathbf{N}}.$$

\begin{lemm}\label{Galalisom}
 Si $f\colon M\to N$ est un presque isomorphisme de $A$-modules (dans le cadre $g^{1/p^{\infty}}$), alors 
 le morphisme induit $\mathcal{G}(f)\colon \mathcal{G}(M)\to \mathcal{G}(N)$ est un isomorphisme.
\end{lemm}

\begin{proof} \'Ecrivons simplement $(x_n)$ au lieu de $(x_n)_{n\geq 0}$. 
 Si $\mathcal{G}(f)(\frac{(m_n)}{s})=0$ alors $\frac{(f(m_n))}{s}=0$ dans 
 $S^{-1}N^{\mathbf{N}}$, donc il existe $s'=(g^{a'_n})\in S$ tel que $s'(f(m_n))=0$, autrement dit 
 $g^{a'_n} f(m_n)=0$ pour tout $n$. Puisque 
 $(g^{p^{-\infty}})$ annule 
 $\ker f$, on obtient $g^{a'_n+1/p^n}m_n=0$. Ainsi $s''\coloneqq (g^{a'_n+1/p^n})\in S$ et $s''(m_n)=0$ dans $M^{\mathbf{N}}$, donc $\frac{(m_n)}{s}=0$.
 
  Ensuite, soit $\frac{(k_n)}{s}\in \mathcal{G}(N)$. Puisque $(g^{p^{-\infty}})$ annule ${\rm coker}(f)$, pour tout $n$ il existe 
  $m_n\in M$ tel que $g^{1/p^n} k_n=f(m_n)$. Posons $s'=s\cdot (g^{1/p^n})\in S$, alors 
  $\frac{(k_n)}{s}=\mathcal{G}(f)(\frac{(m_n)}{s'})$, ce qui permet de conclure.
\end{proof}
  
    Si $C$ est une $A$-alg\`ebre, alors $\mathcal{G}(C)$ est aussi une $A$-alg\`ebre, via le morphisme diagonal 
    $A\to A^{\mathbf{N}}$.

\begin{lemm}\label{GabalCM}
 Soit $C$ une $A$-alg\`ebre et soit $x=(x_1,\ldots, x_d)$ une suite dans $A$. Si $x$ est presque r\'eguli\`ere (dans le cadre 
 $g^{1/p^{\infty}}$) dans $C$, alors $x$ devient une suite r\'eguli\`ere dans $\mathcal{G}(C)$.
\end{lemm}

\begin{proof}
 Montrons d'abord que $\mathcal{G}(C)/(x_1,\ldots, x_d)\mathcal{G}(C)\ne \{0\}$. 
 Sinon, il existe $\alpha_i\in S^{-1}C^{\mathbf{N}}$ tels que 
 $1=\sum_{i=1}^d x_i \alpha_i$, donc $1=\sum_{i=1}^d x_i \frac{y_i}{s}$ pour certains $s\in S$ et 
 $y_i\in C^{\mathbf{N}}$. Il existe $s'\in S$ tel que $s'(s-\sum_{i=1}^d x_i y_i)=0$ dans $C^{\mathbf{N}}$. Si 
 $s=(g^{a_n})_{n\geq 0}$ et $s'=(g^{a'_n})_{n\geq 0}$, on en d\'eduit (projeter sur la $n$-i\`eme composante) que $g^{a_n+a'_n}\in (x_1,\ldots, x_d)C$ pour tout~$n$. Puisque $a_n+a'_n\to 0$, il s'ensuit que $(g^{p^{-\infty}})$ annule $C/(x_1,\ldots, x_d)$, contredisant le fait que $x$ est presque r\'eguli\`ere dans $C$.
 
  Il nous reste \`a v\'erifier que si $a\in \mathcal{G}(C)$ v\'erifie $ax_{i+1}\in (x_1,\ldots, x_i)\mathcal{G}(C)$, alors 
  $a\in (x_1,\ldots, x_i)\mathcal{G}(C)$. Il existe $s\in S$ et $b, z_k\in C^{\mathbf{N}}$
  tels que $a=\frac{b}{s}$ et 
  $ax_{i+1}=\sum_{k=1}^i x_k \frac{z_k}{s}$, et il existe 
  $s'\in S$ tel que $s'(bx_{i+1}-\sum_{k=1}^i x_k z_k)=0$. Si  
  $s'=(g^{a'_n})$, alors (par projection sur la $n$-i\`eme composante) 
  $g^{a'_n}b_n x_{i+1}\in (x_1,\ldots, x_i)C$. Comme 
  $(g^{p^{-\infty}})$ annule $\frac{(x_1,\ldots, x_i)C:(x_{i+1})C}{(x_1,\ldots, x_i)C}$, 
  il s'ensuit que $g^{a'_n+1/p^n}b_n\in (x_1,\ldots, x_i)C$ pour tout $n$. On peut donc \'ecrire 
  $$g^{a_n'+1/p^n}b_n=\sum_{k=1}^i x_k u_{n,k}$$
  pour certains $u_{n,k}\in C$. Si l'on pose $u_k=(u_{n,k})_{n\geq 0}\in C^{\mathbf{N}}$ et 
  $s''=(g^{1/p^n})_{n\geq 0}$, cela s'\'ecrit
  $s's'' b=\sum_{k=1}^i x_k u_k$
  dans $C^{\mathbf{N}}$, donc $a=\frac{b}{s}=\sum_{k=1}^i x_k \frac{u_k}{ss's''}\in (x_1,\ldots, x_i)\mathcal{G}(C)$. 
 \end{proof}
 
 \begin{lemm}\label{Gabperf}
 Si $C$ est un anneau perfectoïde, alors le compl\'et\'e $p$-adique de 
 $\mathcal{G}(C)$ est un anneau perfectoïde. 
 \end{lemm}
 
 \begin{proof}
  $C^{\mathbf{N}}$ est un produit d'anneaux perfectoïdes, donc un anneau perfectoïde. Il suffit donc de montrer que le compl\'et\'e $p$-adique d'un localis\'e d'un anneau perfectoïde est encore perfectoïde, cf.\ cor. $2.1.6$ de \cite{CS}.
   \end{proof}
   
   \subsection{Construction d'alg\`ebres perfectoïdes de Cohen--Macaulay, fonctorialit\'e faible}
   
   Nous pouvons maintenant mettre ensemble les r\'esultats ci-dessus et obtenir une preuve de l'existence de $A$-alg\`ebres de Cohen--Macaulay pour tout 
   $A\in {\rm CLI}_p$. Ceci a \'et\'e d\'emontr\'e pour la premi\`ere fois dans l'article d'\textcite{AndreDSC}, et raffin\'e ensuite dans celui de \textcite{Shimo}
   sous la forme suivante: 
      
\begin{theo}\label{Shimo}
 Pour tout $(A, \mathfrak{m})\in {\rm CLI}_p$ (def.~\ref{symp})
il existe une $A$-algèbre perfectoïde ${\rm CM}$ (d\'ef.~\ref{perfCM}) et $\mathfrak{m}$-complète. 
\end{theo}

\begin{proof}
 Soient $g,x_1,\ldots, x_d, P$ et $Q$ comme dans la proposition~\ref{alCMexist}.
 Le presque isomorphisme 
 $P\to Q$ induit un isomorphisme $\mathcal{G}(P)\simeq \mathcal{G}(Q)$ (lemme~\ref{Galalisom}), donc $\mathcal{G}(Q)\simeq \mathcal{G}(P)$
 devient un anneau perfectoïde apr\`es compl\'etion $p$-adique (lemme~\ref{Gabperf}). La suite $(x_1,\ldots, x_d)$ devient r\'eguli\`ere dans $\mathcal{G}(Q)$ (lemme~\ref{GabalCM}), en particulier 
 $\mathcal{G}(Q)\in {\rm Perf}_{K^0}^{\rm tf}$. Le lemme~\ref{221} permet de conclure que le compl\'et\'e 
 $\mathfrak{m}$-adique de $\mathcal{G}(Q)$ est une $A$-alg\`ebre perfectoïde ${\rm CM}$ et $\mathfrak{m}$-complète. 
\end{proof}
                
               \begin{rema}
              Il n'est pas difficile d'en d\'eduire que tout anneau local noeth\'erien complet 
              $(A, \mathfrak{m})$ d'in\'egale caract\'eristique $(0,p)$ admet une $A$-algèbre perfectoïde ${\rm CM}$ (d\'ef.~\ref{perfCM}) $\mathfrak{m}$-complète. En effet, on voit facilement (utiliser le th\'eor\`eme de Cohen) qu'il existe une $A$-alg\`ebre locale noeth\'erienne compl\`ete $\tilde{A}$, dont le corps r\'esiduel est alg\'ebriquement clos, et fid\`element plate sur
              $A$. Soit $\wp$ un id\'eal premier minimal de $\tilde{A}$ tel que $\dim \tilde{A}/\wp=\dim \tilde{A}$.   Alors 
              $\tilde{A}/\wp\in {\rm CLI}_p$, donc il existe une $\tilde{A}/\wp$-alg\`ebre $C$ perfectoïde {\rm CM}, qui est
              automatiquement une $\tilde{A}$-alg\`ebre perfectoïde ${\rm CM}$. Puisque $A\to \tilde{A}$ est fid\`element plat, tout syst\`eme de param\`etres de 
              $A$ s'\'etend en un syst\`eme de param\`etres de $\tilde{A}$, et donc devient une suite r\'eguli\`ere dans $C$. Ainsi $C$ est une $A$-alg\`ebre perfectoïde {\rm CM}. Il suffit de compl\'eter $C$ pour la topologie $\mathfrak{m}$-adique et d'appliquer le lemme~\ref{221} pour conclure.
                   \end{rema}
                   
                   La m\'ethode employ\'ee ci-dessus est assez souple pour retrouver l'un des r\'esultats fondamentaux 
                   de l'article d'\textcite{AndreJAMS}: 
                   
                   \begin{theo}\label{weak} Soit $f\colon  A\to A'$ un morphisme surjectif dans ${\rm CLI}_p$ (d\'ef.~\ref{symp}).
Pour toute $A$-algèbre perfectoïde {\rm CM} et $\mathfrak{m}_{A}$-complète $C$
 il existe une $A'$-algèbre perfectoïde {\rm CM} et $\mathfrak{m}_{A'}$-complète $C'$ s'insérant dans un diagramme commutatif 
  $$\xymatrix@R=.6cm@C=.8cm{A\ar[r]^-{}\ar[d]^{}&
A' \ar[d]\\
C\ar[r]^-{}& C'}$$
\end{theo}

\begin{proof} On peut supposer que $f$ est la projection canonique $A\to A/\wp$ pour un $\wp\in {\rm Spec} A[\frac{1}{p}]$.
Il suffit de remplacer l'usage de la proposition~\ref{alCMexist} dans la preuve ci-dessus par celui de la proposition~\ref{UGLY}.
En effet, les arguments utilis\'es dans la preuve du th\'eor\`eme~\ref{Shimo} montrent que le compl\'et\'e 
$\mathfrak{m}$-adique $C'$ de $\mathcal{G}(Q)$ (avec $Q$ comme dans la proposition~\ref{UGLY}) r\'epond \`a l'appel: c'est une 
$A'$-algèbre perfectoïde CM et $\mathfrak{m}_{A'}$-complète, et on dispose d'un morphisme $P\to C'$, donc aussi d'un morphisme $C\to C'$, qui s'ins\`ere dans un diagramme comme dans le th\'eor\`eme par construction. 
\end{proof}

\begin{rema}
\begin{enumerate}

\item On trouve dans le th\'eor\`eme $4.1.1$ d'\textcite{AndreJAMS} et dans le livre de  \textcite{GR} des
         formes
plus raffinées et générales, mais le théorème ci-dessus contient déjà bon nombre de difficultés essentielles.

\item  En utilisant des factorisations de Cohen, on peut en déduire
que tout morphisme $f\colon  A\to A'$ dans ${\rm CLI}_p$ s'insère dans un diagramme comme ci-dessus, pour une 
$A$-algèbre CM $C$ et une $A'$-algèbre CM $C'$. Ceci a des multiples
applications, voir par exemple l'article de \textcite{HH2}. Mentionnons simplement une conséquence frappante: si $A$ est un anneau local régulier, extension scindée d'un sous-anneau 
$A'$, alors $A'$ est un anneau de Cohen--Macaulay\footnote{Ceci avait été obtenu en inégale caractéristique par  \textcite{HMa}, avant l'article d'\textcite{AndreJAMS}.}. 


\end{enumerate}
\end{rema}

         

         \section{ Alg\`ebres de Cohen--Macaulay via le lemme d'Abhyankar perfectoïde}
         
         Dans cette derni\`ere section nous pr\'esentons l'approche initiale d'\textcite{AndreDSC} pour construire des alg\`ebres de Cohen--Macaulay, en utilisant le lemme d'Abhyankar perfectoïde. Il est possible de pousser cette m\'ethode pour obtenir une preuve du th\'eor\`eme~\ref{weak} (c'est ce qui est fait dans l'article d'\textcite{AndreJAMS}), mais on ne le fera pas ici.

     Pour toute $K^0$-algèbre $R$ notons 
    $R^{\natural}=W(R^{\flat})/(p-[p^{\flat}])$.
   Si $P\in {\rm Perf}_{K^0}^{\rm tf}$ alors $P^{\natural}\simeq P$ via $\theta_P$, donc 
   si $R$ est une $P$-algèbre, alors $R^{\natural}$ est une $P\simeq P^{\natural}$-algèbre. 
    Si 
   $R$ est $p$-complète,  
  tout morphisme 
   $P\to R$ se factorise canoniquement $P\to R^{\natural}\to R$.

         \subsection{La construction fondamentale}
                  
         \emph{Motivés par la proposition~\ref{UGLY}, considérons le contexte suivant. On se donne}
                  
         $\bullet$ un morphisme $A'\to A$ dans ${\rm CLI}_p$ (d\'ef.~\ref{symp}) et un élément $g\in pA'$.
     
     $\bullet$ une algèbre $P\in {\rm Perf}_{K^0}^{\rm tf}$ dans laquelle $g$ devient $p$-puissant et non nul et s'insérant
     dans un diagramme commutatif 
    $$\xymatrix@R=.6cm@C=.8cm{A'\ar[r]^-{}\ar[d]^{}&
A \ar[d]\\
P\ar[r]^-{}& Q\coloneqq g^{-1/p^{\infty}}P }$$
Notons que l'image de $g$ dans $Q$ est non nulle, car $P$ est r\'eduit.


  Soit $A\to A_1$ une extension finie, étale après inversion de $g$, et posons\footnote{Rappelons que 
    ${\rm fi}(A,B)$ désigne la clôture intégrale de $A$ and $B$.}
 $$\mathcal{F}(A_1)={\rm fi}(Q, A_1\otimes_A Q[\frac{1}{g}]).$$
Puisque
    $A\to A_1$ est entier, le morphisme naturel $A_1\to A_1\otimes_A Q[\frac{1}{g}]$ se factorise à travers $\mathcal{F}(A_1)$, induisant ainsi un diagramme commutatif
      $$\xymatrix@R=.6cm@C=.8cm{A\ar[r]^-{}\ar[d]^{}&
A_1 \ar[d]\\
Q\ar[r]^-{}& \mathcal{F}(A_1)}$$ 
Consid\'erons la $P$-alg\`ebre
   $$ \mathcal{P}(A_1)\coloneqq \mathcal{F}(A_1)^{\natural}.$$
      Il n'est pas clair que $\mathcal{F}(A_1)$ soit $p$-complète, mais la preuve du résultat ci-dessus montrera que c'est bien le cas, donc 
      le morphisme $P\to Q\to \mathcal{F}(A_1)$ se factorise $P\to \mathcal{P}(A_1)\to \mathcal{F}(A_1)$.

     \begin{prop}\label{HARD}
     Soit $A\to A_1$ une extension finie, étale après inversion de 
     $g$. La $P$-algèbre $\mathcal{P}(A_1)$ est presque fidèlement plate (dans le cadre $g^{1/p^{\infty}}$) modulo $p$ sur $P$, sans $g$-torsion et $g^{-1/p^{\infty}}\mathcal{P}(A_1)=\mathcal{F}(A_1)$.       \end{prop}
                     
        Ce résultat est une application du lemme d'Abhyankar perfectoïde, mais la vérification des hypothèses demande quelques 
        préliminaires.

\begin{lemm} \label{alperf}
Si $S\in \{Q, Q^{\natural}\}$ alors l'algèbre
$S$ est $p$-complète, sans $g$-torsion, 
et $p$-close dans $S[\frac{1}{p}]$. 
\end{lemm}
            
        \begin{proof} Posons $R=Q^{\natural}$. Comme $Q$ est sans $g$-torsion et $p\mid g$, l'alg\`ebre
        $Q$ est sans $p$-torsion. Il en est de même de 
        $R$ puisque $p^{\flat}$ n'est pas un diviseur de zéro dans $Q^{\flat}$: si $x=(x_n)_{n\geq 0}\in Q^{\flat}$ vérifie 
  $p^{\flat}x=0$, et si $a_n\in Q$ est un relèvement de $x_n\in Q/p$, alors $a_n\in p^{1-1/p^n}Q$ et $a_n\equiv a_{n+1}^p\equiv 0\pmod {pQ}$ pour tout $n$, 
 donc $x=0$.
  
   L'anneau $R$ est $p$-complet car 
  perfectoïde (proposition~\ref{pcomplet}). Montrons que $Q$ est $p$-complet. 
Soit $g_n=g^{\frac{1}{p^n}-\frac{1}{p^{n+1}}}$. Le morphisme $\iota\colon  P\to P[\frac{1}{g}]$ induit un isomorphisme\footnote{L'injectivité est claire,
et si $(x_n)_{n\geq 0}\in \varprojlim_{\cdot g_n} \iota(P)$ alors la suite $(g^{-1/p^n}x_n)_{n\geq 0}$ est constante dans 
     $P[\frac{1}{g}]$, sa valeur $x$ est dans $Q$ par définition et $\alpha(x)=(x_n)_{n\geq 0}$.}
$$\alpha\colon  Q\to \varprojlim_{\cdot g_n} \iota(P),\,\, x\mapsto (g^{1/p^n}x)_{n\geq 0}.$$  
 Puisque $I\coloneqq (g^{p^{-\infty}})$ 
annule le noyau de $\iota\colon P\to \iota(P)$  
et contient les $g_n$, le morphisme $\iota$ induit un isomorphisme de pro-systèmes $\{P, (\cdot g_n)\}\to \{\iota(P), (\cdot g_n)\}$,
d'où 
$$Q\simeq \varprojlim_{\cdot g_n} \iota(P)\simeq \varprojlim_{\cdot g_n} P.$$
Comme $P$ est $p$-complet et sans $p$-torsion, on en déduit facilement que $Q$ est $p$-complet. 
        
        Montrons que $Q$ (resp.\ $R$) est $p$-clos dans $Q[\frac{1}{p}]$ (resp.\ $R[\frac{1}{p}]$).  
        La proposition~\ref{banperf} montre que $P$ (resp.\ $R$) est $p$-clos dans
  $P[\frac{1}{p}]$ (resp.\ $R[\frac{1}{p}]$). 
  Pour conclure il suffit de montrer que si 
  $x\in Q$ vérifie $x^p\in pQ$, alors $\frac{x}{p^{1/p}}\in P[\frac{1}{g}]$ est dans $Q$,
  ou encore que $\alpha \beta \frac{x}{p^{1/p}}\in \iota(P)$
 pour tous 
  $\alpha,\beta\in I$. Comme $x\in Q$ et $x^p\in pQ$, 
  il existe $a,b\in P$ tels que $\alpha x=\iota(a)$ et $(\alpha x)^p=p\iota(b)$.
  Alors $\beta^p(a^p-pb)=0$ (car $a^p-pb\in \ker(\iota)$ et $I$ annule $\ker(\iota)$), donc $(\beta a)^p\in pP$. Comme $P$ est $p$-clos dans $P[\frac{1}{p}]$, cela force
    $\beta a\in p^{1/p}P$ et  
 $\alpha\beta \frac{x}{p^{1/p}}=\iota(\frac{\beta a}{p^{1/p}})\in \iota(P)$, ce qui permet de conclure.
          \end{proof}
        
        \begin{lemm}\label{technical2}
                Le morphisme
$\theta_Q\colon  Q^{\natural}\to Q$ est injectif et $(g^{p^{-\infty}})$ annule son conoyau. On a donc 
$Q=g^{-1/p^{\infty}}Q^{\natural}\subset Q^{\natural}[\frac{1}{g}]$ et $Q^{\natural}$ est sans $g$-torsion.
\end{lemm}

\begin{proof} Notons encore $R=Q^{\natural}$.
  Montrons que 
  $\theta_Q\colon  R\to Q$ est injective, en particulier $g$ n'est pas un diviseur de zéro dans $R$. 
  Comme $R$ et $Q$ sont sans $p$-torsion et $p$-complets par la proposition ci-dessus, il suffit de montrer l'injectivité 
  de $R/p\simeq Q^{\flat}/p^{\flat}\to Q/p$, qui se déduit de celle du Frobenius 
  $Q/p^{1/p}\to Q/p$, cf.\ lemme.~\ref{alperf}.
  Puisque $\iota\colon  P\to Q$ est un presque isomorphisme dans le cadre $g^{1/p^{\infty}}$ et se factorise 
  $P\to R\to Q$, pour conclure il suffit de montrer que $P\to R$ est un presque isomorphisme, ou encore qu'il en est de même de 
 $\iota^{\flat}\colon  P^{\flat}\to Q^{\flat}$, ce qui est clair.       
\end{proof}

\begin{lemm}\label{appliab}
La $K$-algèbre $B\coloneqq Q^{\natural}[\frac{1}{p}]$ possède une structure de 
     $K$-algèbre de Banach perfectoïde sans $g$-torsion, telle que $g^{-1/p^{\infty}}B^0=Q$, en particulier 
     $B[\frac{1}{g}]=Q[\frac{1}{g}]\simeq P[\frac{1}{g}]$. 
\end{lemm}

\begin{proof}
Par la proposition~\ref{banperf} l'algèbre $Q^{\natural}_*\coloneqq 
  p^{-1/p^{\infty}}Q^{\natural}$ est perfectoïde, donc $B\coloneqq Q^{\natural}[\frac{1}{p}]$ possède une structure de 
     $K$-algèbre de Banach perfectoïde sans $g$-torsion, telle que $B^0=Q^{\natural}_*$.
     Comme 
    $p\mid g$ et $Q=g^{-1/p^{\infty}}Q^{\natural}$ (lemme~\ref{technical2}), on a 
    $$g^{-1/p^{\infty}}B^0=g^{-1/p^{\infty}}Q^{\natural}_*=g^{-1/p^{\infty}}Q^{\natural}=Q,$$
ce qui permet de conclure.
\end{proof}

 Revenons maintenant à la preuve de la proposition~\ref{HARD}.
Par le lemme~\ref{appliab} on peut munir 
  $B\coloneqq Q^{\natural}[\frac{1}{p}]$ d'une structure de 
     $K$-algèbre de Banach perfectoïde sans $g$-torsion, telle que $g^{-1/p^{\infty}}B^0=Q$, 
    en particulier $B[\frac{1}{g}]=Q[\frac{1}{g}]\simeq P[\frac{1}{g}]$.
           Puisque $A\to A_1$ est étale après inversion de $g$, 
     l'algèbre 
 $C=A_1\otimes_A B[\frac{1}{g}]=A_1\otimes_A Q[\frac{1}{g}]$ est finie, étale et fidèlement plate sur $B[\frac{1}{g}]$. Le lemme d'Abhyankar perfectoïde\footnote{Noter que, compte tenu de la discussion ci-dessus on a $\mathcal{F}(A')=\tilde{C}^0$ dans les notations du th\'eor\`eme~\ref{downtownabbey}.} 
 montre que $\mathcal{F}(A')$ 
 est presque fidèlement plat sur $B^0$ modulo $p$, donc aussi sur $P$ modulo $p$ (car $B^0$ et $P$ sont presque isomorphes dans le cadre 
 $g^{1/p^{\infty}}$), et le morphisme 
 $\mathcal{P}(A_1)= \mathcal{F}(A_1)^{\natural}\to \mathcal{F}(A_1)$ est injectif, de conoyau tué par $(g^{p^{-\infty}})$. 
               Cela permet de conclure.

                                
            \subsection{Nouvelle preuve de la proposition~\ref{alCMexist}}
            
          La proposition~\ref{HARD} et le lemme de platitude d'André fournissent une nouvelle preuve\footnote{Il s'agit en fait de la première preuve de ce résultat, due à  \textcite{AndreDSC}.} de la proposition~\ref{alCMexist}. Comme expliqu\'e dans la section pr\'ec\'edente, ceci implique le th\'eor\`eme~\ref{Shimo} via la construction de Gabber.
          
          Soit
               $A\in {\rm CLI}_p$ et soit $(x_1,\ldots, x_d)$ un système de paramètres, avec $x_1=p$. Le morphisme de $W(k)$-algèbres $A_0\coloneqq W(k)[[T_2,\ldots, T_d]]\to A, T_i\mapsto x_i$
            est fini et injectif. Soit $g\in pA_0\setminus \{0\}$ tel que 
            $A_0[\frac{1}{g}]\to A[\frac{1}{g}]$ soit étale, et soit $P\in {\rm Perf}_{K^0}^{\rm tf}$ une $A_0$-algèbre fidèlement plate modulo $p$ sur $A_0$, dans laquelle $g$
               devient $p$-puissant (cf.\ th\'eor\`eme~\ref{perfcover} et~\ref{nopresque}; le th\'eor\`eme~\ref{fflemma} ferait aussi l'affaire). 
                
                \begin{lemm}
                La suite $(x_1,\ldots, x_d)$ est régulière et presque régulière dans $P$ (dans le cadre $g^{1/p^{\infty}}$) et l'image de $g$ dans $P$ n'est pas nulle.
                \end{lemm}
    
    \begin{proof} Si l'image de $g$ dans $P$ était nulle, elle le serait modulo $p^n$ pour tout $n$, or 
    $A_0/p^n\to P/p^n$ est fidèlement plat, donc $g\in \cap_{n\geq 1} p^n A_0=\{0\}$, une contradiction avec $g\in A_0\setminus \{0\}$. Ensuite,
    comme $P$ est fidèlement plat modulo $p$ sur $A_0$ et $x_1=p$, la suite est régulière dans 
    $P$. Il reste à expliquer pourquoi l'inclusion $(g^{p^{-\infty}})P\subset (x_1,\ldots, x_d)P$ est impossible. Si elle avait lieu, on aurait $g\in (x_1,\ldots, x_d)^k P$ pour tout $k$, et comme 
    $P/p^n$ est fidèlement plat sur $A_0/p^n$ pour tout $n$, cela forcerait $g\in \bigcap_{k} (x_1,\ldots, x_d)^k$, puis $g=0$ par le théorème d'intersection de Krull.
    \end{proof}

                En appliquant la proposition~\ref{HARD} au morphisme identité $A_0\to A_0$, avec  $A_1=A$ on obtient une 
                $P$-algèbre perfectoïde $\mathcal{P}(A)$ presque fidèlement plate modulo $p$ sur $P$, dans laquelle $g$ est non diviseur de zéro et qui s'insère dans un diagramme commutatif
                $$\xymatrix@R=.6cm@C=.8cm{A_0\ar[r]^-{}\ar[d]^{}&
A \ar[d]\\
P\ar[r]^-{}& g^{-1/p^{\infty}}\mathcal{P}(A)}$$ 
                 L'algèbre $\mathcal{P}(A)$ répond à l'appel lancé par la proposition~\ref{alCMexist} grâce au lemme ci-dessous:
    
     \begin{lemm}\label{alright} Soit $R$ un anneau sans $p$-torsion, $g\in R$ un élément 
     $p$-puissant et 
     $x=(x_1,\ldots, x_d)$ une suite dans~$R$, avec $x_1=p$. Si $x$~est presque régulière dans~$R$ (dans le cadre $g^{1/p^{\infty}}$), elle le reste 
   dans toute $R$-algèbre~$S$ qui est 
     presque fidèlement plate sur~$R$ modulo~$p$.
           
     \end{lemm}
     
     \begin{proof} Soit $I=(g^{p^{-\infty}})\subset R$. 
      Puisque $S/(x_1,\ldots, x_d)$ est presque fidèlement plat sur $R/(x_1,\ldots, x_d)$ (car $x_1=p$) le morphisme 
      $R/(x_1,\ldots, x_d)\to S/(x_1,\ldots, x_d)$ est presque injectif, donc $S/(x_1,\ldots, x_d)$ n'est pas presque nul. Pour montrer que $I$ annule $\frac{(x_1,\ldots, x_i)S:x_{i+1}S}{(x_1,\ldots, x_i)S}$ on peut supposer que $i>0$ car $S$ est sans $p$-torsion. Notons 
      $\bar{R}\coloneqq R/p$, $\bar{S}\coloneqq S/p$. 
     Il suffit de montrer que $I$ annule 
      $\frac{(x_2,\ldots, x_i)\bar{S}: x_{i+1}\bar{S}}{(x_2,\ldots, x_i)\bar{S}}$.
      Par presque platitude de $\bar{S}$ sur $\bar{R}$ on a un presque isomorphisme\footnote{Il suffit de tensoriser avec $\bar{S}$
       la suite exacte
     $$0\to (x_2,\ldots, x_i): x_{i+1}\bar{R}\to \bar{R}\to \bar{R}/(x_2,\ldots, x_i).$$
}
     $$(x_2,\ldots, x_i)\bar{S}: x_{i+1}\bar{S}\simeq ((x_2,\ldots, x_i)\bar{R}: x_{i+1}\bar{R}) \bar{S},$$ ce qui permet de conclure.     
     \end{proof}

\printbibliography

\end{document}
